\documentclass[11pt, reqno]{amsart}
\usepackage{amsmath,amssymb}
\usepackage{iftex}
\usepackage{tabularray}
\usepackage{subcaption}
\usepackage{amsthm,amsmath,amssymb,amsfonts}
\usepackage[numbers]{natbib}
\usepackage{txfonts}
\usepackage{booktabs}
\usepackage{float}
\usepackage{graphicx}
\usepackage{graphics}
\usepackage{ascmac}
\usepackage{bm}
\usepackage{accents}
\usepackage[colorlinks=true, allcolors=blue]{hyperref}
\usepackage{url}
\usepackage{algorithm}
\usepackage{algorithmic}
\usepackage{enumitem}

\newcommand{\R}{\mathbb{R}}

\newcommand{\bE}{\mathbb{E}}
\newcommand{\bP}{\mathbb{P}}

\renewcommand{\tilde}{\widetilde}
\renewcommand{\hat}{\widehat}
\renewcommand{\le}{\leqslant}
\renewcommand{\ge}{\geqslant}

\newcommand{\mC}{\mathcal{C}}

\newcommand{\mP}{\mathcal{P}}

\newcommand{\mY}{\mathcal{Y}}

\newcommand{\1}{\mbox{1}\hspace{-0.25em}\mbox{l}}

\usepackage{color}

\usepackage{comment}

\newcommand{\obs}{\mathrm{obs}}
\newcommand{\rel}{\mathrm{rel}}
\newcommand{\rank}{\mathrm{rank}}

\DeclareMathOperator{\range}{range}

\DeclareMathOperator*{\argmin}{argmin}

\newcommand{\Normal}{\mathrm{N}}

\theoremstyle{definition}\newtheorem{problem}{Problem}[section]
\theoremstyle{definition}
\theoremstyle{remark}\newtheorem{assumption}{Assumption}

\theoremstyle{remark}\newtheorem{remark}[problem]{Remark}
\theoremstyle{definition}
\theoremstyle{plain}\newtheorem{theorem}[problem]{Theorem}
\theoremstyle{plain}
\theoremstyle{plain}\newtheorem{lemma}[problem]{Lemma}
\theoremstyle{plain}\newtheorem{proposition}[problem]{Proposition}
\theoremstyle{plain}
\theoremstyle{plain}


	\title{Finite-Sample Inference for Sparsely Permuted Linear Regression}
    \author[H. Ota]{Hirofumi Ota}
    \address[H. Ota]{Komaba Institute for Science, Graduate School of Arts and Sciences, The University of Tokyo.}
    \email{hirofumi-ota@g.ecc.u-tokyo.ac.jp}
    \author[M. Imaizumi]{Masaaki Imaizumi}
    \address[M. Imaizumi]{Graduate School of Arts and Sciences, The University of Tokyo. \\RIKEN Center for Advanced Intelligence Project.}
        \email{imaizumi@g.ecc.u-tokyo.ac.jp}
	\date{\today}

\begin{document}
\maketitle
\begin{abstract}
We study a linear observation model with an unknown permutation called \textit{permuted/shuffled linear regression}, where responses and covariates are mismatched and the permutation forms a discrete, factorial-size parameter.
The permutation is a key component of the data-generating process, yet its statistical investigation remains challenging due to its discrete nature.
We develop a general statistical inference framework on the permutation and regression coefficients. First, we introduce a localization step that reduces the permutation space to a small candidate set building on recent advances in the repro samples method, whose miscoverage decays polynomially with the number of Monte Carlo samples. Then, based on this localized set, we provide statistical inference procedures: a conditional Monte Carlo test of permutation structures with valid finite-sample Type-I error control. We also develop coefficient inference that remains valid under alignment uncertainty of permutations. {For computational purposes, we develop a linear assignment problem computable in polynomial time and demonstrate that, under a separation condition, its solution recovers the true permutation with high probability. Extensions to partially permuted designs, ridge regularization, and non-Gaussian errors are further discussed. Simulations, including comparisons with existing methods, and an application to air-quality data illustrate the finite-sample performance.}
\end{abstract}

\section{Introduction}

\subsection{Background and motivation}

Modern data integration and privacy mechanisms often disrupt the one-to-one correspondence between covariates and responses.
In this context, the \emph{shuffled/permuted linear regression model} \citep{abid2017linear, pananjady2017linear} serves as a fundamental regression framework that explicitly accounts for possible mismatches through an unknown permutation acting on the design or response indices.
This model provides a canonical baseline for studying statistical inference under alignment uncertainty and has become a central object in the recent literature on regression with unobserved correspondences.

Such structures arise in settings involving anonymized or aggregated data \citep{narayanan2008robust}, probabilistic record linkage \citep{fellegi1969theory, christen2012evaluation, winkler2014matching}, and privacy-preserving or decentralized data collection protocols \citep{vatsalan2017privacy, dwork2014algorithmic}.
In these situations, the pairings between predictor and response units may be only partially known or implicitly inferred, thereby motivating the study of regression models under alignment uncertainty.
We focus on the regime in which only a small fraction of the sample pairs are mismatched, often referred to as \emph{sparsely permuted} data \citep{slawski2019linear}.
Even a small number of mismatches can distort parameter estimates and invalidate uncertainty quantification if ignored.

The appearance of sparsely permuted data motivates two core inferential questions: \textit{whether any permutation is present in the data, and how to construct confidence sets for the regression coefficients that remain valid under alignment uncertainty}.
Addressing these problems presents methodological challenges, because the permutation is discrete and combinatorial, and the search space grows factorially with the sample size, which undermines local asymptotics, smooth likelihood expansions, and tractable selective conditioning. We therefore adopt a nonasymptotic, finite-sample approach based on the \emph{repro samples method} \citep{xie2022repro, xie2024repro, wang2022finite} and the principle of \emph{locally simultaneous inference} \citep{zrnic2024locally}. At a high level, we localize the permutation space to a data-dependent candidate set and perform inference uniformly over that set; details are presented in Section \ref{sec:methodology}-\ref{sec:inference}.

\subsection{Problem statement}
We consider the linear regression problem with permuted data, where the observed response vector $\bm{Y}^{\obs} \in \R^n$ is paired with an $n \times p$ deterministic design matrix $\bm{X}$ with full column rank $\rank(\bm{X})=p$. The model can be described as
\begin{equation}\label{eq:model intro}
    \bm{Y}^{\obs}  = \Pi_0 \bm{X} \bm{\beta}_{0} + \sigma_0 \bm{u}^{\rel},
\end{equation}
where $\bm{\beta}_{0} \in \R^p$ is an unknown coefficient vector, $\Pi_0$ is an unknown $n \times n$ permutation matrix, and $\bm{u}^{\rel}$ is a realization from $\bm{U} \sim \Normal(0, I_n)$, which is the noise vector associated with the observation $\bm{Y}^{\obs}$. The scale parameter $\sigma_0 > 0$ is also unknown. Under the frequentist framework, the statistical model in \eqref{eq:model intro} can also be described as the following random-sample version of the data-generating process
\begin{equation}\label{eq:random}
    \bm{Y}  = \Pi_0 \bm{X} \bm{\beta}_{0} + \sigma_0 \bm{U}, \qquad \bm{U} \sim \Normal(0, I_n).
\end{equation}

We further assume that the number of mismatches is at most $k$. Such situation is referred to as sparsely permuted data first appeared in \cite{slawski2019linear}, indicating that the permutation affects only a small fraction of all response-predictor pairs.
To quantify the number of mismatches,
we define the Hamming distance between an identity matrix $I_n$ and a $n \times n$ permutation matrix as
\begin{equation}
d(\Pi, I_n) = \sum_{i = 1}^{n} \1\{\Pi_{ii} = 0\},
\end{equation}
where $\Pi_{ii}$ denotes the $i$-th diagonal element of $\Pi$, and focus on the $k$-sparse permutation matrix that satisfies $d(\Pi, I_n) \le k$. Furthermore, we denote the set of $k$-sparse permutation matrices as
\begin{equation}
    \mP_{n,k}=\{\Pi\in\mP_n:d(\Pi,I_n)\le k\}.
\end{equation}

Throughout, $k$ denotes this prescribed upper bound of the Hamming distance, and $k_\star:=d(\Pi_0,I_n)\le k$ denotes the actual number of mismatches, which is unknown.

Based on the canonical setting, this paper addresses two goals of the statistical inference:
\begin{itemize}
\item[(G1)] \textbf{Testing for mismatches.}
Construct a finite-sample valid test of the composite hypothesis $H_0: d(\Pi_0,I_n)\le k_0$ (equivalently, $\Pi_0\in\mP_{n,k_0}$) versus $H_1: d(\Pi_0,I_n)>k_0$ (equivalently, $\Pi_0\in\mP_{n,k}\setminus\mP_{n,k_0}$). The special case $k_0=0$ tests whether any mismatch is present.
\item[(G2)] \textbf{Confidence sets for $\bm{\beta}_0$ under unknown $\Pi_0$.}
Build a confidence set for $\bm{\beta}_0$ that accounts for the uncertainty in estimating $\Pi_0$ and attains nominal coverage in finite samples.
\end{itemize}

\subsection{Overview of contributions}\label{subsec:contributions}
We develop a methodology for statistical inference toward the above goals, which consists of two stages: generating candidate sets and conducting inference based on them.
In the first step, we generate a candidate set of plausible permutations based on the repro samples method \citep{xie2022repro,xie2024repro}. We generate artificial noise $L$ times, and for each noise draw, we optimize an objective function that mimics the least-squares objective under the corresponding noise realization, and then collect the resulting optimizers into a data-dependent candidate set $\mathcal{C}^{(L)}\subset\mathcal{P}_{n,k}$. This localized domain is then used to conduct coefficient inference by taking union of projection-based $F$-statistics over $\mathcal{C}^{(L)}$, which yields locally simultaneous confidence regions for $\bm{\beta}_0$.
In the second step, we develop a test for the hypothesis about the permutation structure: $H_0:\Pi_0\in\mP_{n,k_0}\quad\text{vs.}\quad H_1:\Pi_0\in\mP_{n,k}\setminus\mP_{n,k_0}$. We construct a discrete statistic based on $\mC^{(L)}$ and calibrate it via conditional Monte Carlo uniformly over the localized null set $\hat{\mP}_{n,k_0}:=\mC^{(L)}\cap\mP_{n,k_0}$. These steps achieve the two statistical inference goals (G1) and (G2).

We summarize the contributions of our study.
\begin{itemize}
    \item The approach comes with finite-sample guarantees. For the generation of the candidate set in the first step with the localization, we prove an explicit bound on the miscoverage probability $\mathbb{P}(\Pi_0\notin\mathcal{C}^{(L)})$, which decays at the rate $1/L$ in the number of repro samples $L$ and displays its dependence on $n$, $p$ and $k$ (Theorem~\ref{thm:recovery}). Also, we show the coverage of the confidence regions for $\bm{\beta}_0$ at the nominal level up to the localization error that vanishes at a polynomial rate in $L$ (Theorem~\ref{thm:confidence}). For testing, the developed statistical test controls type-I error at any sample size (Theorem~\ref{thm:p-value}) localized on the event $\{\Pi_0\in\mC^{(L)}\}$. Marginally, the size inflation is bounded by the miscoverage probability of $\mC^{(L)}$.
    \item We develop a computationally effective scheme to implement the method. In particular, we develop a sparsity-penalized linear assignment problem (LAP) as a surrogate to the original least square problem for generating the candidate set. This surrogate problem is solved by the Hungarian algorithm with the computational complexity $O(n^3)$, which avoids the combinatorial explosion of the computational complexity of the original problem. We also prove the high-probability equivalence of the surrogate problem to the original problem (Theorem~\ref{thm:high-probability equivalence}).
    \item Our methodology is examined through both simulation, including comparisons with an oracle benchmark and with existing methods, and real-data analysis in an application to Beijing air-quality data. We also apply our approach  to the partially permuted designs and the permuted ridge regression model.
    Further, we propose a selection method for the tuning parameters that does not require the number of mismatches and has a finite-sample guarantee when no mismatch is present (Proposition~\ref{prop:tuning-null}).

\end{itemize}

\subsection{Related works}\label{subsec:related-works}

\subsubsection{Shuffled regression}
Several recent methods use Hungarian/LAP solvers as computational surrogates for shuffled regression, e.g., spectral matching via covariance eigenbasis alignment \citep{liu2025shuffled} or LAP subproblems within alternating schemes \citep{li2021generalized}.
Relatedly, local search over $k$-sparse permutations performs greedy swaps with monotone loss descent but offers no guarantee of reaching a global optimum in noisy settings \citep{mazumder2023linear}.
In contrast, Theorem~\ref{thm:high-probability equivalence} provides a finite-sample guarantee: under an explicit separation condition, our sparsity-penalized LAP recovers the true permutation with high probability, and we bound the probability that it also minimizes the repro-sample objective.

Research related to shuffled regression includes the estimation of shape-constrained models when the correspondence between covariates and responses is unknown.
For one-dimensional covariate, \cite{rigollet2019uncoupled} obtain minimax rates of an estimation error for uncoupled isotonic regression via minimum Wasserstein deconvolution.
In multivariate settings, there are more studies.
For coordinate-wise isotonic regression with axis-wise permutations, \cite{pananjady2022isotonic} derive minimax and adaptation results and propose a polynomial-time partition estimator. \cite{slawski2024permuted} establish identification via cyclical monotonicity and develop optimal-transport/likelihood-based estimators with certain error bounds. Permutation uncertainty also appears in structured arrays: \cite{ma2021optimal} give information-theoretic limits and near-optimal polynomial-time recovery for monotone matrices, and \cite{lee2025statistical} provide statistically optimal and computationally efficient estimators for smooth tensors. 

A complementary line of work develops statistical inference for shuffled linear regression. \cite{slawski2021pseudo} propose a pseudo-likelihood-based asymptotic inference for the coefficient vector and provide EM-type algorithm for the optimization. \cite{chakraborty2024learning} study sparsely permuted data from a robust Bayesian perspective, treating mispaired entries as outliers within a generalized fractional posterior, and establish posterior consistency alongside practical sampling strategies.

\subsubsection{Repro samples method}

Our work adopts the repro samples method \citep{xie2022repro,xie2024repro}, which provides non-asymptotic, simulation-based inference by quantifying uncertainty through artificial samples.
In high-dimensional linear models, repro samples have been used to construct model confidence sets and valid inference for coefficients \citep{wang2022finite}, with extensions to high-dimensional logistic models \citep{hou2025repro}.
We transfer these ideas from support uncertainty (scaling primarily with $p$) to the qualitatively different regime of permutation uncertainty (scaling factorially with $n$): we build a data-dependent candidate set of permutations $\mC^{(L)}$ that contains the truth with high probability, then conduct inference uniformly over $\mC^{(L)}$ in the spirit of locally simultaneous inference.
This yields (i) a finite-sample valid test for the presence of any mismatch and (ii) union-type confidence sets for $\bm{\beta}_0$ that account for alignment uncertainty.

The general principles are shared with \cite{wang2022finite}: a candidate set is constructed from artificial noise copies, the so-called repro samples, and fixed-model confidence sets $\Gamma_\Pi$ are combined over a random candidate set $\mathcal{C}$ by the union bound
\begin{equation}\label{eq:union-principle}
\bP\Big(\bm{\beta}_0\notin\bigcup_{\Pi\in\mathcal{C}}\Gamma_\Pi\Big)
\le \bP\left(\Pi_0\notin\mathcal{C}\right)+\bP\left(\bm{\beta}_0\notin\Gamma_{\Pi_0}\right).
\end{equation}
Our contributions are tailored to permuted regression: a bound on the miscoverage probability of the candidate set that decays at the rate $1/L$ (Theorem~\ref{thm:recovery}), the sparsity-penalized LAP with its recovery guarantee, which requires separation only among the mismatched rows (Theorem~\ref{thm:high-probability equivalence}), and the tuning rule of Section~\ref{sec:tuning}, which does not require the number of mismatches.

\subsection{Organization}
Section~\ref{sec:methodology} develops a localization step via repro samples, yielding a data-dependent candidate set of permutations and a practical sparsity-penalized LAP. Section~\ref{sec:inference} then introduces finite-sample procedures: a conditional Monte Carlo test of no permutation and coefficient inference that remains valid under alignment uncertainty. Section~\ref{sec:theory} presents the theoretical guarantees: oracle identifiability, polynomial-rate miscoverage bounds for the candidate set, finite-sample validity of the test, coverage of the coefficient regions, and a high-probability recovery guarantee for the Hungarian surrogate. Section~\ref{sec:extensions} discusses partially permuted designs and ridge regularization. Section \ref{sec:tuning} discusses how to select the tuning parameters and Section~\ref{sec:numerical} reports empirical evidence from extensive simulations and an application to Beijing air-quality data. Finally, Section~\ref{sec:conclusion} concludes. The Appendix provides methodological background, additional theory, full proofs, and supplementary figures.

\subsection{Notation}
Vectors and matrices are bold.
$\|\cdot\|_2$ denotes the Euclidean norm, $\|\cdot\|_\infty$ the max norm, and $\|\cdot\|_{\mathrm{op}}$ the operator norm.
For a matrix $\bm A\in\R^{n\times p}$, $M_{\bm{A}} \in \R^{n \times n}$ denotes the orthogonal projection matrix onto $\range(\bm{A})$, and $R_{\bm{A}}:=I_n-M_{\bm{A}}\in \R^{n \times n}$ denotes the residual projection matrix.
The Hamming distance between permutations $\Pi$ and $I_n$ is $d(\Pi,I_n)=\sum_{i=1}^n\1\{\Pi_{ii}=0\}$.
We use $\mP_n$ for the set of all $n\times n$ permutation matrices and $\mP_{n,k}:=\{\Pi\in\mP_n: d(\Pi,I_n)\le k\}$ for the $k$-sparse subset.
For random variables, $\Normal(\mu,\Sigma)$ denotes the Gaussian distribution, and $F_{a,b}$ denotes an $F$-distribution with $(a,b)$ degrees of freedom.
\section{Candidate set of permutations via repro samples}\label{sec:methodology}
We develop the candidate set of permutation matrices, which is the key ingredient of our statistical inference framework.
This section includes (i) a localization approach based on the repro sample method, and (ii) a sparsity-penalized linear assignment problem (LAP) to handle our setup.

\subsection{Localization principle}\label{subsec:candidate}

We present the localization approach to make a candidate set $\mC^{(L)}\subset\mP_{n,k}$, which focuses on a small subset containing the permutation $\Pi_0$ with high probability and avoids the exhaustive search over all permutations in $\mP_{n,k}$.
This approach leverages the fact that the permutation estimated under artificially generated noise, known as a repro sample, is close to the true permutation identified under the realized noise.

To illustrate the idea, we first consider an ideal situation in which the unobserved realization of noise $\bm{u}^{\rel}$ in \eqref{eq:model intro} is fully available. The following Proposition \ref{prop:identification} claims that we can identify the true permutation $\Pi_0$ with $\bm{u}^{\rel}$ under $k$-sparse setup.

We make the following baseline assumption on the design matrix $\bm X$.
\begin{assumption}[Design matrix]\label{ass:generic-design}
The rows of $\bm X$ are generated i.i.d.\ from a distribution on
$\R^p$ with a Lebesgue density.
\end{assumption}
Throughout the paper, we treat $\bm X$ as a fixed design matrix and
conduct inference conditionally on its realized value.
The assumption above is imposed to guarantee that, with probability
one, the realized design matrix satisfies the required full-rank
conditions and, for a fixed $\bm\beta_0\neq\bm 0$, the entries of
$\bm X\bm\beta_0$ are pairwise distinct.

\begin{proposition}[Unique recovery in the oracle setting]\label{prop:identification}
    In addition to Assumption \ref{ass:generic-design}, suppose $n-2k \ge p$ and $\bm\beta_0\neq \bm 0$ hold. Then, almost surely, the true permutation $\Pi_0$ uniquely satisfies

\begin{equation}\label{eq:repro_ideal_Pi}
   \Pi_0 = \argmin_{\Pi \in \mP_{n,k}}\| (I_n - M_{(\Pi \bm{X}, \bm{u}^{\rel})}) \bm{Y}^{\obs}\|_2^2,
\end{equation}
where $M_{(\Pi \bm{X}, \bm{u}^{\rel})}$ is the orthogonal projector onto the range of $(\Pi \bm{X}, \bm{u}^{\rel})$.
\end{proposition}

When $\bm{u}^{\rel}$ is available, Proposition \ref{prop:identification} states that $\Pi_0$ is uniquely recovered. The complete proof is deferred to Appendix \ref{proof:prop identification}.

Next, we construct a candidate set in the case where $\bm{u}^{\rel}$ is not available.
By utilizing the insight from the oracle result as Proposition \ref{prop:identification}, we generate artificial samples of the noise that mimic the objective in \eqref{eq:repro_ideal_Pi}. Specifically, we generate $L$ i.i.d.\ vectors $\bm{u}_1^*, \ldots, \bm{u}_L^* \sim \Normal(0, I_n)$
and, for each $\ell=1,\ldots,L$, we define the least square problem and its optimizer
\begin{equation}\label{eq:repro_Pi}
   \hat{\Pi}_{\ell} = \argmin_{\Pi \in \mP_{n,k}}\| (I_n - M_{(\Pi \bm{X}, \bm{u}_{\ell}^*)}) \bm{Y}^{\obs}\|_2^2.
\end{equation}
This is an analogue of the oracle problem \eqref{eq:repro_ideal_Pi} by replacing the realization $\bm{u}^{\rel}$ by its artificial counterpart $\bm{u}^*_\ell$. We expect $\hat{\Pi}_{\ell}$ is close to $\Pi_0$ when $\bm{u}_{\ell}^*$ is sufficiently aligned with $\bm{u}^{\rel}$.

Finally, we collect $\hat{\Pi}_{1},\ldots,\hat{\Pi}_{L}$ to form the candidate set
\begin{equation}\label{eq:permutation candidates set}
    \mathcal{C}^{(L)}=\left\{\hat{\Pi}_{\ell}: \ell =1, \ldots, L\right\}.
\end{equation}

\subsection{Sparsity-penalized linear assignment problem} \label{sec:algorithm}

We develop a methodology to approximately solve the least square objective in \eqref{eq:repro_Pi}.
Since the objective in \eqref{eq:repro_Pi} minimizes a nonlinear projection residual over permutation matrices,  the problem remains computationally demanding due to its combinatorial nature even when we restrict the search to $\mP_{n,k}$ (see \cite{pananjady2016linear,pananjady2017linear}).
To obtain a tractable surrogate, we replace \eqref{eq:repro_Pi} by a sparsity-penalized linear assignment problem (LAP), which has the form $\min_{\Pi} \sum_{i} \sum_j \Omega_{ij} \Pi_{ij}$ with a cost matrix $\Omega \in \R^{n \times n}$ and can be solved exactly in $O(n^3)$ time by the Hungarian algorithm \citep{kuhn1955hungarian}.

To develop the LAP problem as a surrogate, for a fixed repro sample $\bm u^*$, we reform the original objective \eqref{eq:repro_Pi} into a (scaled) LAP cost plus a projection-mismatch term. Concretely, we expand \eqref{eq:repro_Pi} as follows: \begin{align}
&\left\|(I_n-M_{(\Pi\bm X,\bm u^*)})\bm Y^{\obs}\right\|_2^2 \label{eq:lap-core-decomp}\\
&=
\underbrace{\frac{1}{2}\sum_{i=1}^n\sum_{j=1}^n (Y_i-m_j)^2\Pi_{ij}}_{=:T_1(\Pi)}
+ \underbrace{\left(\sum_{i,j} Y_i m_j\Pi_{ij}-{\bm Y^{\obs}}^\top M_{(\Pi\bm X,\bm u^*)}\bm Y^{\obs}\right)}_{=: T_2(\Pi)}
+ \underbrace{\frac12\left(\|\bm Y^{\obs}\|_2^2-\sum_{j=1}^n m_j^2\right)}_{=:T_3}. \notag
\end{align}
The first term $T_1(\Pi)$ is a standard LAP objective with costs $(Y_i-m_j)^2$ where $m_j$ denotes the $j$-th element of $M_{(\bm{X}, \bm{u}^*)}\bm{Y}^{\obs}$, and the third term $T_3$ is ignorable since it does not depend on $\Pi$.
About the second term $T_2(\Pi)$, which vanishes at $\Pi=I_n$ and is not linear in $\Pi$, we replace it by a cost which induces an identity to be a minimizer.
In particular, with a coefficient $\lambda_1\ge0$,
we consider a LAP with a sparsity-inducing cost $\lambda_1\1\{j\neq i\}$, which penalizes each off-diagonal reassignment (an $\ell_0$-type cost on the number of moved indices).

Consequently, we consider an estimator by solving the following sparsity-penalized LAP as a surrogate for the least square problem \eqref{eq:repro_Pi}
\begin{align}  \label{def:tilde_Pi}
  \tilde{\Pi}
  = \argmin_{\Pi \in \mP_n}
    \sum_{i=1}^n \sum_{j=1}^n \Omega_{ij}(\bm{u}^*)\Pi_{ij},
\end{align}
where $\Omega(\bm{u}^*) \in \R^{n \times n}$ is a cost matrix with $\bm{u}^*$ defined as
\begin{align} \label{eq:Omega_ustar}
  \Omega_{ij}(\bm{u}^*)
  =  (Y_i - m_j)^2 + \lambda_1\1\{j\neq i\}.
\end{align}
In the implementation, the minimizer $\hat{\Pi}_{\ell}$ in \eqref{eq:permutation candidates set} is replaced by the solution $\tilde{\Pi}_{\ell}$ of \eqref{def:tilde_Pi} with $\bm{u}^*=\bm{u}^*_{\ell}$. The feasible set of \eqref{def:tilde_Pi} is the whole set $\mP_n$, so that neither the bound $k$ nor $k_\star$ is required.
Theorem~\ref{thm:high-probability equivalence} formalizes this reduction: for suitable $\lambda_1$ and under a separation condition on the mismatched rows, the sparsity-penalized LAP minimizer equals $\Pi_0$ with high probability, and then coincides with the argmin of \eqref{eq:repro_Pi} whenever this argmin equals $\Pi_0$ as well.
We will further discuss the selection of the parameter $\lambda_1$ in Section \ref{sec:tuning}.
For the sake of completeness, we present the original Hungarian algorithm in Algorithm \ref{alg:scoring hungarian} in Appendix \ref{sec:hungarian alg}.

\subsection{Unique recovery of the parameters}

We extend the results of the recovery of permutation matrices $\Pi_0$ in Proposition \ref{prop:identification} to include coefficients $(\bm{\beta}_0, \sigma_0) $. This is crucial for performing parameter inference.

\begin{proposition}\label{prop:identification_joint}
    Under the assumptions of Proposition~\ref{prop:identification}, suppose also $n>p$. Then, given the noise realization $\bm u^{\rel}$, the triple $(\Pi_0,\bm\beta_0,\sigma_0) \in \mP_{n,k} \times \R^p \times \R_+$ is almost surely the unique minimizer of
\begin{equation}\label{eq:repro_ideal}
   \min_{\Pi \in \mP_{n,k},\ \bm{\beta} \in \R^p,\ \sigma \ge 0 }
   \left\| \bm{Y}^{\obs} -  \Pi \bm{X} \bm{\beta} - \sigma \bm{u}^{\rel}\right\|_2^2.
\end{equation}

\end{proposition}

We compare the identifiability condition $n-2k\ge p$ in Proposition \ref{prop:identification_joint} to that of a previous study without the $k$-sparsity.
In particular, for the recovery problem without any constraints,  \citep{unnikrishnan2018unlabeled} derived the identifiability condition $n\ge 2p$.
Compared with this condition, our derived condition shows that $k$-sparsity can reduce the required sample size compared to the unrestricted case when $k$ is small.
See Theorem \ref{thm:counterexample} in Appendix \ref{section:appendix theory} for details.
\section{Finite-sample statistical inference}\label{sec:inference}

\subsection{Inference on the permutation matrix}\label{subsubsection:testing permutation}
This section constructs a statistic for the $k_0$-sparsity test with the hypotheses
\begin{equation}\label{eq:hypothesis-k0}
    H_0:\Pi_0\in\mP_{n,k_0}\quad\text{vs.}\quad H_1:\Pi_0\in\mP_{n,k}\setminus\mP_{n,k_0},
\end{equation}
where $k_0 \le k$, by introducing a conditional Monte Carlo method for valid finite-sample statistical inference of the permutation matrix $\Pi_0$. The special case with $k_0=0$ reduces to testing whether any mismatch is present (i.e., $\Pi_0 = I_n$). In the context of shuffled regression, the presence of a permutation is of direct interest \citep{slawski2021pseudo}.

\subsubsection{Applying localization}

First, we localize the null domain using the candidate set $\mC^{(L)}$ constructed in Section \ref{sec:methodology}.
The key property established there is that $\mC^{(L)}$ contains the true permutation $\Pi_0$ with high probability (see Theorem \ref{thm:recovery}).
Localization restricts attention to the data-dependent but typically small set $\hat{\mP}_{n,k_0}:=\mC^{(L)}\cap\mP_{n,k_0}$, on which we calibrate the rejection threshold uniformly over $\Pi\in\hat{\mP}_{n,k_0}$.
This approach reduces the computational burden of the conditional Monte Carlo calibration under
$H_0:\Pi_0\in\mP_{n,k_0}$, which requires controlling null distribution whose size is combinatorial in the sample size $n$.

More formally, we compute a least squares estimator over the candidate set,
\begin{equation}\label{eq:Pi-hat-LS}
    \hat{\Pi}(\bm Y)\in\argmin_{\Pi\in\mC^{(L)}}\left\|(I_n-M_{\Pi\bm X})\bm Y\right\|_2^2,
\end{equation}
and use the test statistic
\begin{equation}\label{eq:Dk-stat}
    D(\bm Y):=d\left(\hat{\Pi}(\bm Y),I_n\right),\qquad D^{\obs}:=D(\bm Y^{\obs}).
\end{equation}
Intuitively, under $H_0$, the best-fitting permutation in $\mC^{(L)}$ should typically have at most $k_0$ mismatches, whereas under $H_1$, it must typically use more than $k_0$ mismatches to explain the data.

\subsubsection{Handling nuisance}

Next, we remove an effect of the nuisance parameter $(\bm{\beta}, \sigma)$ to enable inference focused on $\Pi$, which is our primary interest.

To this aim, we consider a statistical model $\bm{Y}_{\bm{\theta}} = \Pi \bm{X} \bm{\beta} + \sigma \bm{U}$ with $\bm{\theta} = (\Pi, \bm{\beta}, \sigma) \in \mP_{n,k} \times \R^p \times \R_+$ and $\bm U\sim\Normal(0,I_n)$, then study a conditional distribution of $\bm{Y}_{\bm{\theta}}$, which is necessary for the conditional Monte Carlo approximation.

We derive the conditional distribution of $ \bm{Y}_{\bm \theta}$ using a formula of sufficient statistics.
By the reformulation $
\bm{Y}_{\bm{\theta}} = \bm{M}_{\Pi \bm{X}} \bm{Y}_{\bm{\theta}} + (I_n - \bm{M}_{\Pi \bm{X}})\bm{Y}_{\bm{\theta}} = \bm{M}_{\Pi \bm{X}} \bm{Y}_{\bm{\theta}} + \sigma(I_n - \bm{M}_{\Pi \bm{X}})\bm{U}
$ with a fixed $\Pi$,
we define a pair of sufficient statistics for $(\bm \beta, \sigma)$ as
\begin{equation}\label{eq:sufficient statistics}
    S_{1,\Pi}(\bm Y_{\bm \theta}):=M_{\Pi\bm X}\bm Y_{\bm \theta},\qquad S_{2,\Pi}(\bm Y_{\bm \theta}):=\left\|(I_n-M_{\Pi\bm X})\bm Y_{\bm \theta}\right\|_2.
\end{equation}
Then, we have the representation of the conditional distribution as
\begin{equation}\label{eq:conditional-representation}
    \bm Y_{\bm \theta}  \mid \left(S_{1,\Pi}(\bm Y_{\bm \theta}),S_{2,\Pi}(\bm Y_{\bm \theta})\right)
     \overset{d}{=}\
    S_{1,\Pi}(\bm Y_{\bm \theta})+S_{2,\Pi}(\bm Y_{\bm \theta})\frac{(I_n-M_{\Pi\bm X})\bm U}{\|(I_n-M_{\Pi\bm X})\bm U\|_2}.
\end{equation}
For a fixed $\Pi$, the pair $(S_{1,\Pi}(\bm Y_{\bm \theta}), S_{2,\Pi}(\bm Y_{\bm \theta}))$ is sufficient for $(\bm \beta, \sigma)$, and the direction $\frac{(I_n-M_{\Pi\bm X})\bm U}{\|(I_n-M_{\Pi\bm X})\bm U\|_2}$ is independent of this pair and has a distribution free of $(\bm \beta, \sigma)$. Consequently, the conditional distribution of $\bm Y_{\bm \theta}  \mid (S_{1,\Pi}(\bm Y_{\bm \theta}),S_{2,\Pi}(\bm Y_{\bm \theta}))$ is free of $(\bm \beta, \sigma)$. This yields the conditional Monte Carlo approximation of the distribution of $D(\bm Y_{\bm \theta})$ in \eqref{eq:Dk-stat}.

\subsubsection{Critical value}

We derive a critical value localized on $\hat{\mP}_{n,k_0}$ by the conditional Monte Carlo approximation.
Given $\bm Y^{\obs}$, we set $s_{1,\Pi}^{\obs}=S_{1,\Pi}(\bm Y^{\obs})$ and $s_{2,\Pi}^{\obs}=S_{2,\Pi}(\bm Y^{\obs})$.
Let $\widehat c_{1-\alpha}(\Pi)$ denote $(1-\alpha)$-quantile of the distribution of $D(\bm Y) \mid (S_{1,\Pi}(\bm Y),S_{2,\Pi}(\bm Y)) =(s_{1,\Pi}^{\obs}, s_{2,\Pi}^{\obs})$, and define the composite critical value localized on $\hat{\mP}_{n,k_0}$:
\begin{equation}\label{eq:critical-value-composite}
    \widehat c_{1-\alpha}^{(k_0)}:=\max_{\Pi\in\hat{\mP}_{n,k_0}}\widehat c_{1-\alpha}(\Pi), \mbox{~~and~~}
     \hat B^{(k_0)}_{1-\alpha}:=\{0,1,\ldots,\widehat c_{1-\alpha}^{(k_0)}\}.
\end{equation}
We reject $H_0$ when $D^{\obs}\notin \hat B^{(k_0)}_{1-\alpha}$ holds, i.e., when $D^{\obs}>\hat c_{1-\alpha}^{(k_0)}$ holds. When $\hat{\mP}_{n,k_0}$ is empty, we set $\widehat c_{1-\alpha}^{(k_0)}=-\infty$, so that $H_0$ is rejected. To approximate the critical value, we generate i.i.d.\ $\bm u_1^*,\ldots,\bm u_M^*\sim\Normal(0,I_n)$ and define
\[
    \bm Y_{m,\Pi}^*:=s_{1,\Pi}^{\obs}+s_{2,\Pi}^{\obs}\frac{(I_n-M_{\Pi\bm X})\bm u_m^*}{\|(I_n-M_{\Pi\bm X})\bm u_m^*\|_2},
    \qquad m=1,\ldots,M.
\]
For each $m=1,\dots, M$, we compute a test statistic $D_{m,\Pi}^*:=D(\bm Y_{m,\Pi}^*)$ as \eqref{eq:Dk-stat} and reject $H_0$ when the Monte Carlo $p$-value $(1+\sum_{m=1}^M\1\{D_{m,\Pi}^*\ge D^{\obs}\})/(M+1)$ is at most $\alpha$ for every $\Pi\in\hat{\mP}_{n,k_0}$; that is, $\widehat c_{1-\alpha}(\Pi)$ is replaced by the $\lceil(M+1)(1-\alpha)\rceil$-th smallest of $\{ D_{m,\Pi}^*\}_{m=1}^M$ ($+\infty$ if this index exceeds $M$). Here, the candidate set in \eqref{eq:Pi-hat-LS} is regenerated from each $\bm Y_{m,\Pi}^*$ in the same way as from $\bm Y^{\obs}$, so that $D$ is the same function of both responses.

For the ideal conditional quantiles $\widehat c_{1-\alpha}(\Pi)$, Theorem \ref{thm:p-value} shows that the type-I error is at most $\alpha+\delta_L$ in finite samples, where $\delta_L$ bounds the miscoverage probability of the candidate set. With the Monte Carlo $p$-values, $\alpha$ is replaced by $\alpha_M=\lfloor\alpha(M+1)\rfloor/(M+1)\le\alpha$ (Theorem \ref{thm:conditional-tv-transfer} in the Appendix, with zero total-variation error).

\begin{remark}\label{rem:alternative-statistic}
An alternative test statistic is possible. For example, the statistic $| \min_{\Pi \in \mP_{n,k}}\| (I_n - M_{\Pi \bm{X}}) \bm{Y}_{\bm{\theta}}\|_2^2 - \| (I_n - M_{\bm{X}}) \bm{Y}_{\bm{\theta}}\|_2^2 |$
is continuous and supported on $[0,\infty)$, which avoids computing the minimizer explicitly but complicates conditional Monte Carlo approximation. We use $D$ because it estimates the number of mismatches tested in \eqref{eq:hypothesis-k0} and is computed from $\mC^{(L)}$ alone. The type-I error bound of Theorem~\ref{thm:p-value} holds for either statistic, because its proof does not use the form of $D$. With $\mP_{n,k}$ replaced by $\mC^{(L)}\cup\{I_n\}$, the alternative statistic is computed from $\mC^{(L)}$ as well; both statistics then vanish when $\mC^{(L)}=\{I_n\}$, and they lead to the same decision in all the simulated tests of Section~\ref{sec:numerical}.
\end{remark}

\subsection{Inference on the coefficient vector}\label{subsec:coef}
We now construct a finite-sample valid confidence set for the coefficient vector $\bm{\beta}_0$, accounting for uncertainty from estimating the permutation. Since the parameter of interest is $\bm{\beta}$, we treat $\Pi$ and $\sigma$ as nuisance. We first consider a pivotal map that depends only on $(\Pi, \bm{\beta})$, and then quantify uncertainty about $\Pi$ via the candidate set $\mathcal{C}^{(L)}$. The resulting region is the union of coefficient confidence sets obtained by fixing each $\Pi \in \mathcal{C}^{(L)}$.

We define a statistics for the purpose. With given $(\Pi, \bm{\beta})$, we define it as
\begin{align*}
    T(\bm Y_{\bm \theta}, (\Pi, \bm{\beta}))
    &= \frac{\bm{U}^\top \bm{M}_{\Pi \bm{X}}\bm{U}/p}{\bm{U}^\top (I_n - \bm{M}_{\Pi \bm{X}}) \bm{U}/(n-p)}
     = \frac{(\bm Y_{\bm \theta} - \Pi \bm{X} \bm{\beta})^\top \bm{M}_{\Pi \bm{X}} (\bm Y_{\bm \theta} - \Pi \bm{X} \bm{\beta} )/p}{(\bm Y_{\bm \theta} - \Pi \bm{X} \bm{\beta} )^\top (I_n - \bm{M}_{\Pi \bm{X}})(\bm Y_{\bm \theta} - \Pi \bm{X} \bm{\beta} )/(n-p)}.
\end{align*}
Under the case $(\Pi, \bm{\beta}) = (\Pi_0, \bm{\beta}_0)$ with the true parameters, this statistic has an $F_{p, n-p}$ distribution.

Then, for $\alpha\in(0,1)$, we define the confidence set with the nominal level $1-\alpha$ as
\begin{equation}\label{eq:confidence set beta zero}
    \Gamma_{\alpha}^{\mathrm{coef}}(\bm{Y}^{\obs}) = \bigcup_{\Pi \in \mathcal{C}^{(L)}} \left\{\bm{\beta}\in \R^p : T(\bm{Y}^{\obs}, (\Pi, \bm{\beta})) \le F_{p, n-p}^{-1}(1-\alpha) \right\},
\end{equation}
where $F_{p, n-p}^{-1}$ is the quantile function of the $F_{p, n-p}$ distribution. We show in Theorem \ref{thm:confidence} how the coverage probability of $\Gamma_{\alpha}^{\mathrm{coef}}(\bm{Y}^{\obs})$ is controlled by the nominal level $1-\alpha$ and by the miscoverage probability $\delta_L$ of the candidate set.
\section{Theory}\label{sec:theory}

\subsection{High-probability recovery of $k$-sparse permutation matrices}

We study the candidate set \eqref{eq:permutation candidates set} and show that it contains the true permutation $\Pi_0$ with high-probability.
Let $\bP_{\left(\mathcal{U}^L, \bm{U}\right)}$ denote probability with respect to the joint law of $(\mathcal{U}^L, \bm{U})$, where $\mathcal{U}^L = (\bm{U}_1,\dots, \bm{U}_L)$.

To state the main results, we define the numerical constant
\begin{equation}\label{eq:c min}
     C_{\min}(\Pi_0) = \min_{\substack{\Pi \in \mP_{n,k}\\ \Pi \neq \Pi_0}} \frac{\|(I_n - \Pi M_{\bm{X}} \Pi^\top) \Pi_0\bm{X}\bm{\beta}_{0} \|_2^2}{n \max (d(\Pi_0, I_n) - d(\Pi, I_n), 1)},
\end{equation}
where $M_{\bm{X}} = \bm{X} (\bm{X}^\top \bm{X})^{-1}\bm{X}^\top$. 
The constant $C_{\min}(\Pi_0)$ is the degree of separation between the true structure $\Pi_0\bm{X}\bm{\beta}_{0}$ and its projections onto shuffled designs.

We give a theoretical analysis on the miscoverage probability of the candidate set $\mathcal{C}^{(L)}$.

\begin{theorem}[Coverage probability of the candidate set]\label{thm:recovery}
If $\mP_{n,k}=\{\Pi_0\}$, that is, $k\le1$, then $\mC^{(L)}=\{\Pi_0\}$ and we set $\delta_L=0$.
Otherwise, suppose $C_{\min}(\Pi_0) > 0$ and $n > p + 1$ hold. Then, for all $L \ge 1$, we have
\begin{equation}
\label{eq:recovery-finite-L}
\bP_{\left(\mathcal{U}^L, \bm{U}\right)}\left( \Pi_0 \notin \mathcal{C}^{(L)} \right) \le
\delta_L,
\qquad
\delta_L
= \min\left\{ 1,\ \frac{C_{\rm loc}}{L+1},\ \sum_{h=0}^{k} N_h u_h \right\}.
\end{equation}
Here, $N_h$ denotes the number of permutations $\Pi\in\mP_{n,k}\setminus\{\Pi_0\}$ with $d(\Pi,I_n)=h$, and, with $N=\sum_{h=0}^{k}N_h$ and $\kappa=5N(24\sqrt{p+1}+3)$, the quantities $u_h$ and $C_{\rm loc}$ are defined as
\begin{equation}\label{eq:delta_L}
\begin{aligned}
u_h&=\min\left\{1,\ \left(\frac43\right)^{p/2}\left(1+\frac{nC_{\min}(\Pi_0)\max(k_\star-h,1)}{4\sigma_0^2(n-p)}\right)^{-(n-p-1)/2}\right\},\\
C_{\rm loc}&=2e^{1/(2\kappa^2)}N\left\{\kappa\left(\frac{\|\Pi_0\bm X\bm\beta_0\|_2}{\sigma_0}+\sqrt n\right)\right\}^{n-p-1}\left(\frac{\|\Pi_0\bm X\bm\beta_0\|_2^2}{nC_{\min}(\Pi_0)}\right)^{(n-p-1)/2}.
\end{aligned}
\end{equation}

\end{theorem}

Theorem~\ref{thm:recovery} theoretically validates the localization principle that underlies our framework.
The candidate set $\mathcal{C}^{(L)}$ acts as an empirical analogue of a confidence region for $\Pi_0$: as $L$ grows, the probability that $\Pi_0$ is excluded from $\mathcal{C}^{(L)}$ decays polynomially.
Thus, Theorem~\ref{thm:recovery} shows that $\mathcal{C}^{(L)}$ provides a localized representation of the combinatorial parameter space of $\Pi$, turning an intractable search into a manageable and possibly small random subset that contains the truth with high probability.

The bound \eqref{eq:recovery-finite-L} makes the roles of $L$, $n$, $p$ and $k$ explicit. Its second term decays at the rate $1/L$, whose exponent does not depend on $n$ or $p$, although the constant $C_{\rm loc}$ is large. Its third term does not depend on $L$. For fixed $p$ and $k$, the number $N_h$ of competing permutations with $h$ mismatches grows polynomially in $n$, with degree $h\le k$, whereas $u_h$ decays exponentially in $n-p-1$ as long as $C_{\min}(\Pi_0)/\sigma_0^2$ is bounded away from zero, so that the third term tends to zero as $n$ increases; the separation can degenerate, for instance, when $k\ge k_\star+2$ and two correctly matched rows have nearly equal signals. The sparsity bound $k$ enters through $N$ and through $C_{\min}(\Pi_0)$, which is nonincreasing in $k$. The same quantities control the size of the candidate set: its expected size is at most $1+L\sum_{h=0}^{k}N_hu_h$ (Proposition~\ref{cor:compact-small-list} in the Appendix), so that the candidate set is small when this bound is close to one; the expected size is nondecreasing in $L$, and the bound is nondecreasing in $\sigma_0$.

Theorem~\ref{thm:recovery} is the permutation analogue of the model candidate set results  \citep{wang2022finite,hou2025repro} by the repro-sample approach: whereas those works treat support uncertainty (scaling with $p$), Theorem~\ref{thm:recovery} addresses permutation uncertainty (scaling with $n$), providing a finite-sample localization guarantee on the discrete setup.

\begin{remark}[Comparison with selective inference]
    Theorem~\ref{thm:recovery} offers a principled alternative to classical selective inference, where exact conditioning on the selection event is analytically and computationally infeasible for permutations due to their nonpolyhedral, combinatorial nature.
Our candidate-set approach replaces exact conditioning with a localization step: it identifies a small data-dependent subset that contains $\Pi_0$ with high probability and then performs inference uniformly over this subset.
\end{remark}

\subsection{A finite-sample valid test for the permutation matrix}
We now formalize the validity of the $k_0$-sparsity test introduced in Section~\ref{subsubsection:testing permutation}.
Fix $k_0\le k$ and define the statistic $D(\bm Y)$ in \eqref{eq:Dk-stat} and the localized null set $\hat{\mP}_{n,k_0}$.

Then, we validate the critical value $\hat{c}_{1-\alpha}^{(k_0)}$ and the confidence set $\hat{B}_{1-\alpha}^{(k_0)}$ in \eqref{eq:critical-value-composite}.

\begin{theorem}[Type-I error control for the sparsity test] \label{thm:p-value}
Suppose $C_{\min}(\Pi_0) > 0$ and $n > p + 1$ hold.
Then, for the true parameter $\theta_0=(\Pi_0,\bm\beta_0,\sigma_0)$ with $\Pi_0\in\mP_{n,k_0}$, we have the finite-sample bound
\begin{equation}\label{eq:marginal-validity-k0}
    \bP_{\left(\mathcal{U}^L, \bm{U}\right)}\left(D(\bm Y_{\theta_0})\notin \hat{B}^{(k_0)}_{1-\alpha}\right)\le \alpha+\delta_L,
\end{equation}
where $\delta_L$ is specified in Theorem \ref{thm:recovery}.
\end{theorem}

There are two terms in the right-hand side of \eqref{eq:marginal-validity-k0}.
The second term $\delta_L$ is the error from the localization, which is revealed in Theorem \ref{thm:recovery}.
The first term $\alpha$ is the type-I error without the localization.
To derive the term, we define the composite critical value with fixed $\Pi$:
\begin{equation}\label{eq:c-composite-theory}
    c_{1-\alpha}^{(k_0)}:=\max_{\Pi\in {\mP}_{n,k_0}}c_{1-\alpha}(\Pi), \mbox{~~and~~}
    B^{(k_0)}_{1-\alpha}:=\{0,1,\ldots,c_{1-\alpha}^{(k_0)}\},
\end{equation}
where $c_{1-\alpha}(\Pi)$ denotes the $(1-\alpha)$-quantile of the conditional distribution of $D(\bm Y_{\bm \theta})$ given $(S_{1,\Pi}(\bm Y_{\bm \theta}),S_{2,\Pi}(\bm Y_{\bm \theta}))$, then we show $\bP(D(\bm Y_{\theta_0})\notin B^{(k_0)}_{1-\alpha}\mid S_{1,\Pi_0}(\bm Y_{\theta_0}),\ S_{2,\Pi_0}(\bm Y_{\theta_0}))\le \alpha$ (See Proposition \ref{thm:p-value_without_delta}).
The key ingredients are the conditional representation \eqref{eq:conditional-representation} and the fact that $c_{1-\alpha}^{(k_0)}$ is chosen uniformly over the localized null $\hat \Pi\in {\mP}_{n,k_0}$.
In practice, we approximate the conditional quantiles by conditional Monte Carlo as described in Section \ref{subsubsection:testing permutation}.

\subsection{A finite-sample valid confidence set for the coefficient vector}

We next state the finite-sample coverage guarantee for the confidence region
constructed from the candidate permutation set $\mathcal{C}^{(L)}$.

\begin{theorem}[Coverage probability of the confidence set]\label{thm:confidence}
Suppose $C_{\min}(\Pi_0) > 0$ and $n > p + 1$ hold. Then, it holds that
\begin{equation}
         \bP_{\left(\mathcal{U}^L, \bm{U}\right)} \left( \bm{\beta}_0 \in \Gamma_{\alpha}^{\mathrm{coef}}(\bm{Y}^{\obs}) \right) \ge 1 - \alpha - \delta_L,
\end{equation}
where $\delta_L$ is specified in Theorem \ref{thm:recovery}.
\end{theorem}

Theorem~\ref{thm:confidence} establishes that the coverage of $\Gamma_{\alpha}^{\mathrm{coef}}(\bm{Y}^{\obs})$ for $\bm{\beta}_0$ attains the nominal level $1-\alpha$ up to the error $\delta_L$ from the localization for the candidate set.
As $L$ grows, the lower bound converges to the nominal level~$1-\alpha$.

Theorem~\ref{thm:confidence} concerns the candidate set formed by the minimizers of \eqref{eq:repro_Pi}. For the candidate set of the solutions of the sparsity-penalized LAP \eqref{def:tilde_Pi}, used in our numerical studies, \eqref{eq:union-principle} gives the coverage $1-\alpha-\bP(\Pi_0\notin\mC^{(L)})$; this probability is bounded through \eqref{eq:lap-recovery} for a fixed penalty under the condition \eqref{eq:lambda-window}, and by Proposition~\ref{prop:tuning-null} for the penalty of Section~\ref{sec:tuning} when no mismatch is present; otherwise, we examine the coverage by simulation in Section~\ref{sec:numerical}.

Theorem \ref{thm:confidence} represents a \emph{locally simultaneous} form of selective inference, which is conducted uniformly over a localized, data-dependent neighborhood of plausible permutations $\hat{\mP}_{n,k_0}$ that contains the true alignment with high probability (as shown in Theorem \ref{thm:recovery}), instead of conditioning on a single realized permutation. Within this neighborhood, classical reference laws such as the $F$ distribution remain valid, providing finite-sample coverage without the need to describe the full combinatorial selection event. At the same time, it highlights a notion of robustness to alignment uncertainty. In particular, if the true data-generating alignment is not uniquely identifiable but lies within this candidate set (as in record linkage, partial mismatching, or anonymization), then inference on $\beta_0$ still remains valid.

\subsection{Algorithmic guarantees}

In this section, we give conditions under which the solution to the LAP \eqref{def:tilde_Pi} of Section \ref{sec:algorithm} equals $\Pi_0$, and coincides with the solution to the original least square problem \eqref{eq:repro_Pi}, with high probability.
To clarify the dependence of the solutions on $\bm{U}$ and $\bm{U}^*$, we denote the solution to the LAP \eqref{def:tilde_Pi} as $\tilde{\Pi}(\bm{U}, \bm{U}^*)$, and the solution to the original problem \eqref{eq:repro_Pi} as $\hat{\Pi}(\bm{U}, \bm{U}^*)$.

To state the result, let $\bm{\mu}=\bm{X}\bm{\beta}_0$ be the vector of signals, $\mathcal{M}=\{i:(\Pi_0)_{ii}=0\}$ the set of mismatched indices, and $\bm{\delta}=(\Pi_0-I_n)\bm{X}\bm{\beta}_0$ the displacement of the signals caused by $\Pi_0$. Let $g_\star=\min\{|\mu_a-\mu_b|:a,b\in\mathcal{M},\ a\neq b\}$ be the smallest gap between the signals of the mismatched indices, with $g_\star=+\infty$ when $\mathcal{M}=\emptyset$. For $\bm{u}\in\R^n$ and $\xi\in(0,1)$, let $h_{\bm{u}}$ denote the largest diagonal element of $M_{(\bm{X},\bm{u})}$ and define
\begin{equation}\label{eq:lap-rho}
\rho_\xi(\bm{u})=\sigma_0\left(1+h_{\bm{u}}^{1/2}\right)\sqrt{2\log(4n/\xi)}+\left\|M_{(\bm{X},\bm{u})}\bm{\delta}\right\|_\infty.
\end{equation}

\begin{theorem}[High-probability recovery and equivalence]\label{thm:high-probability equivalence}
Let $\lambda_1\ge0$ be fixed. For $\xi \in (0,1)$, let $\mathcal{G}_\xi$ be the set of vectors $\bm{u}\in\R^n$ satisfying
\begin{equation}\label{eq:lambda-window}
g_\star>2\rho_\xi(\bm{u})
\quad\mbox{and}\quad
2\rho_\xi(\bm{u})^2<\lambda_1<g_\star\left\{g_\star-2\rho_\xi(\bm{u})\right\}.
\end{equation}
Then, for every fixed $\bm{u}\in\mathcal{G}_\xi$, the LAP \eqref{def:tilde_Pi} with $\bm{u}^*=\bm{u}$ has the unique solution $\Pi_0$ with probability at least $1-\xi$. Consequently, we have
\begin{equation}\label{eq:lap-recovery}
    \bP_{\bm{U}, \bm{U}^*}\left(\tilde{\Pi}(\bm{U}, \bm{U}^*) = \Pi_0 \right) \ge (1 - \xi) \bP_{\bm{U}^*}\left(\bm{U}^*\in\mathcal{G}_\xi\right),
\end{equation}
and, if $C_{\min}(\Pi_0) > 0$ and $n > p + 1$ hold in addition, we have
\begin{equation}\label{eq:lap-agreement}
    \bP_{\bm{U}, \bm{U}^*}\left(\hat{\Pi}(\bm{U}, \bm{U}^*) = \tilde{\Pi}(\bm{U}, \bm{U}^*) \right) \ge (1 - \xi) \bP_{\bm{U}^*}\left(\bm{U}^*\in\mathcal{G}_\xi\right) - \delta_1,
\end{equation}
where $\delta_1$ is specified in Theorem \ref{thm:recovery} with $L=1$.
\end{theorem}

About the constraint \eqref{eq:lambda-window}, $\rho_\xi(\bm{u})$ is the upper bound of the noise in the responses and in the fitted values $M_{(\bm{X},\bm{u})}\bm{Y}$, including the bias caused by the mismatched rows,
while $g_\star$ measures the separation between the signals of the mismatched rows. The lower bound on $\lambda_1$ prevents the correctly matched rows from being reassigned to fit the noise, and the upper bound ensures that every mismatched row is reassigned
(see Proposition \ref{prop:lap-rowwise} in Appendix~\ref{app:lap-recovery-proof}).

Theorem~\ref{thm:high-probability equivalence} establishes a nonasymptotic guarantee for a single linear assignment problem. Specifically, under an explicit finite-sample separation condition (quantified by $g_\star$) and a tuning window \eqref{eq:lambda-window}, the solution of the sparsity-penalized LAP equals the true permutation $\Pi_0$ with high probability, and is then identical to the global minimizer of the least square problem except on an event of probability at most $\delta_1$.
The window involves the unknown quantities $(\sigma_0,\bm{\delta},g_\star)$ and describes a regime rather than a tuning rule; a data-driven choice of $\lambda_1$ is given in Section~\ref{sec:tuning}.
The bound \eqref{eq:lap-agreement} is informative when $\bP_{\bm U^*}(\bm U^*\in\mathcal G_\xi)$ is close to one and $\delta_1$ is small. The latter requires $C_{\min}(\Pi_0)$ to be large relative to $\sigma_0^2$; otherwise, the least square problem \eqref{eq:repro_Pi} can exchange correctly matched rows with nearly equal signals to fit the noise, whereas the penalized LAP does not.

\begin{remark}[Non-Gaussian errors]\label{rem:compact-localization-noise}
The results of this section assume Gaussian errors; Appendix~\ref{app:extension-details} states what remains valid otherwise, with the repro samples kept Gaussian.
For jointly spherical errors, such as multivariate $t$ errors with a common scale, the conditional representation \eqref{eq:conditional-representation} and the $F$-distribution of $T$ remain exact, and a bound of the order $1/L$ on the miscoverage probability holds without moment conditions (Proposition~\ref{prop:spherical}, Theorem~\ref{ta-s-thm-spherical-tube}).
For sub-Gaussian errors, the miscoverage probability is bounded under a large separation (Proposition~\ref{ta-n-sg-sharp}).
For general errors, the main obstacle is the loss of \eqref{eq:conditional-representation}: if the test uses the Monte Carlo rank $p$-values of Theorem~\ref{thm:conditional-tv-transfer}, its type-I error is at most $\alpha$ plus the expected total-variation distance between the true and the simulated conditional distributions, plus the miscoverage probability $\bP\{\Pi_0\notin\mC^{(L)}\}$ of the candidate set (Theorem~\ref{thm:conditional-tv-transfer}).
The bound \eqref{eq:lap-recovery} holds whenever $\bE\exp(\bm{t}^\top\bm{U})\le\exp(\|\bm{t}\|_2^2/2)$ for all $\bm{t}\in\R^n$.
\end{remark}
\section{Extensions}\label{sec:extensions}

\subsection{Partially permuted designs}\label{subsec:partial-design}

In some applications the covariates can be decomposed as $(\bm X,\bm Z)$, where only $\bm X$ is
misaligned with the response.  We consider the partially permuted
Gaussian linear regression model
\begin{equation}\label{eq:model_partial}
\bm Y = \Pi_0 \bm X \bm\beta_{1,0} + \bm Z \bm\beta_{2,0} + \sigma_0 \bm U,
\qquad
\bm U\sim N(0,I_n),
\end{equation}
where $\bm X\in\mathbb R^{n\times p_1}$ and $\bm Z\in\mathbb R^{n\times p_2}$ are fixed.

For any $\Pi\in\mP_{n,k}$, define the expanded design matrix
\begin{equation}\label{eq:Wpartial-def}
W(\Pi):=\big[\Pi\bm X,\ \bm Z\big]\in\mathbb R^{n\times (p_1+p_2)},
\qquad
M_{W(\Pi)}:=W(\Pi)\left(W(\Pi)^\top W(\Pi)\right)^{-1}W(\Pi)^\top.
\end{equation}
Then the objective \eqref{eq:repro_Pi} becomes
\begin{equation}\label{eq:repro_Pi_partial}
F_{\rm partial}(\Pi):=\left\| (I_n-M_{W(\Pi)})\bm Y^{\obs}\right\|_2^2.
\end{equation}

Permutation candidates are generated exactly as in Section \ref{subsec:candidate}
by augmenting $W(\Pi)$ with the random direction $\bm u_\ell^*$:
\[
\widehat\Pi_{\ell}^{\rm partial}
\in \arg\min_{\Pi\in\mP_{n,k}}
\left\| (I_n - M_{(W(\Pi),\bm u_{\ell}^*)})\bm Y^{\obs}\right\|_2^2,
\qquad
\ell=1,\dots,L,
\]
and we define the candidate set
\begin{equation}\label{eq:perm-candset-partial}
\mathcal{C}_{\rm partial}^{(L)}
:=\left\{
\Pi\in\mP_{n,k}:\ \exists \ell\in[L]\ \text{s.t.}\ \widehat\Pi_{\ell}^{\rm partial}=\Pi
\right\}.
\end{equation}
Under a
separation condition $C_{\min}^{\rm partial}(\Pi_0)>0$, the miscoverage probability
$\bP(\Pi_0\notin \mathcal{C}_{\rm partial}^{(L)})$ is bounded by $\delta_L^{\rm partial}$, which is defined as $\delta_L$ in Theorem~\ref{thm:recovery} with $p$, $C_{\min}(\Pi_0)$ and $\Pi_0\bm X\bm\beta_0$ replaced by $p_1+p_2$, $C_{\min}^{\rm partial}(\Pi_0)$ and $\Pi_0\bm X\bm\beta_{1,0}+\bm Z\bm\beta_{2,0}$, respectively, because the proof of Theorem~\ref{thm:recovery} applies to any finite collection of design matrices of the same rank (Appendix~\ref{sec:partial}).

Given $\Pi\in\mathcal{C}_{\rm partial}^{(L)}$, a joint confidence set for $(\bm\beta_1,\bm\beta_2)$
can be built using the standard $F_{p_1+p_2,n-(p_1+p_2)}$-distribution based on $W(\Pi)$:
\begin{equation}\label{eq:Tjointpartial-def}
T(\bm Y^{\obs},(\Pi,\bm\beta_1,\bm\beta_2))
:=\frac{n-(p_1+p_2)}{p_1+p_2}
\frac{\left(\bm Y^{\obs}-W(\Pi)\binom{\bm\beta_1}{\bm\beta_2}\right)^\top M_{W(\Pi)}
      \left(\bm Y^{\obs}-W(\Pi)\binom{\bm\beta_1}{\bm\beta_2}\right)}
     {\left(\bm Y^{\obs}-W(\Pi)\binom{\bm\beta_1}{\bm\beta_2}\right)^\top (I_n-M_{W(\Pi)})
      \left(\bm Y^{\obs}-W(\Pi)\binom{\bm\beta_1}{\bm\beta_2}\right)},
\end{equation}
and the resulting confidence set with the nominal level $1-\alpha$ is given by
\begin{equation}\label{eq:confset-partial-joint}
\Gamma_{\alpha}^{\mathrm{joint}}(\bm{Y}^{\obs})
:=\bigcup_{\Pi \in \mathcal{C}_{\mathrm{partial}}^{(L)}} \left\{
(\bm\beta,\bm\beta')\in \R^{p_1+p_2}:\
T(\bm Y^{\obs},(\Pi,\bm\beta,\bm\beta')) \le F_{p_1+p_2,n-(p_1+p_2)}^{-1}(1-\alpha)
\right\}.
\end{equation}

If the main interest is $\bm\beta_1$ with $\bm\beta_2$ treated as a nuisance parameter, we use the
partialled--out $F$-distribution.  Let $M_{\bm Z}$ be the orthogonal projector onto $\range(\bm Z)$ and
set $P_{\bm Z}:=I_n-M_{\bm Z}$. For any $\Pi\in\mP_{n,k}$ and $\bm\beta_1\in\mathbb R^{p_1}$,
define
\begin{equation}\label{eq:Tpartial-def}
T(\bm Y^{\obs},(\Pi,\bm\beta_1)\mid \bm Z)
:=\frac{n-(p_1+p_2)}{p_1}
\frac{
(\bm Y^{\obs}-\Pi\bm X\bm\beta_1)^\top (M_{W(\Pi)}-M_{\bm Z}) (\bm Y^{\obs}-\Pi\bm X\bm\beta_1)
}{
(\bm Y^{\obs}-\Pi\bm X\bm\beta_1)^\top (I_n-M_{W(\Pi)}) (\bm Y^{\obs}-\Pi\bm X\bm\beta_1)
}.
\end{equation}
The resulting confidence set with the nominal level $1-\alpha$ is given by
\[
\Gamma_{\alpha}^{\rm partial}(\bm Y^{\obs})
:=\bigcup_{\Pi\in\mathcal C_{\rm partial}^{(L)}}
\left\{
\bm\beta_1\in\mathbb R^{p_1}:\
T(\bm Y^{\obs},(\Pi,\bm\beta_1)\mid \bm Z)\le F_{p_1,n-(p_1+p_2)}^{-1}(1-\alpha)
\right\},
\]
and it covers $\bm\beta_{1,0}$ with probability at least $1-\alpha-\delta_L^{\rm partial}$ (Appendix~\ref{sec:partial}).

\subsection{Permuted ridge regression}
Our methodology can also be extended to a permuted regression model with a ridge penalty.
Consider the data generating process analogous to~\eqref{eq:random},
\[
\bm Y^{\mathrm{obs}}=\Pi_{0}\bm X\bm{\beta}_{0}+\sigma_{0}\bm U,
\qquad \Pi_0\in\mP_{n,k}.
\]
For a fixed $\lambda>0$, define the ridge objective
\[
F_\lambda(\Pi):=\min_{\bm{\beta}\in\R^p}
\left\{\|\bm Y^{\mathrm{obs}}-\Pi\bm X\bm{\beta}\|_2^2+\lambda\|\bm{\beta}\|_2^2\right\},
\qquad \Pi\in\mP_{n,k}.
\]

It is convenient to rewrite the ridge objective in an augmented least-squares form
(e.g.,~\cite{hastie2020ridge}). Define
\[
\bm Y^{\mathrm{aug}}:=\begin{pmatrix}\bm Y^{\mathrm{obs}}\\ \bm 0_p\end{pmatrix}\in\R^{n+p},
\qquad
\bm X^{\mathrm{aug}}:=\begin{pmatrix}\bm X\\ \sqrt{\lambda}I_p\end{pmatrix}\in\R^{(n+p)\times p},
\qquad
\tilde\Pi:=\begin{pmatrix}\Pi & 0\\ 0 & I_p\end{pmatrix}\in\mathcal P_{n+p}.
\]
Then $F_\lambda(\Pi)=\min_{\bm{\beta}}\|\bm Y^{\mathrm{aug}}-\tilde\Pi\bm X^{\mathrm{aug}}\bm{\beta}\|_2^2$.
Let $\tilde\Pi_0$ denote the block-diagonal embedding of $\Pi_0$, namely
\[
\tilde\Pi_0 := \begin{pmatrix}\Pi_0 & 0\\ 0 & I_p\end{pmatrix}\in\mathcal P_{n+p},
\]
and define the augmented noise $\tilde{\bm U}:=(\bm U,\bm 0_p)\in\R^{n+p}$.
Under the model we have $\bm Y^{\mathrm{aug}}=\tilde\Pi_0\bm X^{\mathrm{aug}}\bm{\beta}_0+\sigma_0\tilde{\bm U}-\bm b_\lambda$ with the bias $\bm b_\lambda:=(\bm 0_n,\sqrt{\lambda}\bm{\beta}_0)\in\R^{n+p}$.
For $\Pi\in\mP_{n,k}$ write $
W_\lambda(\Pi):=\tilde\Pi\bm X^{\mathrm{aug}}\in\R^{(n+p)\times p}$.

For i.i.d.\ $\bm U^*_1,\dots,\bm U^*_L\sim N(0,I_n)$ let
$\tilde{\bm U}^*_\ell:=(\bm U^*_\ell,\bm 0_p)$ and define
\begin{equation}\label{eq:perm_candidates_ridge}
\hat\Pi_{\lambda,\ell}\in
\arg\min_{\Pi\in\mP_{n,k}}
\left\|
(I_{n+p}-M_{(W_\lambda(\Pi),\tilde{\bm U}^*_\ell)})\bm Y^{\mathrm{aug}}
\right\|_2^2,
\qquad
\mC_\lambda^{(L)}:=\{\hat\Pi_{\lambda,1},\dots,\hat\Pi_{\lambda,L}\}.
\end{equation}

\begin{remark} In the augmented representation, the last $p$ pseudo-observations are noiseless and carry the bias $\bm b_\lambda$, so that the objective in \eqref{eq:perm_candidates_ridge} does not vanish at the true permutation even when $\tilde{\bm U}^*_\ell$ is replaced by the realized noise; its value there is at most $\lambda\|\bm\beta_0\|_2^2$ (Proposition~\ref{prop:ridge-profile} in Appendix~\ref{sec:ridge}). The miscoverage probability $\bP(\Pi_0\notin\mC_\lambda^{(L)})$ is bounded when this bias is small relative to the separation $nC_{\min}(\Pi_0)$ (Proposition~\ref{prop:ridge-finite}). The ridge penalty is used only for generating the candidates: the confidence set for $\bm\beta_0$ is the union of the sets in \eqref{eq:confidence set beta zero} over $\mC_\lambda^{(L)}$, and its coverage is at least $1-\alpha-\bP(\Pi_0\notin\mC_\lambda^{(L)})$. \end{remark}
\section{Selection of tuning parameters}\label{sec:tuning}

We propose a selection method of the regularization parameter $\lambda_1$ for the sparsity-penalized LAP in Section~\ref{sec:algorithm}.
The method requires neither $k_\star$ nor the bound $k$.

We choose $\lambda_1$ at the level of the largest squared residual that pure noise produces. Reassigning the $i$-th row lowers the cost \eqref{eq:Omega_ustar} by at most $(Y_i-m_i)^2$ and incurs the penalty $\lambda_1$, so that $I_n$ is a solution of the LAP \eqref{def:tilde_Pi} whenever $\max_{1\le i\le n}(Y_i-m_i)^2\le\lambda_1$ holds; when no mismatch is present, the residuals $Y_i-m_i$ consist of noise only.

\begin{enumerate}[nosep]
\item For $\bm y\in\R^n$, let $s(\bm y)$ be the normalized median absolute deviation of the residual vector $\bm r=R_{\bm X}\bm y$, that is,
      $s(\bm y):=\operatorname{median}_{i}|r_i-\operatorname{median}_{j}r_j|/\Phi^{-1}(3/4)$, where $\Phi$ is the standard normal distribution function, and compute $s(\bm Y^{\obs})$.
\item For the repro samples $\bm u_1^*,\ldots,\bm u_L^*$, define the statistic
      \begin{equation}\label{eq:tuning-pivot}
      Q(\bm y):=\frac{\max_{1\le\ell\le L}\left\|R_{(\bm X,\bm u_\ell^*)}\bm y\right\|_\infty^2}{s(\bm y)^2}.
      \end{equation}
      Generate $B$ i.i.d.\ vectors $\bm g_1,\ldots,\bm g_B\sim\Normal(0,I_n)$ and, for a level $\eta\in(0,1)$, let $\widehat q$ be the $\lceil(B+1)(1-\eta)\rceil$-th smallest value among $Q(\bm g_1),\ldots,Q(\bm g_B)$.
\item Set $\lambda_1=s(\bm Y^{\obs})^2 \widehat q$.
\item If the LAP returns $\tilde\Pi\neq I_n$, recompute the fitted values, the scale $s$ and the penalty $\lambda_1$ with $\bm X$ replaced by $\tilde\Pi\bm X$, and solve the LAP again. Repeat this step until the solution no longer changes.
\end{enumerate}

Since $Q(\bm X\bm b+c \bm y)=Q(\bm y)$ holds for all $\bm b\in\R^p$ and $c>0$, the distribution of $Q(\bm Y)$ is free of $(\bm\beta_0,\sigma_0)$ when $\Pi_0=I_n$, and $\widehat q$ is a Monte Carlo quantile of this distribution, in the manner of the quantile universal threshold of \cite{giacobino2017quantile}. This yields the following guarantee, which is proved in Appendix~\ref{app:tuning-null-proof}.

\begin{proposition}[Guarantee without mismatches]\label{prop:tuning-null}
Suppose $\Pi_0=I_n$ and $n>p$ hold, and let $\mC^{(L)}=\{\tilde\Pi_\ell:\ell=1,\ldots,L\}$ be the candidate set generated by the LAP \eqref{def:tilde_Pi} with the tuning parameters above. Then, we have
\begin{equation}\label{eq:tuning-null-guarantee}
\bP\left(\mC^{(L)}=\{I_n\}\right)\ge1-\eta_B,
\qquad
\eta_B:=\frac{\lfloor\eta(B+1)\rfloor}{B+1}\le\eta .
\end{equation}
Consequently, the confidence set \eqref{eq:confidence set beta zero} constructed from this candidate set covers $\bm\beta_0$ with probability at least $1-\alpha-\eta_B$, and the test of $H_0:\Pi_0=I_n$ in Section~\ref{subsubsection:testing permutation} has a type-I error of at most $\eta_B$, because $D^{\obs}=0$ holds when the candidate set is $\{I_n\}$.
\end{proposition}

When mismatches are present, $\lambda_1^{1/2}=s(\bm Y^{\obs}) \widehat q^{1/2}$ acts as a detection threshold, which does not depend on $k_\star$: by the argument above, the LAP ignores mismatches whose signal displacements are small relative to $\lambda_1^{1/2}$. The threshold is of the order $\sigma_0(2\log n)^{1/2}$ \citep[cf.][Section~2.3]{slawski2019linear} when $s(\bm Y^{\obs})$ remains comparable to $\sigma_0$, which need not hold, because $R_{\bm X}\bm Y^{\obs}$ also contains the projected displacements $R_{\bm X}(\Pi_0-I_n)\bm X\bm\beta_0$. This $\lambda_1$ is not guaranteed to lie in the window \eqref{eq:lambda-window}, and we examine the case with mismatches, including the detection of large displacements, by simulation in Section~\ref{sec:numerical}.

Neither $k_\star$ nor $k$ is an input of the procedure; $k\ge k_\star$ is an assumption of Theorem~\ref{thm:recovery} only. The computational cost is linear in $L$, since each repro sample requires one LAP with cost $O(n^3)$. In Section~\ref{sec:numerical}, the median computing time for one response with $L=400$, including the test and the confidence set, is $1.22$, $3.51$ and $17.99$ seconds for $n=100$, $200$ and $400$ on a single thread.
In additional simulations with $n=100$, $\sigma_0\in\{0.1,0.3,0.5\}$, $L\in\{25,50,100,200,400,800\}$ and $\eta\in\{0.01,0.05,0.2,0.5\}$, and otherwise the design of Section~\ref{sec:numerical}, the mean coverage of the confidence set over the configurations with $k_\star>0$ is between $0.948$ and $0.962$ for every combination of $L$ and $\eta$.

Through our experiments in Section \ref{sec:numerical}, we adopt $\eta=0.05$, $B=199$ and $L\in\{200,400\}$ as a default rule.
\section{Numerical results}\label{sec:numerical}

\subsection{Simulation analysis}
In this section, we assess the finite-sample performance of the proposed method.
We focus on three aspects:
(i) the size and coverage probability of the candidate set (Section \ref{subsec:candidate});
(ii) the size and power of the conditional test for the existence of a permutation (Section \ref{subsubsection:testing permutation}); and
(iii) the coverage probability and the volume of the confidence set for the coefficient vector (Section \ref{subsec:coef}).

All experiments are based on the shuffled linear regression model
\[
\bm{Y} = \Pi_0 \bm{X}\bm{\beta}_0 + \sigma_0 \bm{U}, \qquad \bm{U}\sim\Normal(\bm{0}, I_n),
\]
where $\bm{\beta}_0=(0.5,-1,2)^\top$ and the design matrix $\bm{X} \sim\Normal(\bm{0},I_p)$ with $p=3$.
For each configuration, $\Pi_0$ and $\bm{X}$ are fixed across Monte Carlo replications.

We vary the sample size $n\in\{100,200,400\}$, the number of mismatches $k_\star\in\{0,2,5,10,15,20\}$,
the noise level $\sigma_0\in\{0.01,0.05,0.10,0.20,0.30,0.40,0.50\}$,
and the number of repro samples $L\in\{200,400\}$.
Each setting is replicated $200$ times.
The repro step employs the sparsity-penalized LAP (Algorithm~\ref{alg:scoring hungarian})
with the tuning parameters specified in Section \ref{sec:tuning}; thus, $\mathcal{C}^{(L)}$ in this section consists of the LAP solutions $\tilde{\Pi}_1,\ldots,\tilde{\Pi}_L$, and neither $k_\star$ nor the bound $k$ is supplied to the procedure.
The tests use $M=199$ conditional Monte Carlo samples.

\subsubsection{Performance of the candidate set}
We assess the empirical performance of the candidate set $\mathcal{C}^{(L)}$.
For each configuration, we compute the average size of $\mathcal{C}^{(L)}$ (the number of unique permutations) across Monte Carlo replications, the empirical probability that the true permutation $\Pi_0$ is contained in $\mathcal{C}^{(L)}$, and the average matching fraction defined below.

As summarized in Figure~\ref{fig:size candidate},
the size of $\mathcal{C}^{(L)}$ tends to increase with the noise level $\sigma_0$ and the number of mismatches $k_\star$,
reflecting greater combinatorial ambiguity under weaker signal or stronger misalignment, and the candidate sets are small: the average size is at most $7.1$ in every configuration. When $k_\star=0$, we have $\mathcal{C}^{(L)}=\{I_n\}$ for both values of $L$ and all the $4{,}200$ simulated responses, in accordance with Proposition~\ref{prop:tuning-null}.
To evaluate the alignment quality, we report a matching-fraction diagnostic that measures
how closely the best candidate permutation aligns with the true $\Pi_0$.
For each Monte Carlo iteration, we compute the normalized complement of the
Hamming distance $
\psi = 1 - \frac{1}{n} \min_{1 \le \ell \le L} d(\tilde{\Pi}_\ell, \Pi_0)$,
which represents the proportion of correctly matched indices in the best case among the $L$ repro samples.

We then average $\psi$ across the Monte Carlo replications and report it in Figure~\ref{fig:fraction}.
The matching fraction decreases with increasing $\sigma_0$ and $k_\star$,
improves with larger $n$, and approaches $100\%$ when $k_\star=0$.
Although perfect recovery is generally hard due to the discrete nature of permutation matrices,
the candidate sets achieve sufficiently high matching fractions,
indicating that most rows are correctly aligned even under moderate noise.
These results confirm that repro-sample localization effectively narrows the permutation space
while keeping the best candidate permutation close to the true $\Pi_0$.

\subsubsection{Size and power of the test for the existence of a permutation}
Next, we assess the finite-sample size and power of the conditional Monte Carlo test described in Section \ref{subsubsection:testing permutation}.
We use the same baseline model and parameter grid as above.
For each configuration, we record whether the null $H_0:\Pi_0=I_n$ is rejected at nominal level $\alpha=0.05$, where the critical value is calibrated at $\Pi=I_n$.
The empirical rejection rates
are plotted in Figure~\ref{fig:test size power}.
When $k_\star=0$ (no mismatch), the rejection frequency remains at or below the nominal significance level across all $\sigma_0$.
When $k_\star>0$, the power is governed by the largest signal displacement $\|\bm{\delta}\|_\infty$ of the mismatched rows relative to the threshold $\lambda_1^{1/2}\approx4.1 s(\bm{Y}^{\obs})$ of Section~\ref{sec:tuning}, rather than by $k_\star$ itself.
The rejection frequency averaged over the configurations with $k_\star>0$ is $0.946$, and it falls substantially below one only for $k_\star=5$ with $n=200$ and $n=400$ at the larger noise levels, where $\|\bm{\delta}\|_\infty$ equals $0.81$ and $1.42$, respectively.
Differences between $L=200$ and $L=400$ are modest.
With $\mP_{n,k}$ replaced by $\mC^{(L)}\cup\{I_n\}$, the statistic of Remark~\ref{rem:alternative-statistic} gives the same decision as $D$ in all tests.

\subsubsection{Coverage probability of the confidence set for the coefficient vector}
Finally, we evaluate the finite-sample coverage and size of the confidence sets
constructed for the regression coefficient vector $\bm{\beta}_0$
using the union method described in Section~\ref{subsec:coef}.
For each configuration, we construct the confidence set $
\Gamma_\alpha^{\mathrm{coef}}(\bm{Y}^{\mathrm{obs}})$ with $1-\alpha=0.95$, and record whether $\bm{\beta}_0$ lies inside the region as well as its width, measured by the geometric mean of the lengths of its coordinate projections.

The results in Figure~\ref{fig:coverage coef} show that the empirical coverage probability closely matches the nominal level across all noise regimes.

Table~\ref{tab:sim-benchmarks} compares the proposed confidence set with four alternatives computed from the same responses: the oracle set that uses the true permutation $\Pi_0$, the naive set that uses $I_n$, a robust-regression set, namely the Wald-type set of the Huber estimator used by \cite{slawski2019linear} with the constant $1.345$ and an iteratively re-estimated scale, and the Wald-type set of the pseudo-likelihood estimator of \cite{slawski2021pseudo}.
The proposed set has coverage close to that of the oracle and a median width ratio of $1.01$ relative to the oracle, so that the efficiency loss due to the uncertainty of the permutation is small in these designs.
The minimum coverage is $0.000$ for the naive regression set, $0.730$ for the robust-regression set and $0.910$ for the asymptotic pseudo-likelihood set.
In a comparison with the fractional posterior of \cite{chakraborty2024learning}, which we implemented from the description in that paper with the fractional power $1/n$ recommended there, on twelve data sets for each pair of $\sigma_0\in\{0.1,0.5\}$ and $k_\star\in\{0,10\}$ with $n=100$ and $p=5$, the proposed set covers $\bm{\beta}_0$ in $45$ of the $48$ cases, as does the oracle set; for each pair, the median volume of the smallest box containing the proposed set is smaller than that of the box formed by the posterior credible intervals with the Bonferroni adjustment by a factor of more than $10^{9}$.

In addition to the scalar summaries of coverage, we visualize the geometry of the constructed confidence region.
Figure~\ref{fig:repro_union_three_planes} shows three two-dimensional projections of the union confidence set
onto the coordinate planes
for one representative configuration ($n=100$, $k_\star=5$, $L=200$, $\sigma_0=0.10$).
In all projections, the true parameter $\bm{\beta}_0=(0.5,-1,2)^\top$ lies well within the confidence region,
demonstrating that the procedure is consistent with the finite-sample guarantee.
Moreover, the projected regions are compact and tightly concentrated around the true value,
illustrating that taking unions over $\mC^{(L)}$
does not lead to excessive conservatism: the volumes relative to the oracle set are reported in Figure~\ref{fig:app-coef-volume} in the Appendix.

\begin{figure}[htbp]
    \centering
    \begin{subfigure}[b]{0.32\textwidth}
        \centering
        \includegraphics[width=\textwidth]{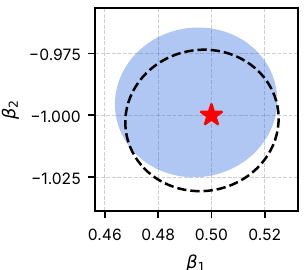}
        \caption{Projection onto $(\beta_1,\ \beta_2)$}
    \end{subfigure}
    \hfill
    \begin{subfigure}[b]{0.32\textwidth}
        \centering
        \includegraphics[width=\textwidth]{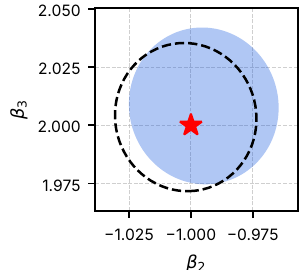}
        \caption{Projection onto $(\beta_2,\ \beta_3)$}
    \end{subfigure}
    \hfill
    \begin{subfigure}[b]{0.32\textwidth}
        \centering
        \includegraphics[width=\textwidth]{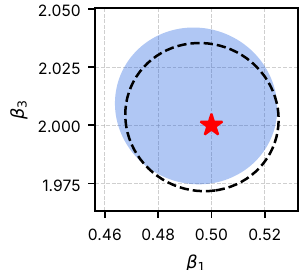}
        \caption{Projection onto $(\beta_1,\ \beta_3)$}
    \end{subfigure}
    \caption{Two-dimensional projections of the union confidence set (shaded), of the oracle confidence set that uses $\Pi_0$ (dashed), and of $\bm\beta_0$ (red star), for $n=100$, $k_\star=5$, $L=200$ and $\sigma_0=0.10$.}
    \label{fig:repro_union_three_planes}
\end{figure}

\begin{figure}[htbp]
  \centering
  \begin{subfigure}[t]{0.49\linewidth}
    \centering
    \includegraphics[width=\linewidth]{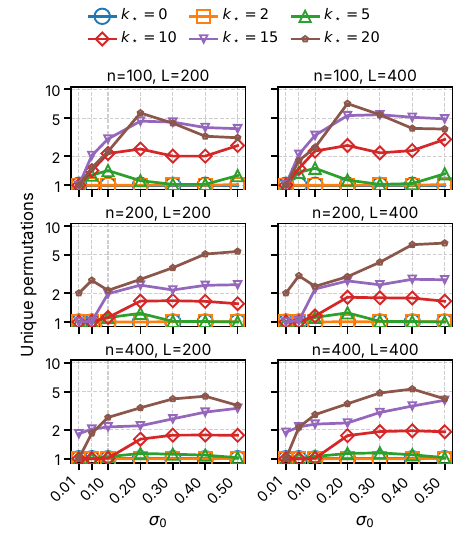}
    \caption{Unique permutations in $\mC^{(L)}$ (logarithmic scale).}
    \label{fig:size candidate}
  \end{subfigure}\hfill
  \begin{subfigure}[t]{0.49\linewidth}
    \centering
        \includegraphics[width=\linewidth]{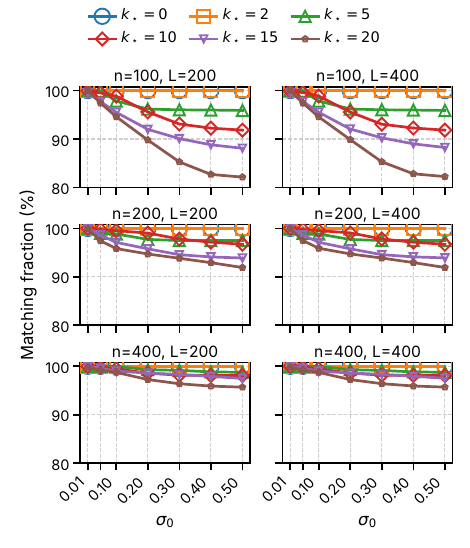}
    \caption{Matching fraction (\%).}
    \label{fig:fraction}
  \end{subfigure}

  \vspace{0.6em}

  \begin{subfigure}[t]{0.49\linewidth}
    \centering
    \includegraphics[width=\linewidth]{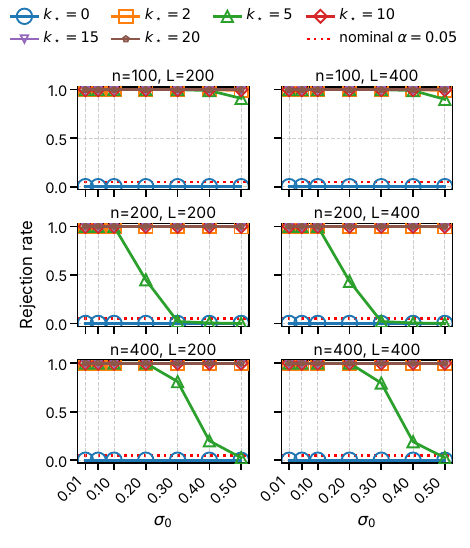}
    \caption{Rejection rate for $k_\star=0$ (size) and power for $k_\star>0$; nominal $\alpha=0.05$ shown (red dotted).}
    \label{fig:test size power}
  \end{subfigure}\hfill
  \begin{subfigure}[t]{0.49\linewidth}
    \centering
    \includegraphics[width=\linewidth]{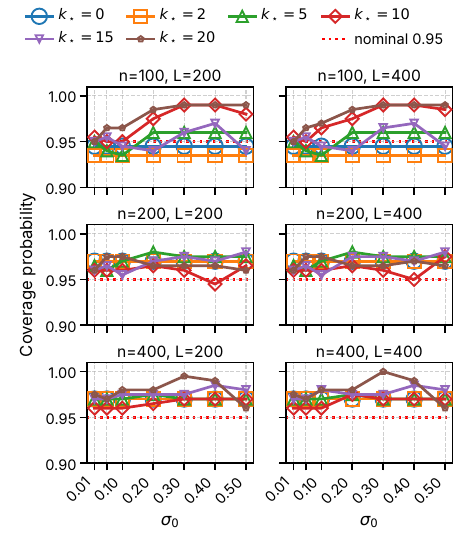}
    \caption{Coverage probability of coefficient regions; nominal level $1-\alpha=0.95$ shown (red dotted).}
    \label{fig:coverage coef}
  \end{subfigure}
  \caption{Characteristics of the candidate set $\mC^{(L)}$, of the test, and of the confidence set for $\bm\beta_0$, as functions of the noise level $\sigma_0$, averaged over $200$ replications; in each panel, the rows correspond to $n\in\{100,200,400\}$, the columns to $L\in\{200,400\}$, and the lines to $k_\star\in\{0,2,5,10,15,20\}$.  \label{fig:simulations}}
\end{figure}

\begin{table}[!t]
\centering\small
\setlength{\tabcolsep}{4pt}
\caption{Coverage of the confidence sets for $\bm{\beta}_0$ (nominal level $0.95$) and their widths relative to the oracle set for $L=400$: mean and minimum of the coverage frequencies over the $42$ configurations $(k_\star,\sigma_0)$ of each $n$, and median of the ratio of the width of the set to that of the oracle set (Width), where the width of a set is the geometric mean of the lengths of its coordinate projections.}
\label{tab:sim-benchmarks}
\begin{tabular}{@{}lrrrrrrrrr@{}}
\toprule
 & \multicolumn{3}{c}{$n=100$} & \multicolumn{3}{c}{$n=200$} & \multicolumn{3}{c}{$n=400$}\\
\cmidrule(lr){2-4}\cmidrule(lr){5-7}\cmidrule(lr){8-10}
Method & Mean & Min. & Width & Mean & Min. & Width & Mean & Min. & Width\\
\midrule
Proposed union & 0.956 & 0.935 & 1.02 & 0.968 & 0.950 & 1.00 & 0.972 & 0.960 & 1.00\\
Oracle (true $\Pi_0$) & 0.948 & 0.935 & 1.00 & 0.965 & 0.960 & 1.00 & 0.969 & 0.960 & 1.00\\
Naive least squares & 0.810 & 0.000 & 2.84 & 0.575 & 0.000 & 2.11 & 0.875 & 0.000 & 1.84\\
Robust regression & 0.921 & 0.730 & 1.14 & 0.934 & 0.810 & 1.09 & 0.958 & 0.905 & 1.06\\
Pseudo-likelihood & 0.924 & 0.910 & 0.99 & 0.959 & 0.945 & 1.01 & 0.960 & 0.945 & 1.00\\
\bottomrule
\end{tabular}
\end{table}

\subsection{Empirical application: Beijing air-quality data}\label{subsec:real data}

We illustrate the proposed finite-sample inference for shuffled regression using the Beijing Multi-Site Air Quality dataset \citep{beijing_multi-site_air_quality_501}. The data that support the findings of this study are openly available in UCI Machine Learning Repository at \texttt{10.24432/C5RK5G}. The objective is to verify that the permutation test behaves consistently with the theoretical guarantees. To this end, we construct two controlled scenarios on the real data: a baseline with no mismatch and a perturbed setting in which we artificially shuffle the response–covariate matching within short time windows. The dataset reports hourly concentrations of six major air pollutants together with meteorological covariates from twelve stations between March~2013 and February~2017.
Following the empirical design of \cite{slawski2020two}, we focus on the Nongzhanguan station and restrict the sample to the years 2016–2017.
After removing rows with missing entries, we extract a contiguous window of $n=1000$ hourly observations (from 25 September to 6 November 2016, approximately six weeks of data) and standardize each covariate to zero mean and unit variance.
The response variable $\bm{Y}$ is the concentration of PM$_{2.5}$, and the design matrix $\bm X$ contains an intercept and ten continuous covariates: PM$_{10}$, SO$_2$, NO$_2$, CO, O$_3$, temperature, dew point, pressure, precipitation, and wind speed.
All computations employ the sparsity-penalized LAP described in Section~\ref{sec:methodology} within the repro samples framework, with the tuning parameters of Section~\ref{sec:tuning}.

We analyze two complementary scenarios.
First, a sparse and local mismatch is introduced: $8\%$ of the rows of $\bm Y$ are permuted within the same day and within a $\pm 3$-hour window to emulate record-linkage or time-synchronization errors, as in \cite{slawski2020two}.
Second, no mismatch is added, so that the alignment between $\bm Y$ and $\bm X$ remains exact.  For both settings, the candidate set $\mC^{(L)}$ is built from $L=250$ repro samples, and the conditional Monte Carlo test of $H_0:\Pi_0=I_n$ is conducted at nominal level $\alpha = 0.05$ with $M=199$ conditional Monte Carlo samples.

Under the local-mismatch configuration, the candidate set expands to $|\mC^{(L)}|=4$ and $H_0$ is rejected at $\alpha = 0.05$, indicating statistically detectable misalignment.
Under the no-mismatch configuration, the procedure produces a singleton candidate set $\mC^{(L)}=\{I_n\}$ and the null $H_0$ is not rejected, confirming
Type-I error control in accordance with Proposition~\ref{prop:tuning-null}.

These results are consistent with the theoretical predictions: when the data contain no mismatches, the candidate set collapses to the identity, whereas local shuffling enlarges the set and triggers rejection through the conditional test.

Overall, the two real-data scenarios are in line with the theoretical properties established for the repro samples framework.

\begin{table}[htbp]
\centering\small
\caption{Results for the Beijing air-quality dataset (Nongzhanguan St., 2016–2017).}
\label{tab:realdata}
\begin{tabular}{lccc}
\toprule
\textbf{Scenario} & \textbf{Mismatch rate} & \textbf{Candidate set size} $|\mC^{(L)}|$ & \textbf{Test result} \\
\midrule
Local mismatch & 8\% (within $\pm 3$h) & 4& Rejected ($p=0.005$) \\
No mismatch & 0\% & 1\ ($\{I_n\}$) & Not rejected ($p=1.000$) \\

\bottomrule
\end{tabular}
\end{table}
\section{Conclusion}\label{sec:conclusion}
This paper develops a finite-sample inference framework for linear regression with sparsely permuted data. Building upon the repro sample method and a localization step, the method converts an intractable search over $\mathcal{P}_{n,k}$ into inference that is uniform over a small, data-dependent candidate set $\mathcal{C}^{(L)}$.

Our theory shows: (i) the miscoverage probability of the candidate set $\bP\{\Pi_0\notin \mathcal{C}^{(L)}\}$ decays at the rate $1/L$ as the number of repro samples grows, with an explicit dependence on $n$, $p$ and $k$ (Theorem~\ref{thm:recovery}); (ii) a conditional Monte Carlo test for the $k_0$-sparsity with finite-sample type-I error control (Theorem~\ref{thm:p-value}); (iii) confidence regions for $\bm \beta_0$ obtained by taking union over the candidate set $\mathcal{C}^{(L)}$ with finite-sample coverage (Theorem~\ref{thm:confidence}); and (iv) a computational guarantee that the proposed sparsity-penalized LAP recovers the true permutation with high probability under a separation condition, with a bound on the probability that it attains the same global minimizer as the least square problem (Theorem~\ref{thm:high-probability equivalence}). Extensions to partially permuted designs, ridge regularization and non-Gaussian errors are also discussed.

Empirically, a few hundred repro samples deliver near-nominal behavior: the test controls the type-I error, the union region is robust to mild misalignment, and computation scales as $O(n^3)$ with a tuning rule that does not require the number of mismatches.

Future research includes two main directions. First, extending the calibration beyond Gaussian and spherical errors to general sub-Gaussian or heavy-tailed settings, for which Remark~\ref{rem:compact-localization-noise} gives only partial results, would enhance robustness. Second, exploring high-dimensional extensions with structured regularization and multi-response \cite{slawski2020two} or block-permutation designs \cite{abbasi2025alternating} would broaden applicability. In addition, the framework relies on the linear model. For an unknown regression function, the permutation is not identified without restrictions on the function; under shape constraints such as monotonicity \citep{rigollet2019uncoupled,pananjady2022isotonic,slawski2024permuted}, the conditional representation \eqref{eq:conditional-representation} and the $F$-distribution of $T$ are not available, and finite-sample inference is left for future work. Overall, the results demonstrate that locally simultaneous, repro samples-based inference offers a general and computationally feasible foundation for exact finite-sample validity in irregular, combinatorial statistical models including permutations.

\section*{Acknowledgements} 
The authors are deeply grateful to Shinji Ito, Kaoru Irie and Takeru Matsuda for their helpful discussions and insightful comments, which greatly improved the clarity and scope of this work. We also thank all members of our research group for their constructive feedback during the early stages of this project. Hirofumi Ota is supported by JSPS KAKENHI (25K23133). Masaaki Imaizumi is supported by JSPS KAKENHI (24K02904), JST CREST (JPMJCR21D2),  JST FOREST (JPMJFR216I), and JST BOOST (JPMJBY24A9).

\newpage
\appendix
\numberwithin{equation}{section}
\numberwithin{figure}{section}
\numberwithin{table}{section}

\section{Additional Methodological Background}

\subsection{Repro samples method}
In this subsection, we give a brief review on the repro samples framework to obtain a performance-guaranteed confidence set on general parameters. The \textit{repro samples method} \cite{xie2022repro, xie2024repro} is a simulation-inspired general inference method, which is particularly effective for making inference on irregular inference problems \cite{wasserman2020universal}. The general setup is as follows: suppose the sample data $\bm{Y} \in \mathcal{Y}$ are generated from a model with the true parameter $\bm{\theta}_{0} \in \Theta$. We also assume that we know how to simulate $\bm{Y}$ from the statistical model indexed by $\bm{\theta} \in \Theta$ and it has the form of generative algorithmic model:
\begin{equation}\label{eq:algorithmic}
    \bm{Y} = G(\bm{\theta}, \bm{U}),
\end{equation}
where $\bm{U} = (U_1, \cdots, U_r) \in \mathcal{U} \subset \R^r$ is a random vector whose distribution is free of $\bm{\theta} \in \Theta$ and $G: \Theta \times \mathcal{U} \mapsto \mY$ is a known mapping. Let
\begin{equation}\label{eq:realization}
    \bm{Y}^{\mathrm{obs}} = G(\bm{\theta}_{0}, \bm{u}^{\mathrm{rel}})
\end{equation}
be the realized version of $\bm{Y}$ and also $\bm{u}^{\mathrm{rel}}$ be the unobserved realization of $\bm{U}$.

The general framework of repro samples method to construct a performance-guaranteed confidence set is as follows. We generate $\bm{u}^{*} \sim \bm{U}$ given $\bm{\theta}^{*}$ based on the generative algorithmic model (\ref{eq:algorithmic}), and call $\bm{Y}^{*} = G(\bm{\theta}^{*}, \bm{u}^{*})$ as the repro samples. The key idea to construct a valid confidence set is to quantify the uncertainty of $\bm{U}$ through $\bm{u}^{*}$. Let $B_\alpha \subset \mathcal{U}$ be a level-$\alpha$ Borel set for some $\alpha \in (0,1)$, such that $\bP(\bm{U} \in B_\alpha) \ge \alpha$. Then, we define
\begin{equation*}
    \Gamma_\alpha(\bm{Y}^{\mathrm{obs}}) = \left\{ \bm{\theta} : \exists \bm{u}^{*} \in B_\alpha \ \mathrm{s.t.}\ \bm{Y}^{\mathrm{obs}} = G(\bm{\theta}, \bm{u}^{*}) \right\}.
\end{equation*}
It is shown that the $\Gamma_\alpha(\bm{Y}^{\mathrm{obs}})$ is a level-$\alpha$ confidence set for $\bm{\theta}_{0}$, i.e.,
    \[
    \bP(\bm{\theta}_{0} \in \Gamma_\alpha(\bm{Y})) \ge \alpha,\ \hbox{for }\alpha \in (0,1).
    \]
In the simplest case above, we only need to quantify the uncertainty of $\bm{U}$ via a level-$\alpha$ Borel set $B_\alpha$. More generally, we can consider the case where a nuclear mapping $T: \mathcal{U} \times \Theta \mapsto \R^d$ is available. Let $B_\alpha(\bm{\theta})$ be a level-$\alpha$ Borel set such that $\bP(T(\bm{U}, \bm{\theta}) \in B_\alpha(\bm{\theta}) )\ge \alpha$. Then, the following
\begin{equation}\label{eq:confidence set2}
    \tilde{\Gamma}_\alpha(\bm{Y}^{\mathrm{obs}}) = \left\{ \bm{\theta} : \exists \bm{u}^{*} \ \mathrm{s.t.}\ \bm{Y}^{\mathrm{obs}} = G(\bm{\theta}, \bm{u}^{*}),  T(\bm{u}^*, \bm{\theta}) \in B_\alpha(\bm{\theta})\right\}
\end{equation}
is also a valid level-$\alpha$ confidence set, i.e., it holds that,
\begin{equation}\label{eq:confidence set2 validity}
    \bP(\bm{\theta}_{0} \in  \tilde{\Gamma}_\alpha(\bm{Y})) \ge \bP(T(\bm{U}, \bm{\theta}) \in B_\alpha(\bm{\theta}) )\ge \alpha,\ \hbox{for }\alpha \in (0,1).
\end{equation}

\subsection{Connection to selective/post-selection inference}\label{app:selective-inference}

The classical selective inference provides valid post–model‐selection inference by
conditioning on the event that a data–driven selection procedure chooses a particular model. For example, \cite{fithian2014optimal, lee2016exact, liu2018more} proceed statistical inference by conditioning on the realised event and deriving the exact conditional distribution of the test statistic under this restriction. The
conditional distribution is characterized as a truncated multivariate normal whose
support is a polyhedron, via the Gaussian polyhedral lemma \cite{lee2016exact}.
While the resulting procedures control coverage and type-I error conditional on the selected event, three practical difficulties are common: (i) the selection event is described by high-dimensional linear
      inequalities, so explicit formulas are rare and simulation is costly; (ii) conditional guarantees do not automatically imply unconditional
      guarantees so that aggressive selection can produce undercoverage when the conditioning is removed; (iii) conditioning discards information that would otherwise contribute to
      statistical power, especially in small samples.
      Another notable methodology for the selective inference is locally simultaneous inference \cite{zrnic2024locally},
which addresses these issues by replacing conditioning with simultaneous control over a data-dependent neighbourhood. Compared to the standard solution such as the Bonferroni-type method taking the unions across all candidate parameters, the locally simultaneous inference provides a valid uncertainty quantification over the data-depoendent localized parameters and therefore is less conservative.

While the conditional inference fixes the selected model and conditions on that event, the locally simultaneous inference keeps the selection random and enforces validity uniformly over the random set. Furthermore, the conditional methods theoretically and practically require a precise characterization of the conditional law which is often achieved by the Gaussian polyhedral lemma, while locally simultaneous methods use standard parametric (or possibly nonparametric) reference distributions and avoid any geometric description of the selection event.

From these perspectives, we adopt the locally simultaneous inference for shuffled/permuted linear regression models. Our approach first provides the valid localization scheme based on the repro samples method: Theorem \ref{thm:recovery} establishes the true permutation is included in the candidate set of permutations with probability. Within this candidate set, we test the null hypothesis for the underlying permutation structure and construct a union confidence set for the regression coefficient. Our results demonstrate that locally simultaneous inference extends naturally to problems with discrete and combinatorial parameters.
The methodology is conceptually simple, requires only classical distribution
theory, and exhibits strong finite-sample performance, thereby providing a
practical alternative to conditional selective inference.

\section{Details of Hungarian Algorithm}\label{sec:hungarian alg}

Given a cost matrix $\Omega\in\R^{n\times n}$, the linear assignment problem (LAP) is given by
\begin{equation}\label{eq:LAP-appendix}
  \hat{\Pi}\ \in\ \argmin_{\Pi\in\mP_n} \langle\Omega,\Pi\rangle
  =\argmin_{\Pi\in\mP_n}\sum_{i=1}^n\sum_{j=1}^n \Omega_{ij}\Pi_{ij}.
\end{equation}
In Algorithm~\ref{alg:scoring hungarian}, we apply this LAP with the sparsity-penalized cost matrix
$\Omega=\Omega(\bm u^*)$.

A convenient viewpoint is to consider the linear-programming relaxation
\[
\min_{\Pi\ge 0}\left\{\langle\Omega,\Pi\rangle:
\sum_{j=1}^n\Pi_{ij}=1\ \forall i,\ \sum_{i=1}^n\Pi_{ij}=1\ \forall j\right\},
\]
whose extreme points are exactly the permutation matrices. Hence the relaxation attains an optimizer in
$\mP_n$. Also, the dual problem is given by
\[
\max_{\bm a,\bm b\in\R^n}\left\{\sum_{i=1}^n a_i+\sum_{j=1}^n b_j:
a_i+b_j\le\Omega_{ij}\ \forall i,j\right\}.
\]

The Hungarian algorithm \cite{kuhn1955hungarian} solves \eqref{eq:LAP-appendix} by maintaining
(i) a dual-feasible pair $(\bm a,\bm b)$ with $a_i+b_j\le \Omega_{ij}$ for all $i,j$ and
(ii) a partial matching $\mathcal M\subseteq[n]\times[n]$ using only the \emph{tight} (equality) edges
\[
E(\bm a,\bm b) := \left\{(i,j)\in[n]\times[n]: a_i+b_j=\Omega_{ij}\right\}.
\]
For any set of rows $S\subseteq[n]$, define its tight-neighborhood
\[
N(S)\ :=\ \left\{j\in[n]:\ \exists i\in S\ \text{such that}\ (i,j)\in E(\bm a,\bm b)\right\}.
\]
At each iteration the algorithm either augments the matching along tight edges
or updates the dual variables to create at least one new tight edge, while preserving dual feasibility.
When a perfect matching on $E(\bm a,\bm b)$ is obtained, complementary slackness implies primal optimality.

\begin{algorithm}[H]

\caption{Hungarian algorithm for $\min_{\Pi\in\mP_n}\langle \Omega,\Pi\rangle$}
\label{alg:scoring hungarian}
\begin{algorithmic}[1]
\REQUIRE $\Omega\in\R^{n\times n}$
\ENSURE $\hat{\Pi}\in\mP_n$
\STATE $a_i\leftarrow \min_{j\in[n]} \Omega_{ij}$ for all $i\in[n]$; \quad
       $b_j\leftarrow 0$ for all $j\in[n]$; \quad $\mathcal M\leftarrow\emptyset$.
\WHILE{$|\mathcal M|<n$}
  \STATE Pick an unmatched row $i_0$; set $S\leftarrow\{i_0\}$; $T\leftarrow\emptyset$.
  \LOOP
     \IF{$N(S)=T$}
        \STATE $\delta\leftarrow \min\{\Omega_{ij}-a_i-b_j:\ i\in S,\ j\notin T\}$.
        \STATE $a_i\leftarrow a_i+\delta$ for all $i\in S$; \quad $b_j\leftarrow b_j-\delta$ for all $j\in T$.
     \ENDIF
     \STATE Pick a column $j\in N(S)\setminus T$.
     \IF{$j$ is unmatched in $\mathcal M$}
        \STATE Augment $\mathcal M$ along an alternating path of tight edges from $i_0$ to $j$, and leave the loop.
     \ELSE
        \STATE $S\leftarrow S\cup\{i:\ (i,j)\in\mathcal M\}$; \quad $T\leftarrow T\cup\{j\}$.
     \ENDIF
  \ENDLOOP
\ENDWHILE
\STATE Output $\hat{\Pi}$ induced by the perfect matching $\mathcal M$.
\end{algorithmic}
\end{algorithm}

The potentials $(\bm a,\bm b)$ define reduced costs $\Omega_{ij}-(a_i+b_j)\ge 0$,
and the tight-edge set $E(\bm a,\bm b)$ collects the zero reduced-cost edges.
The algorithm attempts to increase the matching size using only tight edges.
If this is not possible for the current row set $S$, the slack update with
\[
\delta=\min\{\Omega_{ij}-a_i-b_j:\ i\in S,\ j\notin T\}
\]
preserves dual feasibility and creates at least one new tight edge, allowing the process to continue.

When $N(S)=T$, every column of $T$ is matched to a row of $S$, so that $|S|=|T|+1$, the minimum defining $\delta$ is positive, and the update keeps the matched edges tight.

Here, ``augment'' means flipping matched/unmatched status along an alternating path of tight edges,
which increases $|\mathcal M|$ by one. After $n$ augmentations a perfect matching is obtained, and
complementary slackness implies that the induced permutation matrix $\hat{\Pi}$ is optimal for
\eqref{eq:LAP-appendix}.

\section{Proofs}\label{sec:proofs}
We use the notation of the main text. In addition, for $\Pi\in\mP_n$ and
$\bm u\in\R^n$, we write
\[
\bm s=\Pi_0\bm X\bm\beta_0,\qquad
F_\Pi(\bm u)=\|R_{(\Pi\bm X,\bm u)}\bm Y\|_2^2,\qquad
w_\Pi=\|R_{\Pi\bm X}\bm s\|_2,\qquad
\underline w=\{nC_{\min}(\Pi_0)\}^{1/2}
\]
for the mean of $\bm Y$ in \eqref{eq:random}, the objective of
\eqref{eq:repro_Pi} with the response $\bm Y$, the distance of $\bm s$
from $\range(\Pi\bm X)$, and a lower bound for these distances: the
definition \eqref{eq:c min} of $C_{\min}(\Pi_0)$ gives
\begin{equation}
w_\Pi^2\ge\underline w^2\max\{k_\star-d(\Pi,I_n),1\}\ge\underline w^2
\qquad\text{for every }\Pi\in\mP_{n,k}\setminus\{\Pi_0\}.
\label{eq:Cmin-basic}
\end{equation}
As in Section~\ref{sec:theory}, $\hat\Pi(\bm U,\bm U^*)$ denotes the
minimizer of $F_\Pi(\bm U^*)$ over $\Pi\in\mP_{n,k}$; ties are broken by a
fixed measurable rule.

\subsection{Supportive lemmas}
\begin{lemma}\label{lem:k-sparse permutation}
    The number of $k$-sparse permutation matrices $|\mP_{n,k}|$ is exactly calculated as follows:
    \begin{equation}
        |\mP_{n,k}| = \sum_{m = n - k}^{n} \binom{n}{m} D_{n - m},
    \end{equation}
    where $D_{n-m}$ is the number of permutations with no fixed points, computed as
    \[
D_{n-m} = (n-m)! \sum_{j=0}^{n-m} \frac{(-1)^j}{j!}.
\]
\end{lemma}
\begin{proof}
A permutation has exactly $m$ fixed points if we (i) choose the fixed-point set $T\subset[n]$ with $|T|=m$ (there are $\binom{n}{m}$ choices),
and (ii) derange the remaining $n-m$ indices (there are $D_{n-m}$ choices).
A $k$-sparse permutation satisfies $d(\Pi,I_n)\le k$, i.e., it has at least $n-k$ fixed points, so $m$ ranges from $n-k$ to $n$.
Summing over $m=n-k,\dots,n$ yields the formula.
\end{proof}

\begin{lemma}
\label{lem:sharp-pairwise}
Let $M_0$ and $M_1$ be orthogonal projection matrices of the same rank in
$\R^{n\times n}$, and let $m=\rank(M_0,M_1)-\rank(M_0)$. Suppose
$\bm s\in\range(M_0)$ and $\bm Y=\bm s+\sigma_0\bm U$ with
$\bm U\sim\Normal(0,I_n)$. Then
\begin{equation}
\bP\left\{\bm Y^\top(M_0-M_1)\bm Y\le0\right\}
\le\left(\frac43\right)^{m/2}
 \exp\left\{-\frac{\|(I_n-M_1)\bm s\|_2^2}{8\sigma_0^2}\right\}.
\label{eq:sharp-projection-bound}
\end{equation}
\end{lemma}
\begin{proof}
Put $\Delta=M_0-M_1$. A singular value decomposition of orthonormal bases of
the two ranges gives $m$ mutually orthogonal planes such that $\Delta$ vanishes
on the orthogonal complement of their sum and, in suitable orthonormal
coordinates of the $i$th plane, $M_0$ and $M_1$ are the projection
matrices onto the lines spanned by $(1,0)^\top$ and
$(\cos\theta_i,\sin\theta_i)^\top$, so that
\[
\Delta=\begin{pmatrix}\sin^2\theta_i&-\cos\theta_i\sin\theta_i\\
-\cos\theta_i\sin\theta_i&-\sin^2\theta_i\end{pmatrix},
\]
where $0<\theta_i\le\pi/2$ are the nonzero principal angles between the
two ranges. On the $i$th plane we have $\Delta^2=\sin^2\theta_i I_2$,
and $\bm s\in\range(M_0)$ has the coordinates $(\mu_i,0)$, so that
\[
\|(I_n-M_1)\bm s\|_2^2=\bm s^\top\Delta\bm s=\sum_{i=1}^m\mu_i^2\sin^2\theta_i .
\]
The matrix $I_n+\Delta/2$ is positive definite, and completing the square in
the Gaussian density gives
\begin{equation}
\bE\exp\left\{-\frac{\bm Y^\top\Delta\bm Y}{4\sigma_0^2}\right\}
=\det(I_n+\Delta/2)^{-1/2}
 \exp\left\{-\frac{\bm s^\top\Delta(I_n+\Delta/2)^{-1}\bm s}{4\sigma_0^2}\right\}.
\label{eq:sharp-mgf}
\end{equation}
Since $(I_2+\Delta/2)(I_2-\Delta/2)=(1-\sin^2\theta_i/4)I_2$ on the $i$th plane,
\begin{align*}
\det(I_n+\Delta/2)
&=\prod_{i=1}^m\left(1-\frac{\sin^2\theta_i}4\right)\ge\left(\frac34\right)^m,\\
\bm s^\top\Delta(I_n+\Delta/2)^{-1}\bm s
&=\sum_{i=1}^m\frac{\mu_i^2\sin^2\theta_i}{2(1-\sin^2\theta_i/4)}
\ge\frac12\|(I_n-M_1)\bm s\|_2^2 .
\end{align*}
On the event $\{\bm Y^\top\Delta\bm Y\le0\}$, the exponential on the left-hand
side of \eqref{eq:sharp-mgf} is at least one, and Markov's inequality
proves \eqref{eq:sharp-projection-bound}.
\end{proof}

\begin{lemma}
\label{lem:sharp-beta-gamma}
For $Z\sim\operatorname{Beta}(a,b)$ with $a,b>0$, and for $t\ge0$,
\begin{equation}
\bE e^{-tZ}\le\left(1+\frac{t}{a+b}\right)^{-a}.
\end{equation}
\end{lemma}
\begin{proof}
Let $G_1\sim\operatorname{Gamma}(a,1)$ and
$G_2\sim\operatorname{Gamma}(b,1)$ be independent, and put $T=G_1+G_2$.
Then $G_1/T\sim\operatorname{Beta}(a,b)$ is independent of
$T\sim\operatorname{Gamma}(a+b,1)$, so that we may take $Z=G_1/T$, and
$\bE(T\mid Z)=a+b$. Jensen's inequality for the conditional expectation
given $Z$ yields
\[
e^{-tZ}
=\exp\left\{-\frac{tZ \bE(T\mid Z)}{a+b}\right\}
\le\bE\left[\left.\exp\left\{-\frac{tZT}{a+b}\right\}\right|Z\right].
\]
Since $ZT=G_1$, taking expectations gives
$\bE e^{-tZ}\le\bE\exp\{-tG_1/(a+b)\}=\{1+t/(a+b)\}^{-a}$.
\end{proof}

\subsection{Localization bounds for design matrices of equal rank}
\label{sec:wxz-gaussian-bounds}
The bounds behind Theorem~\ref{thm:recovery} hold for any finite
collection of design matrices of the same rank in place of the matrices
$\Pi\bm X$. Let $\bm A_0,\bm A_1,\ldots,\bm A_N$ be
fixed $n\times p$ matrices of rank $p$, where $N\ge1$, $p\ge1$ and
$n\ge p+2$, and suppose
\begin{equation}
\bm Y=\bm s+\sigma_0\bm U,\qquad
\bm s\in\range(\bm A_0),\qquad
w_i=\|R_{\bm A_i}\bm s\|_2>0\quad(1\le i\le N),
\label{ta-g-eq-setting}
\end{equation}
where $\sigma_0>0$ and $\bm U,\bm U_1^*,\ldots,\bm U_L^*$ are independent
$\Normal(0,I_n)$ vectors. For $\bm y,\bm u\in\R^n$, let
$\hat{\bm A}(\bm y,\bm u)$ be a minimizer of
$\|R_{(\bm A,\bm u)}\bm y\|_2^2$ over
$\bm A\in\{\bm A_0,\ldots,\bm A_N\}$, with ties broken by a fixed
measurable rule, and let
$\mC^{(L)}=\{\hat{\bm A}(\bm Y,\bm U_\ell^*):\ell=1,\ldots,L\}$ be the
set of the selected matrices. The permutation problem is the case of the
matrices $\Pi\bm X$, $\Pi\in\mP_{n,k}$, with $\bm A_0=\Pi_0\bm X$ and
$\bm s=\Pi_0\bm X\bm\beta_0$: the matrix $\bm A_0$ is selected if and
only if $\Pi_0$ is, the distances $w_i$ are the $w_\Pi$, and, as in
Theorem~\ref{thm:recovery},
\begin{equation}
N=|\mP_{n,k}|-1=\sum_{h=0}^{k}N_h,\qquad
N_h=\binom nhD_h-\1\{h=k_\star\},
\end{equation}
where $D_h=h!\sum_{j=0}^h(-1)^j/j!$ is the number of permutations of $h$
points without fixed points (Lemma~\ref{lem:k-sparse permutation}).
Throughout, we put
\begin{equation}
\nu=n-p-1,\qquad\kappa=5N(24\sqrt{p+1}+3),
\label{ta-g-eq-tube-parameters}
\end{equation}
we let $I_x(a,b)$ denote the regularized incomplete beta function, that
is, the distribution function of the $\operatorname{Beta}(a,b)$
distribution, and we write
\[
q(\bm y)=\bP\{\hat{\bm A}(\bm y,\bm U^*)=\bm A_0\},\qquad
\bm U^*\sim\Normal(0,I_n),
\]
for the probability that one repro sample selects $\bm A_0$ when the
response is $\bm y$. Since the repro samples are independent of
$\bm Y$,
\begin{equation}
\bP\{\bm A_0\notin\mC^{(L)}\}=\bE\{1-q(\bm Y)\}^L.
\label{ta-g-eq-conditional-identity}
\end{equation}

\begin{theorem}
\label{ta-g-thm-tube}
In the setting \eqref{ta-g-eq-setting}, for every $L\ge1$,
\begin{equation}
\bP\{\bm A_0\notin\mC^{(L)}\}
\le\min\left\{1,\ \frac{2e^{1/(2\kappa^2)}}{L+1}
 \left\{\kappa\left(\frac{\|\bm s\|_2}{\sigma_0}+\sqrt n\right)\right\}^{\nu}
 \sum_{i=1}^N\left(\frac{\|\bm s\|_2}{w_i}\right)^{\nu}\right\}.
\label{ta-g-eq-tube-main}
\end{equation}
\end{theorem}

A repro sample selects $\bm A_0$ when it is close to
$\range(\bm A_0,\bm Y)$, a space of codimension $\nu$, and the same
number $\nu$ governs the probability that the margin $\rho(\bm Y)$
defined below is small; this gives the exponent one of $L$. The proof
rests on the following lemma, which does not involve the distribution of
the response.

\begin{lemma}
\label{lem:tube-pointwise}
In the setting \eqref{ta-g-eq-setting}, for every
$\bm y\in\R^n$ outside $\range(\bm A_0)$ and outside
$\range(\bm A_i,\bm s)$, $1\le i\le N$,
\begin{equation}
q(\bm y)\ge c_\nu\rho(\bm y)^\nu,\qquad
c_\nu=\frac{e^{-1/(2\kappa^2)}}{2^{\nu/2+1}\Gamma(\nu/2+1)\kappa^\nu},\qquad
\rho(\bm y)=\min_{1\le i\le N}
 \frac{w_i\|R_{(\bm A_i,\bm s)}\bm y\|_2}{\|\bm s\|_2\|\bm y\|_2}.
\label{ta-g-eq-conditional-tube-probability}
\end{equation}
\end{lemma}
\begin{proof}
For such a $\bm y$, the matrices $(\bm A_0,\bm y)$, $(\bm A_i,\bm y)$ and
$(\bm A_i,\bm s)$ have rank $p+1$, and $0<\rho(\bm y)\le1$. For an
$n\times p$ matrix $\bm A$ and $\bm y,\bm u\in\R^n$,
\begin{equation}
\|R_{(\bm A,\bm u)}\bm y\|_2\|R_{\bm A}\bm u\|_2
=\|R_{\bm A}\bm y\|_2\|R_{(\bm A,\bm y)}\bm u\|_2,
\label{ta-g-eq-gram-score}
\end{equation}
because the squares of both sides are equal to
$\|R_{\bm A}\bm y\|_2^2\|R_{\bm A}\bm u\|_2^2-(\bm y^\top R_{\bm A}\bm u)^2$.
With $\bm A=\bm A_i$ and $\bm u=\bm s$, this gives
\begin{equation}
\|R_{\bm A_i}\bm y\|_2\|R_{(\bm A_i,\bm y)}\bm s\|_2
=w_i\|R_{(\bm A_i,\bm s)}\bm y\|_2
\ge\|\bm s\|_2\|\bm y\|_2\rho(\bm y)
\qquad(1\le i\le N).
\label{ta-g-eq-wedge-identity}
\end{equation}

Decompose a repro sample as $\bm U^*=\bm g+\bm h$ with
$\bm g=M_{(\bm A_0,\bm y)}\bm U^*$ and $\bm h=R_{(\bm A_0,\bm y)}\bm U^*$.
The two vectors are independent, $\bm h$ is a standard Gaussian vector
in a space of dimension $\nu$, we have $\bE\|\bm g\|_2^2=p+1$, and
$\|R_{\bm A_0}\bm g\|_2$ is distributed as $|Z|$ with
$Z\sim\Normal(0,1)$, because $R_{\bm A_0}\bm g$ is the projection of
$\bm U^*$ on the line spanned by $R_{\bm A_0}\bm y$. For $1\le i\le N$,
let $\bm v_i=R_{(\bm A_i,\bm y)}\bm s/\|R_{(\bm A_i,\bm y)}\bm s\|_2$.
Since $\bm s\in\range(\bm A_0,\bm y)$, the variable $\bm v_i^\top\bm g$
is centred Gaussian with variance
\[
\|M_{(\bm A_0,\bm y)}\bm v_i\|_2^2
\ge\frac{(\bm v_i^\top\bm s)^2}{\|\bm s\|_2^2}
=\frac{\|R_{(\bm A_i,\bm y)}\bm s\|_2^2}{\|\bm s\|_2^2},
\]
and $\|R_{(\bm A_i,\bm y)}\bm g\|_2\ge|\bm v_i^\top\bm g|$. By the bound
$\bP\{|Z|<\epsilon\}\le\epsilon(2/\pi)^{1/2}$, Markov's inequality for
$\|\bm g\|_2^2$ and a union bound, the event
\begin{equation}
\|R_{\bm A_0}\bm g\|_2\ge\frac14,\qquad
\|\bm g\|_2\le3\sqrt{p+1},\qquad
\|R_{(\bm A_i,\bm y)}\bm g\|_2
 \ge\frac{\|R_{(\bm A_i,\bm y)}\bm s\|_2}{5N\|\bm s\|_2}
 \quad(1\le i\le N)
\label{ta-g-eq-good-g}
\end{equation}
has probability at least
$1-\frac14(2/\pi)^{1/2}-\frac19-\frac15(2/\pi)^{1/2}>\frac12$.

On this event, suppose in addition that
$\|\bm h\|_2\le\epsilon:=\rho(\bm y)/\kappa$, so that $\epsilon\le1/8$.
By \eqref{ta-g-eq-gram-score} with $\bm A=\bm A_0$ and $\bm u=\bm U^*$,
for which $R_{(\bm A_0,\bm y)}\bm U^*=\bm h$ and
$\|R_{\bm A_0}\bm U^*\|_2\ge\|R_{\bm A_0}\bm g\|_2-\|\bm h\|_2\ge1/8$,
\[
\|R_{(\bm A_0,\bm U^*)}\bm y\|_2
=\frac{\|R_{\bm A_0}\bm y\|_2\|\bm h\|_2}{\|R_{\bm A_0}\bm U^*\|_2}
\le8\epsilon\|\bm y\|_2 .
\]
For $1\le i\le N$, by \eqref{ta-g-eq-gram-score},
\eqref{ta-g-eq-good-g} and \eqref{ta-g-eq-wedge-identity},
\[
\|R_{(\bm A_i,\bm U^*)}\bm y\|_2
\ge\frac{\|R_{\bm A_i}\bm y\|_2
 \left(\|R_{(\bm A_i,\bm y)}\bm g\|_2-\|\bm h\|_2\right)}
 {\|\bm g\|_2+\|\bm h\|_2}
\ge\|\bm y\|_2\frac{\rho(\bm y)/(5N)-\epsilon}{3\sqrt{p+1}+\epsilon}
>8\epsilon\|\bm y\|_2,
\]
where the last inequality follows from
$\rho(\bm y)/(5N)=\epsilon(24\sqrt{p+1}+3)$ and $8\epsilon\le1$. Hence
$\hat{\bm A}(\bm y,\bm U^*)=\bm A_0$. A standard Gaussian vector in dimension
$\nu$ has norm at most $\epsilon$ with probability at least
$e^{-\epsilon^2/2}\epsilon^\nu/\{2^{\nu/2}\Gamma(\nu/2+1)\}$, the product
of the smallest value of its density on the ball and the volume of the
ball. Since $\epsilon\le1/\kappa$ and $\bm g$ and $\bm h$ are
independent, \eqref{ta-g-eq-conditional-tube-probability} follows.
\end{proof}

\begin{proof}[Proof of Theorem~\ref{ta-g-thm-tube}]
Fix $i\in\{1,\ldots,N\}$. Since $\bm s\in\range(\bm A_i,\bm s)$, the
variables $\|R_{(\bm A_i,\bm s)}\bm Y\|_2^2/\sigma_0^2$ and
$V=\|M_{(\bm A_i,\bm s)}\bm Y\|_2^2/\sigma_0^2$ are independent, the
first has the $\chi^2_\nu$ distribution, and $V$ has the noncentral
chi-square distribution with $p+1$ degrees of freedom and noncentrality
parameter $\|\bm s\|_2^2/\sigma_0^2$, which is the mixture of the
$\chi^2_{p+1+2J}$ distributions over
$J\sim\operatorname{Poisson}(\|\bm s\|_2^2/(2\sigma_0^2))$. Hence,
conditionally on $J$,
\[
\frac{\|R_{(\bm A_i,\bm s)}\bm Y\|_2^2}{\|\bm Y\|_2^2}
\sim\operatorname{Beta}\left(\frac\nu2,\frac{p+1}2+J\right).
\]
For $b\ge1$, we have $I_x(a,b)\le x^a/\{aB(a,b)\}$ for $0\le x\le1$, and
the right-hand side is at least one for $x\ge1$. Since $(p+1)/2\ge1$,
this gives, for every $z\ge0$,
\begin{align}
\bP\left\{\frac{w_i\|R_{(\bm A_i,\bm s)}\bm Y\|_2}
 {\|\bm s\|_2\|\bm Y\|_2}\le z\right\}
&\le z^\nu\left(\frac{\|\bm s\|_2}{w_i}\right)^\nu
 \bE\frac{\Gamma(n/2+J)}{\Gamma(\nu/2+1)\Gamma((p+1)/2+J)}\notag\\
&=\frac{\bE V^{\nu/2}}
 {2^{\nu/2}\Gamma(\nu/2+1)}
 \left(\frac{\|\bm s\|_2}{w_i}\right)^\nu z^\nu,
\label{ta-g-eq-rho-one-tail}
\end{align}
because the moment of order $\nu/2$ of the $\chi^2_{p+1+2j}$
distribution is $2^{\nu/2}\Gamma(n/2+j)/\Gamma((p+1)/2+j)$. By
Minkowski's inequality, the bound $\Gamma(x+a)/\Gamma(x)\le(x+a)^a$ for
$x>0$ and $a\ge0$, and $\nu\ge1$,
\[
\left(\bE V^{\nu/2}\right)^{1/\nu}
\le\frac{\|\bm s\|_2}{\sigma_0}
 +\left\{\frac{2^{\nu/2}\Gamma(n/2)}{\Gamma((p+1)/2)}\right\}^{1/\nu}
\le\frac{\|\bm s\|_2}{\sigma_0}+\sqrt n .
\]
The response $\bm Y$ lies outside the subspaces excluded in
Lemma~\ref{lem:tube-pointwise} with probability one. That lemma,
\eqref{ta-g-eq-rho-one-tail} and a union bound over $i$ therefore give,
for every $t\ge0$,
\begin{equation}
\bP\{q(\bm Y)\le t\}
\le\bP\left\{\rho(\bm Y)\le(t/c_\nu)^{1/\nu}\right\}
\le C t,
\end{equation}
where $C$ denotes the coefficient of $1/(L+1)$ in
\eqref{ta-g-eq-tube-main}. Finally, by
\eqref{ta-g-eq-conditional-identity},
\[
\bP\{\bm A_0\notin\mC^{(L)}\}
=L\int_0^1(1-t)^{L-1}\bP\{q(\bm Y)\le t\}dt
\le CL\int_0^1t(1-t)^{L-1}dt
=\frac{C}{L+1}.
\]
\end{proof}

The next proposition concerns the permutation problem. It gives the
term of Theorem~\ref{thm:recovery} that does not depend on $L$, and
it bounds the size of the candidate set.

\begin{proposition}
\label{cor:compact-small-list}
Suppose $\rank(\bm X)=p$, $n\ge p+2$ and $C_{\min}(\Pi_0)>0$. For
$0\le h\le k$, let $J_h=\min\{p,\ n-p-1,\ \max(k_\star+h-1,0)\}$ and
\begin{equation}
\underline u_h=\min\left\{1,
 \left(\frac43\right)^{J_h/2}
 \left(1+\frac{nC_{\min}(\Pi_0)\max\{k_\star-h,1\}}
 {4\sigma_0^2(n-p)}\right)^{-(n-p-1)/2}\right\}.
\end{equation}
Then, for every $L\ge1$,
\begin{align}
\bP\{\Pi_0\notin\mC^{(L)}\}
 &\le\min\left\{1,\ \sum_{h=0}^kN_h\underline u_h\right\},\notag\\
\bE|\mC^{(L)}|
 &\le\min\left\{L,\ N+1,\
 1+\sum_{h=0}^kN_h\{1-(1-\underline u_h)^L\}\right\},\notag\\
\bP\{|\mC^{(L)}|>1\}
 &\le\min\left\{1,\ L\sum_{h=0}^kN_h\underline u_h\right\}.
\end{align}
Since $J_h\le p$, we have $\underline u_h\le u_h$ for the quantity $u_h$
of \eqref{eq:delta_L}, and the three bounds also hold with $u_h$ in
place of $\underline u_h$; in particular,
$\bE|\mC^{(L)}|\le1+L\sum_{h=0}^kN_hu_h$.
\end{proposition}
\begin{proof}
Let $\Pi\in\mP_{n,k}\setminus\{\Pi_0\}$ with $d(\Pi,I_n)=h$, and put
\[
j_\Pi=\min\{\rank(\Pi_0\bm X,\Pi\bm X)-p,\ n-p-1\}.
\]
Fix $\bm u\in\R^n$ such that $(\Pi_0\bm X,\bm u)$ and $(\Pi\bm X,\bm u)$
have rank $p+1$, which is the case for almost every $\bm u$. We apply
Lemma~\ref{lem:sharp-pairwise} to the projection matrices
$M_0=M_{(\Pi_0\bm X,\bm u)}$ and $M_1=M_{(\Pi\bm X,\bm u)}$, for which
$\bm s\in\range(M_0)$,
$F_\Pi(\bm u)-F_{\Pi_0}(\bm u)=\bm Y^\top(M_0-M_1)\bm Y$ and
$m=\rank(\Pi_0\bm X,\Pi\bm X,\bm u)-(p+1)\le j_\Pi$. Since
$\hat\Pi(\bm U,\bm u)=\Pi$ implies $F_\Pi(\bm u)\le F_{\Pi_0}(\bm u)$, the
lemma gives
\begin{equation}
\bP\{\hat\Pi(\bm U,\bm u)=\Pi\}
\le\left(\frac43\right)^{j_\Pi/2}
 \exp\left\{-\frac{\|R_{(\Pi\bm X,\bm u)}\bm s\|_2^2}{8\sigma_0^2}\right\}.
\label{eq:sharp-fixed-u}
\end{equation}
For a repro sample $\bm U^*$, the vector $R_{\Pi\bm X}\bm U^*$ is
standard Gaussian in $\range(R_{\Pi\bm X})$, a space of dimension $n-p$,
so that
\begin{equation}
\frac{\|R_{(\Pi\bm X,\bm U^*)}\bm s\|_2^2}{w_\Pi^2}
=1-\frac{(\bm s^\top R_{\Pi\bm X}\bm U^*)^2}
 {w_\Pi^2\|R_{\Pi\bm X}\bm U^*\|_2^2}
\sim\operatorname{Beta}\left(\frac{n-p-1}2,\frac12\right).
\label{eq:sharp-angle}
\end{equation}
Taking the expectation of \eqref{eq:sharp-fixed-u} with respect to
$\bm U^*$ and applying Lemma~\ref{lem:sharp-beta-gamma} with
$a=(n-p-1)/2$, $b=1/2$ and $t=w_\Pi^2/(8\sigma_0^2)$, we obtain
\begin{equation}
\bP\{\hat\Pi(\bm U,\bm U^*)=\Pi\}
\le\left(\frac43\right)^{j_\Pi/2}
 \left(1+\frac{w_\Pi^2}{4\sigma_0^2(n-p)}\right)^{-(n-p-1)/2}.
\label{eq:sharp-one-draw}
\end{equation}
The relative permutation $P=\Pi^\top\Pi_0$ moves at least two and at most
$k_\star+h$ rows. Since multiplication by $\Pi^\top$ preserves the rank,
\[
\rank(\Pi_0\bm X,\Pi\bm X)-p=\rank(P\bm X,\bm X)-p
\le\rank\{(P-I_n)\bm X\}\le\min\{p,\rank(P-I_n)\},
\]
and $\rank(P-I_n)$ equals the number of rows moved by $P$ minus the
number of its nontrivial cycles, which is at most $k_\star+h-1$. Hence
$j_\Pi\le J_h$. The right-hand side of \eqref{eq:sharp-one-draw} is
increasing in $j_\Pi$ and decreasing in $w_\Pi$, so that, by
\eqref{eq:Cmin-basic},
\begin{equation}
\bP\{\hat\Pi(\bm U,\bm U^*)=\Pi\}\le\underline u_h .
\label{eq:sharp-one-draw-h}
\end{equation}

If $\Pi_0\notin\mC^{(L)}$, then the first repro sample selects a wrong
permutation, and a union bound over the $N_h$ wrong permutations with $h$
mismatches, $0\le h\le k$, gives the first bound. For the second bound,
let $q_\Pi(\bm U)=\bP\{\hat\Pi(\bm U,\bm U^*)=\Pi\mid\bm U\}$.
Conditionally on $\bm U$, the repro samples are independent, so that
$\Pi$ belongs to $\mC^{(L)}$ with probability $1-\{1-q_\Pi(\bm U)\}^L$.
The function $x\mapsto1-(1-x)^L$ is increasing and concave on $[0,1]$,
and Jensen's inequality and \eqref{eq:sharp-one-draw-h} give
\[
\bP\{\Pi\in\mC^{(L)}\}\le1-\{1-\bE q_\Pi(\bm U)\}^L
\le1-(1-\underline u_h)^L .
\]
Summing over the wrong permutations, adding at most one for $\Pi_0$, and
using $|\mC^{(L)}|\le\min\{L,N+1\}$ gives the second bound. If
$|\mC^{(L)}|>1$, then some repro sample selects a wrong permutation, and
a union bound over the $L$ repro samples gives the third bound. The last
statement follows from $1-(1-x)^L\le Lx$.
\end{proof}

The following bound uses a repro sample aligned with the realized noise,
as in \cite{wang2022finite} and in the original submission, and is the
basis of the extensions in Appendix~\ref{section:appendix theory}. For
$0<\gamma\le1/4$, let
\begin{equation}
t_\gamma=\gamma\log(e/\gamma),\qquad
a_\gamma=I_{\gamma^2}\left(\frac{n-1}{2},\frac12\right),\qquad
b_\gamma=I_{t_\gamma}\left(\frac{n-p-1}{2},\frac12\right).
\end{equation}

\begin{theorem}
\label{thm:recovery-threshold}
Suppose $\rank(\bm X)=p$, $n>p+1$, $\Pi_0\in\mP_{n,k}$ and
$C_{\min}(\Pi_0)>0$. Then, for every $L\ge1$ and $0<\gamma\le1/4$,
\begin{equation}
\bP\{\Pi_0\notin\mC^{(L)}\}
\le(1-a_\gamma)^{L}+Nb_\gamma
+\bP\left\{\|\bm U\|_2>\frac{\underline w\log(e/\gamma)}{4\sigma_0}\right\}.
\label{eq:explicit-delta}
\end{equation}
\end{theorem}

The three terms bound the probabilities that no repro sample is aligned
with the realized noise, that $\|R_{(\Pi\bm X,\bm U^*)}\bm s\|_2$ is
small compared with $w_\Pi$ for an aligned repro sample $\bm U^*$ and
some wrong permutation $\Pi$, and that the realized noise is large;
here $\|\bm U\|_2^2\sim\chi^2_n$. Moreover,
$a_\gamma\ge\gamma^{n-1}/(n-1)$: for $a=(n-1)/2\ge1$, the density of the
$\operatorname{Beta}(a,1/2)$ distribution at $x$ is at least
$x^{a-1}/B(a,1/2)$, and $B(a,1/2)\le2$. Hence the first term is at most
$\{1-\gamma^{n-1}/(n-1)\}^L$, which is the bound of
\cite{wang2022finite} used in the original submission.

\begin{proof}[Proof of Theorem~\ref{thm:recovery-threshold}]
For a repro sample $\bm U^*$ and a fixed vector $\bm u\ne\bm 0$, the
argument for \eqref{eq:sharp-angle}, with $R_{\Pi\bm X}$ replaced by
$I_n$, gives
$\|R_{\bm U^*}\bm u\|_2^2/\|\bm u\|_2^2\sim\operatorname{Beta}((n-1)/2,1/2)$.
Consequently, for every $\bm u\ne\bm 0$ and every $\Pi\ne\Pi_0$,
\begin{align}
\bP\{\|R_{\bm U^*}\bm u\|_2<\gamma\|\bm u\|_2\}&=a_\gamma,\notag\\
\bP\{\|R_{(\Pi\bm X,\bm U^*)}\bm s\|_2^2\le t_\gamma w_\Pi^2\}&=b_\gamma,
 \label{eq:app-wrong-cap}
\end{align}
where the second identity is \eqref{eq:sharp-angle}.

Let $A_\ell=\{\|R_{\bm U_\ell^*}\bm U\|_2<\gamma\|\bm U\|_2\}$, and let
$\tau$ be the smallest $\ell\le L$ for which $A_\ell$ occurs, with
$\tau=\infty$ if there is none. Conditionally on $\bm U$, the events
$A_1,\ldots,A_L$ are independent with probability $a_\gamma$. Hence
$\bP(\tau=\infty)=(1-a_\gamma)^L$, and, for every event $B$ determined by
$(\bm U,\bm U_\ell^*)$,
\begin{equation}
\bP(\{\tau=\ell\}\cap B)=(1-a_\gamma)^{\ell-1}\bP(A_\ell\cap B).
\label{eq:app-first-hit}
\end{equation}
Moreover, the distribution of $(\bm U,\bm U_\ell^*)$ and the event
$A_\ell$ are invariant under a simultaneous rotation of the two vectors,
so that, conditionally on $A_\ell$, the direction of $\bm U_\ell^*$ is
uniformly distributed on the unit sphere; and $A_\ell$ is independent of
$\|\bm U\|_2$.

Let $B_\ell$ be the event that
$\|R_{(\Pi\bm X,\bm U_\ell^*)}\bm s\|_2^2\le t_\gamma w_\Pi^2$ for some
$\Pi\in\mP_{n,k}\setminus\{\Pi_0\}$, so that
$\bP(B_\ell\mid A_\ell)\le Nb_\gamma$ by \eqref{eq:app-wrong-cap}, and
put $c=\sigma_0\gamma\|\bm U\|_2$. On $A_\ell$, we have
$\|R_{(\Pi\bm X,\bm U_\ell^*)}\bm U\|_2\le\|R_{\bm U_\ell^*}\bm U\|_2
<\gamma\|\bm U\|_2$ for every $\Pi$, so that
$F_{\Pi_0}(\bm U_\ell^*)\le c^2$. On $A_\ell\cap B_\ell^c$, for every
$\Pi\in\mP_{n,k}\setminus\{\Pi_0\}$,
\begin{equation}
F_\Pi(\bm U_\ell^*)-F_{\Pi_0}(\bm U_\ell^*)
\ge\|R_{(\Pi\bm X,\bm U_\ell^*)}\bm s\|_2^2
 -2c\|R_{(\Pi\bm X,\bm U_\ell^*)}\bm s\|_2-c^2
>t_\gamma w_\Pi^2-2cw_\Pi-c^2 .
\label{eq:app-separation}
\end{equation}
If $c\le t_\gamma\underline w/4$, then $c\le t_\gamma w_\Pi/4$ by
\eqref{eq:Cmin-basic}, and the right-hand side is at least
$w_\Pi^2(t_\gamma/2-t_\gamma^2/16)>0$, so that
$\hat\Pi(\bm U,\bm U_\ell^*)=\Pi_0$. Since
$t_\gamma/\gamma=\log(e/\gamma)$, the event $\{c>t_\gamma\underline w/4\}$
is the event $E$ of the third term of \eqref{eq:explicit-delta}.
Therefore $\{\tau=\ell\}\cap\{\Pi_0\notin\mC^{(L)}\}$ is contained in
$\{\tau=\ell\}\cap(B_\ell\cup E)$, and \eqref{eq:app-first-hit} gives
\[
\bP\{\Pi_0\notin\mC^{(L)}\}
\le(1-a_\gamma)^L
 +\sum_{\ell=1}^L(1-a_\gamma)^{\ell-1}a_\gamma\{Nb_\gamma+\bP(E)\}
\le(1-a_\gamma)^L+Nb_\gamma+\bP(E).
\]
\end{proof}

\section{Proofs of main results}
\subsection{Proof of Proposition \ref{prop:identification}}\label{proof:prop identification}
\begin{proof}
Since $\bm Y^{\obs}=\Pi_0\bm X\bm\beta_0+\sigma_0\bm u^{\rel}$, we have
\[
\left\|(I_n-M_{(\Pi_0\bm X,\bm u^{\rel})})\bm Y^{\obs}\right\|_2^2=0,
\]
so the minimum value in \eqref{eq:repro_ideal_Pi} is $0$.

Let $\tilde\Pi\in\mP_{n,k}$ satisfy
$\|(I_n-M_{(\tilde\Pi\bm X,\bm u^{\rel})})\bm Y^{\obs}\|_2^2=0$.
Then, it holds that $\bm Y^{\obs}\in\range([\tilde\Pi\bm X,\bm u^{\rel}])$ so there exist
$(\tilde{\bm\beta},\tilde\sigma)$ such that
\[
\bm Y^{\obs}=\tilde\Pi\bm X\tilde{\bm\beta}+\tilde\sigma\bm u^{\rel}.
\]
Subtracting the true representation gives
\begin{equation}\label{eq:oracle-basic-eq}
\Pi_0\bm X\bm\beta_0-\tilde\Pi\bm X\tilde{\bm\beta}=(\tilde\sigma-\sigma_0)\bm u^{\rel}.
\end{equation}

Define $P:=\Pi_0^\top\tilde\Pi$ and $\tilde{\bm u}:=\Pi_0^\top\bm u^{\rel}$.
Multiplying \eqref{eq:oracle-basic-eq} by $\Pi_0^\top$ yields
\begin{equation}\label{eq:oracle-permuted}
\bm X\bm\beta_0- P\bm X\tilde{\bm\beta}=(\tilde\sigma-\sigma_0)\tilde{\bm u}.
\end{equation}
Note that $P\in\mP_{n,2k}$ , hence its fixed-point set
$T:=\{i\in[n]:P(i)=i\}$ satisfies $|T|\ge n-2k\ge p$.

Restrict \eqref{eq:oracle-permuted} to the coordinates in $T$. On the set $T$, we have
$(P\bm X\tilde{\bm\beta})_T=\bm X_T\tilde{\bm\beta}$, so it holds that
\begin{equation}\label{eq:oracle-on-T}
\bm X_T(\bm\beta_0-\tilde{\bm\beta})=(\tilde\sigma-\sigma_0)\tilde{\bm u}_T.
\end{equation}

We now show that $\tilde\sigma=\sigma_0$ and $\tilde{\bm\beta}=\bm\beta_0$ almost surely (conditional on $\bm X$).
Fix any non-identity $P\in\mP_{n,2k}$ with fixed-point set $T$.
Under Assumption~\ref{ass:generic-design}, $\rank(\bm X_T)=p$ almost surely.
Moreover, since $\bm\beta_0\ne\bm 0$ is fixed, each equality $(\bm x_i-\bm x_j)^\top\bm\beta_0=0$ with $i\ne j$ defines a null set, so that the entries of $\bm X\bm\beta_0$ are pairwise distinct almost surely.

If $|T|>p$, then $\range(\bm X_T)$ is a proper $p$-dimensional subspace of $\R^{|T|}$.
Since $\tilde{\bm u}_T\sim\Normal(0,I_{|T|})$ has a Lebesgue density, we have
$\bP(\tilde{\bm u}_T\in\range(\bm X_T))=0$.
Therefore \eqref{eq:oracle-on-T} cannot hold with $\tilde\sigma\ne\sigma_0$; hence $\tilde\sigma=\sigma_0$ and then
$\tilde{\bm\beta}=\bm\beta_0$.

If $|T|=p$, then $\bm X_T$ is square and invertible almost surely, and \eqref{eq:oracle-on-T} gives
\[
\tilde{\bm\beta}=\bm\beta_0-(\tilde\sigma-\sigma_0)\bm X_T^{-1}\tilde{\bm u}_T.
\]
Plugging this into \eqref{eq:oracle-permuted} on any index $i\notin T$ yields an equation of the form
\[
(\tilde\sigma-\sigma_0)\left\{\tilde u_i-a_i^\top \tilde{\bm u}_T\right\}=b_i,
\]
where $a_i^\top=(P\bm X)_{i\cdot}\bm X_T^{-1}$ and $b_i=(\bm X\bm\beta_0)_i-(P\bm X\bm\beta_0)_i$ are nonrandom given $\bm X$ and $(P,\bm\beta_0)$.
For $i\notin T$ we have $P(i)\ne i$, so that $b_i\ne0$ because the entries of $\bm X\bm\beta_0$ are pairwise distinct; hence $\tilde\sigma\ne\sigma_0$.
Since $P\neq I_n$ implies $|[n]\setminus T|\ge 2$, we can take two distinct indices $i,j\notin T$, and eliminating $\tilde\sigma-\sigma_0$ gives
\[
b_j\tilde u_i-b_i\tilde u_j-(b_ja_i-b_ia_j)^\top\tilde{\bm u}_T=0.
\]
Conditional on $\bm X$ and $\tilde{\bm u}_T$, the left side is Gaussian with variance $b_i^2+b_j^2>0$, so this event has probability zero.
Hence, almost surely, no non-identity $P$ with $|T|=p$ satisfies \eqref{eq:oracle-permuted}.

Taking a union bound over the finite set of possible $P\in\mP_{n,2k}$ yields that,
\[
\tilde\sigma=\sigma_0
\quad\text{and}\quad
P\bm X\bm\beta_0=\bm X\bm\beta_0
\]
almost surely. Finally, since the entries of $\bm X\bm\beta_0$ are pairwise distinct,
the only permutation fixing $\bm X\bm\beta_0$ is the identity, so we have $P=I_n$ and $\tilde\Pi=\Pi_0$.
This proves uniqueness of $\Pi_0$ in \eqref{eq:repro_ideal_Pi}.
For $P=I_n$, the set $T=[n]$ has $n>p$ elements under the assumption of Proposition~\ref{prop:identification_joint}, and the argument for $|T|>p$ gives $\tilde\sigma=\sigma_0$ and $\tilde{\bm\beta}=\bm\beta_0$. Since the minimum in \eqref{eq:repro_ideal} is $0$ and is attained at $(\Pi_0,\bm\beta_0,\sigma_0)$, this proves Proposition~\ref{prop:identification_joint}.
\end{proof}

\subsection{Proof of Theorem \ref{thm:recovery}}\label{app:localization-proof}
\begin{proof}
If $\mP_{n,k}=\{\Pi_0\}$, then $\mC^{(L)}=\{\Pi_0\}$. Otherwise, we apply
Theorem~\ref{ta-g-thm-tube} to the matrices $\Pi\bm X$,
$\Pi\in\mP_{n,k}$, with $\bm A_0=\Pi_0\bm X$ and
$\bm s=\Pi_0\bm X\bm\beta_0$. These matrices have rank $p$, we have
$n\ge p+2$, and $w_\Pi\ge\underline w>0$ for every
$\Pi\in\mP_{n,k}\setminus\{\Pi_0\}$ by \eqref{eq:Cmin-basic}, so that
\[
\sum_{\Pi\ne\Pi_0}\left(\frac{\|\bm s\|_2}{w_\Pi}\right)^{n-p-1}
\le N\left(\frac{\|\bm s\|_2^2}{nC_{\min}(\Pi_0)}\right)^{(n-p-1)/2},
\]
and the bound \eqref{ta-g-eq-tube-main} is at most $C_{\rm loc}/(L+1)$
with the constant $C_{\rm loc}$ of \eqref{eq:delta_L}. The bound
$\sum_{h=0}^kN_hu_h$ is given by
Proposition~\ref{cor:compact-small-list}.
\end{proof}

\subsection{Proof of Theorem \ref{thm:p-value}}
We prove Proposition~\ref{thm:p-value_without_delta} first, and then conclude Theorem~\ref{thm:p-value} by adding the localization error.

\begin{proposition}[Type-I error control for the sparsity test]\label{thm:p-value_without_delta}
For any $k_0\le k$ and $\theta_0=(\Pi_0,\bm\beta_0,\sigma_0)$ with $\Pi_0\in\mP_{n,k_0}$, we have
\begin{equation}\label{eq:conditional-validity-k0}
\bP_{\left(\mathcal{U}^L, \bm{U}\right)}\left(
D(\bm Y_{\theta_0})\notin \hat{B}^{(k_0)}_{1-\alpha}\ \text{ and }\ \Pi_0\in \mC^{(L)}
\mid
S_{1,\Pi_0}(\bm Y_{\theta_0}),S_{2,\Pi_0}(\bm Y_{\theta_0})
\right)\le \alpha.
\end{equation}
\end{proposition}

\begin{proof}
Fix $\theta_0=(\Pi_0,\bm\beta_0,\sigma_0)$ with $\Pi_0\in\mP_{n,k_0}$ and write
\[
s_{1,0}:=S_{1,\Pi_0}(\bm Y_{\theta_0}),
\qquad
s_{2,0}:=S_{2,\Pi_0}(\bm Y_{\theta_0}).
\]
By the conditional representation \eqref{eq:conditional-representation} (with $\Pi=\Pi_0$), conditionally on $(s_{1,0},s_{2,0})$ we have
\[
\bm Y_{\theta_0}\mid (s_{1,0},s_{2,0})
\ \overset{d}{=}\
s_{1,0}
+
s_{2,0}\frac{(I_n-M_{\Pi_0\bm X})\widetilde{\bm U}}{\|(I_n-M_{\Pi_0\bm X})\widetilde{\bm U}\|_2},
\]
where $\widetilde{\bm U}\sim\Normal(0,I_n)$ is independent of $(s_{1,0},s_{2,0})$.

On the event $\{\Pi_0\in \mC^{(L)}\}$, by definition of the localized null set (see Section~\ref{subsubsection:testing permutation}) we have
$\Pi_0\in\hat{\mP}_{n,k_0}$, and hence the composite critical value in \eqref{eq:critical-value-composite} satisfies
\[
\widehat c_{1-\alpha}^{(k_0)}
=\max_{\Pi\in \hat{\mP}_{n,k_0}}\widehat c_{1-\alpha}(\Pi)
\ \ge\
\widehat c_{1-\alpha}(\Pi_0),
\]
where $\widehat c_{1-\alpha}(\Pi_0)$ is the conditional $(1-\alpha)$-quantile of $D(\bm Y)$ given
$\left(S_{1,\Pi_0}(\bm Y),S_{2,\Pi_0}(\bm Y)\right)=(s_{1,0},s_{2,0})$
(as defined in Section~\ref{subsubsection:testing permutation}).

Therefore, we have
\begin{align*}
&\bP_{\left(\mathcal{U}^L, \bm{U}\right)}\left(
D(\bm Y_{\theta_0})\notin \hat B^{(k_0)}_{1-\alpha}\ \text{ and }\ \Pi_0\in \mC^{(L)}
\mid s_{1,0},s_{2,0}
\right)\\
&\qquad\le
\bP_{\left(\mathcal{U}^L, \bm{U}\right)}\left(
D(\bm Y_{\theta_0})>
\widehat c_{1-\alpha}(\Pi_0)
\mid s_{1,0},s_{2,0}
\right)
 \le \alpha,
\end{align*}
where the last inequality follows from the defining property of the conditional $(1-\alpha)$-quantile.
\end{proof}

\begin{proof}[Proof of Theorem~\ref{thm:p-value}]
Let $\theta_0=(\Pi_0,\bm\beta_0,\sigma_0)$ with $\Pi_0\in\mP_{n,k_0}$.
By a union bound, we have
\begin{align*}
\bP_{\left(\mathcal{U}^L,\bm U\right)}\left(D(\bm Y_{\theta_0})\notin \hat{B}^{(k_0)}_{1-\alpha}\right)
&\le
\bP_{\left(\mathcal{U}^L,\bm U\right)}\left(D(\bm Y_{\theta_0})\notin \hat{B}^{(k_0)}_{1-\alpha}\ \text{and}\ \Pi_0\in\mC^{(L)}\right)
+
\bP_{\left(\mathcal{U}^L,\bm U\right)}\left(\Pi_0\notin\mC^{(L)}\right)\\
&\le \alpha+\delta_L,
\end{align*}
where the second inequality uses Proposition~\ref{thm:p-value_without_delta} for the first term and the localization guarantee
$\bP(\Pi_0\notin\mC^{(L)})\le \delta_L$ for the second term (see Theorem~\ref{thm:recovery}).
\end{proof}

\subsection{Proof of Theorem \ref{thm:confidence}}
\begin{proof}
By definition, the confidence set is given by
\[
\Gamma_{\alpha}^{\mathrm{coef}}(\bm{Y}^{\obs})
= \bigcup_{\Pi\in \mC^{(L)}}\left\{\bm{\beta}\in \R^p:\ T(\bm{Y}^{\obs},(\Pi,\bm{\beta}))\le F_{p,n-p}^{-1}(1-\alpha)\right\}.
\]

Hence, on the event $\{\Pi_0\in \mC^{(L)}\}$ we have
\[
\{\bm{\beta}_0\notin \Gamma_{\alpha}^{\mathrm{coef}}(\bm{Y}^{\obs})\}\cap\{\Pi_0\in \mC^{(L)}\}
\subseteq
\left\{T(\bm{Y}^{\obs},(\Pi_0,\bm{\beta}_0))> F_{p,n-p}^{-1}(1-\alpha)\right\}.
\]

Under $\bm{Y}^{\obs}=\Pi_0\bm X\bm\beta_0+\sigma_0\bm U$ with $\bm U\sim \Normal(0,I_n)$ and $\Pi_0\bm X$ having rank $p$, we have
$T(\bm{Y}^{\obs},(\Pi_0,\bm{\beta}_0))\sim F_{p,n-p}$.
Therefore, it holds that
\begin{align*}
\bP_{\left(\mathcal{U}^L, \bm{U}\right)}\left(\bm{\beta}_0\notin \Gamma_{\alpha}^{\mathrm{coef}}(\bm{Y}^{\obs})\right)
&\le
\bP_{\bm U}\left(T(\bm{Y}^{\obs},(\Pi_0,\bm{\beta}_0))> F_{p,n-p}^{-1}(1-\alpha)\right)
+
\bP_{\left(\mathcal{U}^L, \bm{U}\right)}\left(\Pi_0\notin \mC^{(L)}\right)\\
&\le \alpha+\delta_L,
\end{align*}

which is equivalent to the claimed coverage lower bound
$\bP_{\left(\mathcal{U}^L, \bm{U}\right)}\left(\bm{\beta}_0\in \Gamma_{\alpha}^{\mathrm{coef}}(\bm{Y}^{\obs})\right)\ge 1-\alpha-\delta_L$.
\end{proof}

\subsection{Proof of Theorem \ref{thm:high-probability equivalence}}\label{app:lap-recovery-proof}
We identify a permutation matrix $\Pi$ with the map $\pi$ for which
$\Pi_{i\pi(i)}=1$, and write $\pi_0$ for the map of $\Pi_0$. For
$\bm y,\bm m\in\R^n$ and $\lambda_1\ge0$, the objective of the LAP
\eqref{def:tilde_Pi} with the costs
$(y_i-m_j)^2+\lambda_1\1\{j\ne i\}$ is
\[
 C(\pi)=\sum_{i=1}^n(y_i-m_{\pi(i)})^2
              +\lambda_1|\{i:\pi(i)\ne i\}| .
\]
The following proposition compares these costs row by row.

\begin{proposition}
\label{prop:lap-rowwise}
Let $\bm y,\bm m,\bm\mu\in\R^n$, let $\pi_0$ be a permutation of
$\{1,\ldots,n\}$ with the set of moved indices
$\mathcal M=\{i:\pi_0(i)\ne i\}$, and put
\[
 \rho=\max_{1\le i\le n}|y_i-\mu_{\pi_0(i)}|
     +\max_{1\le j\le n}|m_j-\mu_j|,\qquad
 g_\star=\min\{|\mu_a-\mu_b|:a,b\in\mathcal M,\ a\ne b\},
\]
with $g_\star=+\infty$ if $\mathcal M=\emptyset$. If
\begin{equation}
 g_\star>2\rho\qquad\text{and}\qquad
 2\rho^2<\lambda_1<g_\star(g_\star-2\rho),
 \label{eq:lap-rowwise-window}
\end{equation}
then $\pi_0$ is the unique minimizer of $C$ over all permutations of
$\{1,\ldots,n\}$.
\end{proposition}
\begin{proof}
For all $i$ and $j$,
\begin{equation}
 \bigl|(y_i-m_j)-(\mu_{\pi_0(i)}-\mu_j)\bigr|\le\rho,
 \label{eq:lap-rowwise-basic}
\end{equation}
and in particular $|y_i-m_{\pi_0(i)}|\le\rho$. If
$\mathcal M=\emptyset$, then $\pi_0$ is the identity, and
$C(\pi)-C(\pi_0)\ge|\{i:\pi(i)\ne i\}|(\lambda_1-\rho^2)>0$ for every
$\pi\ne\pi_0$. We therefore assume $\mathcal M\ne\emptyset$. Let
$\pi\ne\pi_0$, put $\mathcal D=\{i:\pi(i)\ne\pi_0(i)\}$, and write
$C(\pi)-C(\pi_0)=\sum_{i\in\mathcal D}\Delta_i$ with
\[
 \Delta_i=(y_i-m_{\pi(i)})^2-(y_i-m_{\pi_0(i)})^2
   +\lambda_1\bigl[\1\{\pi(i)\ne i\}-\1\{\pi_0(i)\ne i\}\bigr].
\]
We bound $\Delta_i$ on four sets that partition $\mathcal D$.

(a) $\mathcal D_1=\{i\in\mathcal D:\pi_0(i)=i\}$. Here $\pi(i)\ne i$, and
$\Delta_i\ge\lambda_1-(y_i-m_i)^2\ge\lambda_1-\rho^2$.

(b) $\mathcal D_2=\{i\in\mathcal D:\pi_0(i)\ne i,\ \pi(i)=i\}$. The
indices $i$ and $\pi_0(i)$ are distinct elements of $\mathcal M$, so
that $|\mu_{\pi_0(i)}-\mu_i|\ge g_\star$, and
\eqref{eq:lap-rowwise-basic} gives $|y_i-m_i|\ge g_\star-\rho>0$. Hence
$\Delta_i\ge(g_\star-\rho)^2-\rho^2-\lambda_1
 =g_\star(g_\star-2\rho)-\lambda_1$.

(c) $\mathcal D_3=\{i\in\mathcal D:\pi_0(i)\ne i,\ \pi(i)\ne i,\
\pi(i)\in\mathcal M\}$. The indices $\pi(i)$ and $\pi_0(i)$ are distinct
elements of $\mathcal M$, the penalty terms cancel, and
$\Delta_i\ge(g_\star-\rho)^2-\rho^2=g_\star(g_\star-2\rho)$.

(d) $\mathcal D_4=\{i\in\mathcal D:\pi_0(i)\ne i,\ \pi(i)\ne i,\
\pi(i)\notin\mathcal M\}$. The penalty terms cancel, and
$\Delta_i\ge-\rho^2$.

The map $i\mapsto\pi(i)$ is injective from $\mathcal D_4$ into
$\mathcal D_1$: for $i\in\mathcal D_4$, the index $j=\pi(i)$ satisfies
$\pi_0(j)=j$ because $j\notin\mathcal M$, and $\pi(j)\ne j$ because
$\pi(i)=j$ with $i\ne j$. Therefore $|\mathcal D_4|\le|\mathcal D_1|$,
and
\[
 C(\pi)-C(\pi_0)
 \ge|\mathcal D_1|(\lambda_1-2\rho^2)
   +|\mathcal D_2|\{g_\star(g_\star-2\rho)-\lambda_1\}
   +|\mathcal D_3|g_\star(g_\star-2\rho).
\]
Under \eqref{eq:lap-rowwise-window} the three coefficients are positive.
Since $\mathcal D$ is nonempty and $|\mathcal D_4|\le|\mathcal D_1|$, at
least one of $\mathcal D_1$, $\mathcal D_2$ and $\mathcal D_3$ is
nonempty, and the difference is positive.
\end{proof}

\begin{proof}[Proof of Theorem~\ref{thm:high-probability equivalence}]
Fix $\bm u\in\mathcal G_\xi$, and let $\bm m=M_{(\bm X,\bm u)}\bm Y$ be
the vector of the fitted values. The model gives
$Y_i=\mu_{\pi_0(i)}+\sigma_0U_i$ and
$\bm Y=\bm\mu+\bm\delta+\sigma_0\bm U$, and
$M_{(\bm X,\bm u)}\bm\mu=\bm\mu$ because $\bm\mu\in\range(\bm X)$. Hence
\[
 Y_i-\mu_{\pi_0(i)}=\sigma_0U_i,\qquad
 \bm m-\bm\mu=M_{(\bm X,\bm u)}\bm\delta+\sigma_0M_{(\bm X,\bm u)}\bm U .
\]
The $j$th element of $M_{(\bm X,\bm u)}\bm U$ is a centred Gaussian
variable with variance at most $h_{\bm u}$. Let
$t=\{2\log(4n/\xi)\}^{1/2}$. The bound $\bP\{|Z|>t\}\le2e^{-t^2/2}$ for
$Z\sim\Normal(0,1)$ and a union bound over the $2n$ elements of $\bm U$
and $M_{(\bm X,\bm u)}\bm U$ show that the inequalities
\[
 \|\bm U\|_\infty\le t\qquad\text{and}\qquad
 \|M_{(\bm X,\bm u)}\bm U\|_\infty\le h_{\bm u}^{1/2}t
\]
hold simultaneously with probability at least $1-4ne^{-t^2/2}=1-\xi$. On
this event, the quantity $\rho$ of Proposition~\ref{prop:lap-rowwise}
for $\bm y=\bm Y$ is at most $\rho_\xi(\bm u)$. The conditions
\eqref{eq:lap-rowwise-window} are monotone in $\rho$, so that they
follow from \eqref{eq:lambda-window}, and
Proposition~\ref{prop:lap-rowwise} shows that $\Pi_0$ is the unique
solution of \eqref{def:tilde_Pi}. This is the first claim.

The map $\bm u\mapsto M_{(\bm X,\bm u)}$ is Borel measurable, because
$M_{(\bm X,\bm u)}=M_{\bm X}+R_{\bm X}\bm u(R_{\bm X}\bm u)^\top/\|R_{\bm X}\bm u\|_2^2$
if $R_{\bm X}\bm u\ne\bm 0$ and $M_{(\bm X,\bm u)}=M_{\bm X}$ otherwise;
hence $\mathcal G_\xi$ is a Borel set. Since $\bm U^*$ is independent of
$\bm U$, conditioning on $\bm U^*$ and applying the first claim give
\[
 \bP\{\tilde\Pi(\bm U,\bm U^*)=\Pi_0\}
 \ge\bE\bigl[\1\{\bm U^*\in\mathcal G_\xi\}
   \bP\{\tilde\Pi(\bm U,\bm U^*)=\Pi_0\mid\bm U^*\}\bigr]
 \ge(1-\xi)\bP\{\bm U^*\in\mathcal G_\xi\},
\]
which is \eqref{eq:lap-recovery}. Finally,
Theorem~\ref{thm:recovery} with $L=1$ gives
$\bP\{\hat\Pi(\bm U,\bm U^*)\ne\Pi_0\}\le\delta_1$. The two solutions
agree when both are equal to $\Pi_0$, and a union bound with
\eqref{eq:lap-recovery} proves \eqref{eq:lap-agreement}.
\end{proof}

The proof uses the distribution of $\bm U$ only through the bound
$\bP\{|\bm a^\top\bm U|>t\|\bm a\|_2\}\le2e^{-t^2/2}$ for fixed vectors
$\bm a$. The first claim and \eqref{eq:lap-recovery} therefore hold
for every error vector with
$\bE\exp(\bm t^\top\bm U)\le\exp(\|\bm t\|_2^2/2)$ for all
$\bm t\in\R^n$. The interval for $\lambda_1$ in
\eqref{eq:lambda-window} is nonempty if and only if
$g_\star>(1+\sqrt3)\rho_\xi(\bm u)$.

\subsection{Proof of Proposition \ref{prop:tuning-null}}\label{app:tuning-null-proof}
\begin{proof}
Write $\bm Y=\bm Y^{\obs}$, $R_\ell=R_{(\bm X,\bm u_\ell^*)}$ and
$\bm m_\ell=M_{(\bm X,\bm u_\ell^*)}\bm Y$, so that
$\bm Y-\bm m_\ell=R_\ell\bm Y$, and put $j=\lceil(B+1)(1-\eta)\rceil$, so
that $B+1-j=\lfloor\eta(B+1)\rfloor$. We set $Q(\bm y)=+\infty$ if
$s(\bm y)=0$, and $\widehat q=+\infty$ if $j>B$. As in our implementation,
the LAP \eqref{def:tilde_Pi} returns $I_n$ whenever $I_n$ is one of its
solutions, and $\mC^{(L)}=\{I_n\}$ if $s(\bm Y)=0$ or $\widehat q=+\infty$.

We first show that $\mC^{(L)}\ne\{I_n\}$ implies $Q(\bm Y)>\widehat q$.
Suppose that $s(\bm Y)>0$, $\widehat q<\infty$ and $Q(\bm Y)\le\widehat q$.
Then $\|R_\ell\bm Y\|_\infty^2\le s(\bm Y)^2Q(\bm Y)\le
s(\bm Y)^2\widehat q=\lambda_1$ for every $\ell$. In the notation of
Appendix~\ref{app:lap-recovery-proof}, with $\bm y=\bm Y$ and
$\bm m=\bm m_\ell$, the map $\pi_0$ of $\Pi_0=I_n$ is the identity, and
every permutation $\pi$ satisfies
\begin{align*}
C(\pi)-C(\pi_0)
&=\sum_{i:\pi(i)\ne i}\bigl\{(Y_i-m_{\ell,\pi(i)})^2+\lambda_1-(Y_i-m_{\ell,i})^2\bigr\}\\
&\ge|\{i:\pi(i)\ne i\}| \bigl(\lambda_1-\|R_\ell\bm Y\|_\infty^2\bigr)\ge0 .
\end{align*}
Hence $I_n$ is a solution of the LAP for every repro sample, the LAP
returns $\tilde\Pi_\ell=I_n$ for every $\ell$, step~4 of the procedure in
Section~\ref{sec:tuning} is not invoked, and $\mC^{(L)}=\{I_n\}$.

Next, we bound the probability of $\{Q(\bm Y)>\widehat q\}$. Under
$\Pi_0=I_n$, we have $\bm Y=\bm X\bm\beta_0+\sigma_0\bm U$, and the
invariance $Q(\bm X\bm b+c \bm y)=Q(\bm y)$, which also holds when
$s(\bm y)=0$ because $s(\bm X\bm b+c \bm y)=c s(\bm y)$, gives
$Q(\bm Y)=Q(\bm U)$. Since $\bm U,\bm g_1,\ldots,\bm g_B$ are independent
$\Normal(0,I_n)$ vectors, independent of the repro samples, the variables
$Z_0=Q(\bm U)$ and $Z_b=Q(\bm g_b)$, $b=1,\ldots,B$, are independent and
identically distributed conditionally on $\bm u_1^*,\ldots,\bm u_L^*$. If
$Z_0>\widehat q$, then $Z_0$ is larger than at least $j$ of
$Z_1,\ldots,Z_B$. For any values, at most $B+1-j$ of $Z_0,\ldots,Z_B$ are
larger than at least $j$ of the others, because such a value is preceded
by at least $j$ values in every nondecreasing arrangement. By
exchangeability, each $Z_a$ has this property with the same probability,
so that
\[
\bP\bigl(Q(\bm Y)>\widehat q\bigr)\le\frac{B+1-j}{B+1}=\eta_B .
\]
Together with the first step, this proves \eqref{eq:tuning-null-guarantee}.

Finally, as in the proof of Theorem~\ref{thm:confidence}, the confidence
set \eqref{eq:confidence set beta zero} misses $\bm\beta_0$ only if
$I_n\notin\mC^{(L)}$ or
$T(\bm Y,(I_n,\bm\beta_0))>F_{p,n-p}^{-1}(1-\alpha)$; the former has
probability at most $\eta_B$, and the latter has probability $\alpha$.
When $\mC^{(L)}=\{I_n\}$, we have $D^{\obs}=0$, and the test does not
reject because $\widehat c^{(k_0)}_{1-\alpha}\ge0$; hence its type-I error
is at most $\eta_B$.
\end{proof}

\section{Additional theoretical results}\label{section:appendix theory}

\subsection{A counterexample when $n-2k < p$}
\begin{theorem}\label{thm:counterexample}
    Let $p<n$ and $n-2k \le p - 1$. Further, assume $k \ge 2$ so that a nontrivial transposition is permitted.
Then, there exist a full-column-rank matrix $\bm{X} \in \R^{n\times p}$, a pair of distinct permutations $\Pi_0,\Pi_1 \in \mP_{n,k}$,
and coefficient vectors $\bm{\beta}_0 \ne \bm{\beta}_1 \in \R^p$ such that
\[
\Pi_0 \bm{X} \bm{\beta}_0 = \Pi_1 \bm{X} \bm{\beta}_1
\]
holds. Consequently, the minimizers of \eqref{eq:repro_ideal} are not unique in general, even with $\sigma_0 = 0$.
\end{theorem}

\begin{proof}
Fix two distinct indices $a\neq b$ in $[n]$. Set $\Pi_0:=I_n$ and let $\Pi_1$ be the transposition that swaps $a$ and $b$.
Then $d(\Pi_1,I_n)=2\le k$ because $k\ge2$, hence we have $\Pi_0,\Pi_1\in\mP_{n,k}$.

Set
\[
t:=\max\{n-2k,0\}.
\]
By the assumption $n-2k\le p-1$ we have $t\le p-1$. Choose any subset $T\subset[n]\setminus\{a,b\}$ with $|T|=t$ and set $S:=[n]\setminus T$.

We construct a full-column-rank design $\bm X\in\R^{n\times p}$ together with $\bm\beta_0\ne\bm\beta_1$
and $\Pi_0\bm X\bm\beta_0=\Pi_1\bm X\bm\beta_1$.

First, consider the case $p=1$. Put $\bm X\in\R^{n\times1}$ with $X_a=1$, $X_b=-1$, and $X_i=0$ for $i\notin\{a,b\}$.
Then, we have $\rank(\bm X)=1$ (full column rank). Take $\beta_0=1$ and $\beta_1=-1$.
We have $(\bm X\beta_0)_a=1$ and $(\bm X\beta_0)_b=-1$, while $\bm X\beta_1$ equals $-1$ at $a$ and $1$ at $b$,
so applying the transposition $\Pi_1$ swaps these two values and yields
\[
\Pi_1\bm X\beta_1=\bm X\beta_0=\Pi_0\bm X\beta_0,
\]
with $\beta_0\ne\beta_1$.

Next, consider the case $p\ge2$.
Let $v:=e_1\in\R^p$ and $w:=e_2\in\R^p$, so $w\in v^\perp$ and $\|v\|_2^2=1$. Define two row vectors
\[
r_a:=w+v,\qquad r_b:=w-v,
\]
so that we have $r_a v=1$ and $r_b v=-1$.

We define the rows of $\bm X$ as follows.
\begin{itemize}
\item Put $X_a:=r_a$ and $X_b:=r_b$.
\item For the $t$ indices in $T$, assign distinct standard basis vectors among $\{e_2,e_3,\dots,e_{t+1}\}$:
for example, after ordering $T=\{i_1,\dots,i_t\}$, set $X_{i_j}:=e_{j+1}$.
This gives $\rank(\bm X_T)=t$ and $\bm X_T v=0$.
\item For the remaining indices in $S\setminus\{a,b\}$, assign the missing basis vectors among $\{e_2,\dots,e_p\}$
so that the set of rows contains $\{e_2,\dots,e_p\}$ at least once; if there are extra rows, assign any vectors in $v^\perp$. This is feasible because $n\ge p+1$ (since $p<n$), hence
$|S\setminus\{a,b\}| = n-t-2 \ge p-t-1$.
\end{itemize}
By construction, all rows except $a,b$ lie in $v^\perp$, hence we have $X_i v=0$ for all $i\notin\{a,b\}$.
Moreover, the row set contains $\{r_a,e_2,\dots,e_p\}$, which is linearly independent, so $\rank(\bm X)=p$
(full column rank) holds.

Now choose $\bm\beta_0:=-\tfrac12 v+\tfrac12 w$ and $\bm\beta_1:=\bm\beta_0+v=\tfrac12 v+\tfrac12 w$.
Then, we have
\[
r_a\bm\beta_0=(v+w)\cdot\left(-\tfrac12 v+\tfrac12 w\right)=0,\qquad
r_b\bm\beta_0=(w-v)\cdot\left(-\tfrac12 v+\tfrac12 w\right)=1,
\]
and
\[
r_a\bm\beta_1=(v+w)\cdot\left(\tfrac12 v+\tfrac12 w\right)=1,\qquad
r_b\bm\beta_1=(w-v)\cdot\left(\tfrac12 v+\tfrac12 w\right)=0.
\]
For every $i\notin\{a,b\}$, because of $X_i\in v^\perp$ we have $X_i(\bm\beta_1-\bm\beta_0)=X_i v=0$, hence it holds that $X_i\bm\beta_1=X_i\bm\beta_0$.
Therefore, $\bm X\bm\beta_1$ equals $\bm X\bm\beta_0$ at all rows except $a,b$, where the two values are swapped.
Applying the transposition $\Pi_1$ thus yields
\[
\Pi_1\bm X\bm\beta_1=\bm X\bm\beta_0=\Pi_0\bm X\bm\beta_0,
\]
with $\bm\beta_0\ne\bm\beta_1$.

In either case, we construct the design matrix $\bm X$ with full column-rank, the permutations $\Pi_0\ne\Pi_1$ in $\mP_{n,k}$, and the coefficients $\bm\beta_0\ne\bm\beta_1$
satisfying $\Pi_0\bm X\bm\beta_0=\Pi_1\bm X\bm\beta_1$. This yields the non-uniqueness of \eqref{eq:repro_ideal} even when $\sigma_0=0$.
\end{proof}

\subsection{Theory on the partially permuted designs}\label{sec:partial}

\begin{proposition}\label{prop:identification partial}
    Assume that the rows of design matrix $\bm{X}$ and $\bm{Z}$ are i.i.d. generated by a continuous distribution. Also, suppose $n - 2k \ge p_1 + p_2$. Then, true parameters
$(\Pi_0,\bm\beta_{0,1},\bm\beta_{0,2},\sigma_0)$ can be uniquely recovered by the following optimization problem:
\begin{equation}\label{eq:repro-ideal-partial}
(\Pi_0,\bm\beta_{0,1},\bm\beta_{0,2},\sigma_0)
=
\argmin_{\Pi\in\mP_{n,k},\ \bm\beta_1\in\R^{p_1},\ \bm\beta_2\in\R^{p_2},\ \sigma\ge 0}
\left\|
\bm Y^{\obs}-\Pi \bm X\bm\beta_1-\bm Z\bm\beta_2-\sigma \bm u^{\rel}
\right\|^2 .
\end{equation}
For a fixed $\Pi \in \mP_n$, define an expanded design matrix $W(\Pi) = [\Pi \bm{X}, \bm{Z}]$. Furthermore, it holds that
\begin{equation}
   \Pi_0 = \argmin_{\Pi \in \mP_{n,k}}\| (I_n - M_{(W(\Pi), \bm{u}^{\rel})}) \bm{Y}^{\obs}\|_2^2,
\end{equation}
where $M_{(W(\Pi), \bm{u}^{\rel})}$ is the projection matrix to the range of $(W(\Pi), \bm{u}^{\rel})$.
\end{proposition}

\begin{proof}
Applying same strategy in the proof of Proposition \ref{prop:identification} and treating $M_{(\Pi \bm{X}, \bm{u}^{\rel})}$ as $M_{(W(\Pi), \bm{u}^{\rel})}$ yield the desired result.
\end{proof}

Recall the partially permuted model \eqref{eq:model_partial} and write $p:=p_1+p_2$.
For any $\Pi\in\mP_{n,k}$ define $W(\Pi):=[\Pi\bm X,\bm Z]$ and $M_{W(\Pi)}$ as in
\eqref{eq:Wpartial-def}.  The corresponding separation constant is given by
\begin{equation}\label{eq:c-min-partial}
C_{\min}^{\mathrm{partial}}(\Pi_0)
:=\min_{\substack{\Pi\in\mP_{n,k}\\ \Pi\neq \Pi_0}}
\frac{\left\| (I_n-M_{W(\Pi)})W(\Pi_0)\binom{\bm\beta_{1,0}}{\bm\beta_{2,0}}\right\|_2^2}
     {n\max\{d(\Pi_0,I_n)-d(\Pi,I_n),1\}}
=\min_{\substack{\Pi\in\mP_{n,k}\\ \Pi\neq \Pi_0}}
\frac{\left\| (I_n-M_{(\Pi\bm X,\bm Z)})\Pi_0\bm X\bm\beta_{1,0}\right\|_2^2}
     {n\max\{d(\Pi_0,I_n)-d(\Pi,I_n),1\}}.
\end{equation}

\begin{theorem}\label{thm:recovery partial}
Suppose $C_{\min}^{\mathrm{partial}}(\Pi_0) > 0$ and $n > p_1 + p_2 + 1$ hold.
Suppose also that $W(\Pi)$ has rank $p_1+p_2$ for every $\Pi\in\mP_{n,k}$.
Then, for all $L\ge1$, we have
\begin{equation}\label{eq:recovery-finite-L-partial}
    \bP_{(\mathcal{U}^L, \bm{U})}\left( \Pi_0 \notin \mathcal{C}_{\mathrm{partial}}^{(L)} \right)
    \le
    \delta_L^{\mathrm{partial}},
\end{equation}
where $\delta_L^{\mathrm{partial}}$ is defined as $\delta_L$ in~\eqref{eq:delta_L} with $p$, $C_{\min}(\Pi_0)$ and $\Pi_0\bm X\bm\beta_0$ replaced by $p_1+p_2$, $C_{\min}^{\mathrm{partial}}(\Pi_0)$ and $\Pi_0\bm X\bm\beta_{1,0}+\bm Z\bm\beta_{2,0}$, respectively.
\end{theorem}

\begin{proof}
This is an immediate application of Theorem~\ref{thm:recovery} to the design matrices
$W(\Pi) = (\Pi \bm{X}, \bm{Z})$ with its total dimension $p=p_1+p_2$, for which the separation constant is
$C_{\min}^{\mathrm{partial}}(\Pi_0)$ as defined above.
Indeed, Theorem~\ref{ta-g-thm-tube} applies to the matrices $W(\Pi)$, $\Pi\in\mP_{n,k}$, which have rank $p_1+p_2$, with $\bm s=\Pi_0\bm X\bm\beta_{1,0}+\bm Z\bm\beta_{2,0}$ and the distances $\|(I_n-M_{W(\Pi)})\bm s\|_2$, which satisfy $\|(I_n-M_{W(\Pi)})\bm s\|_2^2\ge nC_{\min}^{\mathrm{partial}}(\Pi_0)\max\{k_\star-h,1\}$ for $d(\Pi,I_n)=h$. The proof of Proposition~\ref{cor:compact-small-list} applies with $\Pi\bm X$ replaced by $W(\Pi)$, because the number $m$ of Lemma~\ref{lem:sharp-pairwise} is then at most $p_1+p_2$.
\end{proof}

\begin{theorem}[Joint confidence set for $(\bm{\beta}_1,\bm{\beta}_2)$]\label{thm:confidence-set-partial-joint}
Fix any $\alpha\in(0,1)$ and consider the partially permuted model~\eqref{eq:model_partial}.
Under the conditions of Theorem~\ref{thm:recovery partial}, the confidence set
$\Gamma_{\alpha}^{\mathrm{joint}}(\bm Y^{\obs})$ defined in~\eqref{eq:confset-partial-joint}
satisfies, for any $L\ge 1$,
\[
\bP_{(\mathcal U^L,\bm U)}\left(
(\bm\beta_{1,0},\bm\beta_{2,0}) \in \Gamma_{\alpha}^{\mathrm{joint}}(\bm Y^{\obs})
\right)
\ge 1-\alpha-\delta_L^{\mathrm{partial}},
\]
where $\delta_L^{\mathrm{partial}}$ is defined in Theorem~\ref{thm:recovery partial}.
\end{theorem}

\begin{proof}
Set $p:=p_1+p_2$.
Recall that
\begin{equation}\label{eq:confset-partial-joint-recalled}
\Gamma_\alpha^{\mathrm{joint}}(\bm Y^{\obs})
:=
\bigcup_{\Pi \in \mathcal{C}_{\mathrm{partial}}^{(L)}(\bm Y^{\obs})}
\left\{
(\bm\beta_1,\bm\beta_2) \in \mathbb R^{p}:
T(\bm Y^{\obs},(\Pi,\bm\beta_1,\bm\beta_2))
\le F^{-1}_{p,n-p}(1-\alpha)
\right\},
\end{equation}
where the statistic $T(\bm Y^{\obs},(\Pi,\bm\beta_1,\bm\beta_2))$ is defined in~\eqref{eq:Tjointpartial-def}.

On $\{\Pi_0\in\mathcal C_{\mathrm{partial}}^{(L)}(\bm Y^{\obs})\}$, we have
\[
(\bm\beta_{1,0},\bm\beta_{2,0})\notin\Gamma_{\alpha}^{\mathrm{joint}}(\bm Y^{\obs})
\ \Longrightarrow\
T(\bm Y^{\obs},(\Pi_0,\bm\beta_{1,0},\bm\beta_{2,0}))>F^{-1}_{p,n-p}(1-\alpha).
\]
Hence, by Theorem~\ref{thm:recovery partial}, it holds that
\begin{align*}
\bP\left((\bm\beta_{1,0},\bm\beta_{2,0})\notin\Gamma_{\alpha}^{\mathrm{joint}}(\bm Y^{\obs})\right)
&\le \delta_L^{\mathrm{partial}}
+\bP\left(T(\bm Y^{\obs},(\Pi_0,\bm\beta_{1,0},\bm\beta_{2,0}))>F^{-1}_{p,n-p}(1-\alpha)\right).
\end{align*}

For the second term, under \eqref{eq:model_partial} we have
\[
\bm Y^{\obs}-\Pi_0\bm X\bm\beta_{1,0}-\bm Z\bm\beta_{2,0}
=\sigma_0\bm U,
\qquad \bm U\sim N(0,I_n).
\]
Define $M_0:=M_{(\Pi_0\bm X,\bm Z)}$ and $P_0:=I_n-M_0$. Under the rank condition
$\rank(\Pi_0\bm X,\bm Z)=p$, $M_0\bm U$ and $P_0\bm U$ are independent Gaussian vectors with
\[
\|M_0\bm U\|_2^2\sim \chi^2_p,
\qquad
\|P_0\bm U\|_2^2\sim \chi^2_{n-p},
\]
and they are independent. Therefore, we have
\[
T(\bm Y^{\obs},(\Pi_0,\bm\beta_{1,0},\bm\beta_{2,0}))
=\frac{\|M_0\bm U\|_2^2/p}{\|P_0\bm U\|_2^2/(n-p)}
\sim F_{p,n-p},
\]
which implies $\bP\left(T(\bm Y^{\obs},(\Pi_0,\bm\beta_{1,0},\bm\beta_{2,0}))\le F^{-1}_{p,n-p}(1-\alpha)\right)=1-\alpha$. Combining the bounds yields
\[
\bP\left((\bm\beta_{1,0},\bm\beta_{2,0})\in\Gamma_{\alpha}^{\mathrm{joint}}(\bm Y^{\obs})\right)
\ge 1-\alpha-\delta_L^{\mathrm{partial}}.
\]
\end{proof}

\begin{theorem}[Marginal confidence set for $\bm{\beta}_1$]\label{thm:confidence set partial beta1}
Fix any $\alpha \in (0,1)$. Consider the partially permuted design~\eqref{eq:model_partial}
defined in~\eqref{eq:confset-partial-beta1}. It holds that, for any $L \ge 1$,
\[
    \bP_{(\mathcal{U}^L, \bm{U})}\left( \bm{\beta}_1 \in \Gamma_{\alpha}^{\mathrm{partial}}(\bm{Y}^{\obs}) \right)
    \ge
    1 - \alpha - \delta_L^{\mathrm{partial}},
\]
where $\delta_L^{\mathrm{partial}}$ is defined in Theorem~\ref{thm:recovery partial}.
\end{theorem}

\begin{proof}
By definition, we have
\begin{equation}\label{eq:confset-partial-beta1}
\Gamma_{\alpha}^{\mathrm{partial}}(\bm{Y}^{\obs})
:=\bigcup_{\Pi \in \mathcal{C}_{\mathrm{partial}}^{(L)}} \left\{
\bm\beta\in \R^{p_1}:\
T(\bm{Y}^{\obs}, (\Pi,\bm\beta) \mid \bm Z) \le F_{p_1, n-p}^{-1}(1-\alpha)
\right\}.
\end{equation}

Hence, we have
\[
\bP\left( \bm{\beta}_1 \in \Gamma_{\alpha}^{\mathrm{partial}}(\bm{Y}^{\obs}) \right)
\ge
\bP\left(
T\left(\bm{Y}^{\obs}, (\Pi_0, \bm{\beta}_1) \mid \bm{Z}\right) \le F_{p_1, n-p}^{-1}(1-\alpha)
\right)
-
\bP\left( \Pi_0 \notin \mathcal{C}_{\mathrm{partial}}^{(L)} \right).
\]
Under the true parameters $(\Pi_0, \bm{\beta}_1)$, the statistic
$T(\bm{Y}^{\obs}, (\Pi_0, \bm{\beta}_1)\mid \bm{Z})$ has an $F_{p_1, n-p}$ distribution, so the first probability equals $1-\alpha$.
Theorem~\ref{thm:recovery partial} bounds the second probability by $\delta_L^{\mathrm{partial}}$.
\end{proof}

\subsection{Theory on permuted ridge regression}\label{sec:ridge}

For a permutation $\Pi$, a vector $\bm u\in\R^n$ and $\lambda>0$, let
$F_{\lambda,\Pi}(\bm u)$ denote the criterion of
\eqref{eq:perm_candidates_ridge} with the response $\bm Y$ and with
$\bm U_\ell^*$ replaced by $\bm u$, that is,
\begin{equation}
F_{\lambda,\Pi}(\bm u)
=\min_{\bm b\in\R^p,\ a\in\R}
 \left\{\|\bm Y-\Pi\bm X\bm b-a\bm u\|_2^2+\lambda\|\bm b\|_2^2\right\},
\label{eq:ridge-criterion}
\end{equation}
so that $F_{\lambda,\Pi}(\bm u)\ge F_\Pi(\bm u)$. Whereas
$F_{\Pi_0}(\bm U)=0$, the criterion $F_{\lambda,\Pi_0}(\bm U)$ is almost
surely positive when $\bm\beta_0\ne\bm 0$, because of the penalty; its
value is as follows.

\begin{proposition}
\label{prop:ridge-profile}
For $\lambda>0$ and
$\bm G=(\Pi_0\bm X)^\top R_{\bm U}\Pi_0\bm X$,
\begin{equation}
F_{\lambda,\Pi_0}(\bm U)
=\lambda\bm\beta_0^\top\bm G(\bm G+\lambda I_p)^{-1}\bm\beta_0
\le\lambda\bm\beta_0^\top\bm X^\top\bm X
 (\bm X^\top\bm X+\lambda I_p)^{-1}\bm\beta_0
\le\lambda\|\bm\beta_0\|_2^2 .
\end{equation}
\end{proposition}
\begin{proof}
Since $\bm Y=\Pi_0\bm X\bm\beta_0+\sigma_0\bm U$, the minimum of
$\|\bm Y-\Pi_0\bm X\bm b-a\bm U\|_2^2$ over $a\in\R$ is
$\|R_{\bm U}\Pi_0\bm X(\bm\beta_0-\bm b)\|_2^2
=(\bm\beta_0-\bm b)^\top\bm G(\bm\beta_0-\bm b)$, and the minimum of
$(\bm\beta_0-\bm b)^\top\bm G(\bm\beta_0-\bm b)+\lambda\|\bm b\|_2^2$
over $\bm b\in\R^p$, attained at
$\bm b=(\bm G+\lambda I_p)^{-1}\bm G\bm\beta_0$, is
$\lambda\bm\beta_0^\top\bm G(\bm G+\lambda I_p)^{-1}\bm\beta_0$. The map
$\bm H\mapsto\lambda\bm H(\bm H+\lambda I_p)^{-1}
=\lambda I_p-\lambda^2(\bm H+\lambda I_p)^{-1}$ is nondecreasing in the
order of positive semidefinite matrices and bounded by $\lambda I_p$,
and $\bm G\preceq\bm X^\top\bm X$; this gives the two inequalities.
\end{proof}

The next proposition bounds the miscoverage probability of the candidate
set $\mC_\lambda^{(L)}$ when $\lambda\|\bm\beta_0\|_2^2$ is small
relative to the separation.

\begin{proposition}
\label{prop:ridge-finite}
Under the assumptions of Theorem~\ref{thm:recovery-threshold}, for every
$L\ge1$ and every $0<\gamma\le1/4$ such that
$\lambda\|\bm\beta_0\|_2^2<t_\gamma\underline w^2$,
\begin{equation}
\bP\{\Pi_0\notin\mC_\lambda^{(L)}\}
\le(1-a_\gamma)^L+Nb_\gamma
+\bP\left\{\|\bm U\|_2\ge
 \frac{\{(1+t_\gamma)\underline w^2-\lambda\|\bm\beta_0\|_2^2\}^{1/2}
 -\underline w}{\sigma_0\gamma}\right\}.
\label{eq:ridge-finite}
\end{equation}
\end{proposition}
\begin{proof}
We use the events $A_\ell$ and $B_\ell$ of the proof of
Theorem~\ref{thm:recovery-threshold}, and put
$c=\sigma_0\gamma\|\bm U\|_2$. On $A_\ell$, taking $\bm b=\bm\beta_0$ in
\eqref{eq:ridge-criterion} gives
$F_{\lambda,\Pi_0}(\bm U_\ell^*)\le\lambda\|\bm\beta_0\|_2^2+c^2$. On
$A_\ell\cap B_\ell^c$, as in \eqref{eq:app-separation} and by
$F_{\lambda,\Pi}\ge F_\Pi$, for every
$\Pi\in\mP_{n,k}\setminus\{\Pi_0\}$,
\[
F_{\lambda,\Pi}(\bm U_\ell^*)-F_{\lambda,\Pi_0}(\bm U_\ell^*)
>t_\gamma w_\Pi^2-2cw_\Pi-c^2-\lambda\|\bm\beta_0\|_2^2 .
\]
Suppose that $c$ is smaller than the numerator of the fraction in
\eqref{eq:ridge-finite}, that is,
$c^2+2c\underline w<t_\gamma\underline w^2-\lambda\|\bm\beta_0\|_2^2$. Then
$c<t_\gamma\underline w/2$, the function
$w\mapsto t_\gamma w^2-2cw$ is increasing for $w\ge\underline w$, and,
since $w_\Pi\ge\underline w$, the right-hand side of the last display is
at least
$t_\gamma\underline w^2-2c\underline w-c^2-\lambda\|\bm\beta_0\|_2^2>0$,
so that $\Pi_0$ is the minimizer for the $\ell$th repro sample. The
conclusion follows as in the proof of
Theorem~\ref{thm:recovery-threshold}.
\end{proof}

The ridge penalty is used only for generating the candidate set: the
confidence set for $\bm\beta_0$ is the union of the sets in
\eqref{eq:confidence set beta zero} over $\mC_\lambda^{(L)}$, and its
coverage is at least $1-\alpha-\bP\{\Pi_0\notin\mC_\lambda^{(L)}\}$ by
\eqref{eq:union-principle}.

\subsection{Non-Gaussian errors and misspecified models}
\label{app:extension-details}
This subsection states what remains valid when the errors are not
Gaussian or the mean is misspecified; the repro samples are kept Gaussian
throughout. For variable selection, \cite{wang2022finite} also treat
heterogeneous, non-Gaussian and sub-Gaussian errors.

\subsubsection{Jointly spherical errors}
Suppose that the error vector $\bm U$ in \eqref{eq:random}
satisfies $\bP(\bm U=\bm 0)=0$ and
\begin{equation}
\bm V=\frac{\bm U}{\|\bm U\|_2}\sim\operatorname{Unif}(\mathbb S^{n-1})
\quad\text{independently of }\|\bm U\|_2,
\label{eq:spherical}
\end{equation}
that is, the direction $\bm V$ of the error vector is uniformly
distributed on the unit sphere $\mathbb S^{n-1}$ of $\R^n$ and independent
of its norm. Multivariate $t$ errors with a common scale are of this form,
whereas vectors with independent Laplace or Student-$t$ elements are not.

\begin{proposition}
\label{prop:spherical}
Under \eqref{eq:spherical}, the bound \eqref{eq:explicit-delta} of
Theorem~\ref{thm:recovery-threshold} remains valid. Moreover, for every
orthogonal projection matrix $M$ of rank $p$,
\begin{equation}
\frac{\bm U^\top M\bm U/p}{\bm U^\top(I_n-M)\bm U/(n-p)}
\sim F_{p,n-p},
\label{eq:spherical-F}
\end{equation}
and the conditional representation
\eqref{eq:conditional-representation} remains valid. Hence
Theorems~\ref{thm:p-value} and \ref{thm:confidence} hold with
$\delta_L$ replaced by $\bP\{\Pi_0\notin\mC^{(L)}\}$.
\end{proposition}
\begin{proof}
The proof of Theorem~\ref{thm:recovery-threshold} uses the distribution
of $\bm U$ only through $\bm U\ne\bm 0$ and the invariance of this
distribution under rotations; it therefore applies to
\eqref{eq:spherical}.

For \eqref{eq:spherical-F}, the ratio does not change when $\bm U$ is
multiplied by a positive number, so that it is a function of $\bm V$.
Since $\bm V$ has the same distribution as $\bm Z/\|\bm Z\|_2$ for
$\bm Z\sim\Normal(0,I_n)$, the ratio has the same distribution as the
same ratio for $\bm Z$, which is the $F_{p,n-p}$ distribution.

For the conditional representation, fix
$\bm\theta=(\Pi,\bm\beta,\sigma)$, and let $O$ be an orthogonal matrix
with $O\bm v=\bm v$ for every $\bm v\in\range(\Pi\bm X)$. Then
$O\bm Y_{\bm\theta}$ has the same distribution as $\bm Y_{\bm\theta}$,
and the statistics $S_{1,\Pi}$ and $S_{2,\Pi}$ take the same values at
$O\bm Y_{\bm\theta}$ and at $\bm Y_{\bm\theta}$. Hence the conditional
distribution of the direction
$R_{\Pi\bm X}\bm Y_{\bm\theta}/\|R_{\Pi\bm X}\bm Y_{\bm\theta}\|_2$ given
these statistics is invariant under all rotations of
$\range(R_{\Pi\bm X})$, and it is therefore the uniform distribution on
the unit sphere of that space; the residual
$R_{\Pi\bm X}\bm Y_{\bm\theta}$ is nonzero almost surely because
$\bm U\ne\bm 0$ almost surely. This is \eqref{eq:conditional-representation}. The
proofs of Theorems~\ref{thm:p-value} and \ref{thm:confidence} use
the Gaussian distribution only through
\eqref{eq:conditional-representation}, the $F_{p,n-p}$ distribution
of $T(\bm Y,(\Pi_0,\bm\beta_0))$ and the bound on
$\bP\{\Pi_0\notin\mC^{(L)}\}$.
\end{proof}

The bound \eqref{eq:explicit-delta} requires a bound for the tail of the
norm $\|\bm U\|_2$. The following bound has the exponent one of $L$ and holds for
every distribution of the norm. Consider the setting
\eqref{ta-g-eq-setting} of Appendix~\ref{sec:wxz-gaussian-bounds} with an
error vector $\bm U$ whose distribution is invariant under rotations, that
is, $O\bm U$ has the same distribution as $\bm U$ for every orthogonal
matrix $O$, as is the case under \eqref{eq:spherical}; here $\bm U=\bm 0$
is allowed with positive probability. The repro samples are independent
$\Normal(0,I_n)$ vectors, independent of $\bm U$. With $\nu$ and $\kappa$ as in \eqref{ta-g-eq-tube-parameters},
let
\begin{equation}
H=\sum_{i=1}^N\left(\frac{\|\bm s\|_2}{w_i}\right)^{\nu}.
\end{equation}

\begin{theorem}
\label{ta-s-thm-spherical-tube}
In this setting, for every $L\ge1$,
\begin{equation}
\bP\{\bm A_0\notin\mC^{(L)}\}\le\min\left\{1,\frac{C_{\rm sph}}{L+1}\right\},
\qquad
C_{\rm sph}=2e^{1/(2\kappa^2)}
 \left[\kappa\sqrt n\left\{1+2(2H)^{1/\nu}\right\}\right]^{\nu}H .
\end{equation}
The constant $C_{\rm sph}$ depends neither on $\sigma_0$ nor on the
distribution of $\|\bm U\|_2$.
\end{theorem}

Since $C_{\rm sph}$ is increasing in $H$, the sum $H$ can be replaced by
an upper bound; for the permutation problem,
$H\le N\{\|\bm s\|_2^2/(nC_{\min}(\Pi_0))\}^{\nu/2}$ by
\eqref{eq:Cmin-basic}.

\begin{proof}[Proof of Theorem~\ref{ta-s-thm-spherical-tube}]
By \eqref{ta-g-eq-conditional-identity}, which only requires that the
repro samples are independent of $\bm Y$, it suffices to show that
$\bE[\{1-q(\bm Y)\}^L\mid\|\bm U\|_2=r]\le C_{\rm sph}/(L+1)$ for every
$r\ge0$. If $r=0$, then $\bm Y=\bm s$, and $q(\bm s)=1$, because
$R_{(\bm A_0,\bm u)}\bm s=\bm 0$ for every $\bm u$, whereas
$R_{(\bm A_i,\bm u)}\bm s\ne\bm 0$ for $i\ge1$ and almost every $\bm u$.
For $r>0$, we distinguish two cases, with
\[
r_\star=\frac{\|\bm s\|_2}{2\sigma_0}(2H)^{-1/\nu},\qquad
\tau_\star=\frac{\sigma_0r_\star}{\|\bm s\|_2+\sigma_0r_\star}
 =\left\{1+2(2H)^{1/\nu}\right\}^{-1}.
\]

First, let $0<r\le r_\star$. Since $\|\bm Y-\bm s\|_2=\sigma_0r$, we
have $\|R_{(\bm A_0,\bm u)}\bm Y\|_2\le\sigma_0r$ and
$\|R_{(\bm A_i,\bm u)}\bm Y\|_2\ge\|R_{(\bm A_i,\bm u)}\bm s\|_2-\sigma_0r$
for every $\bm u$, so that $\bm A_0$ is selected whenever
$\|R_{(\bm A_i,\bm u)}\bm s\|_2>2\sigma_0r$ for every $i\ge1$. For a
repro sample $\bm U^*$, as in \eqref{eq:sharp-angle},
$\|R_{(\bm A_i,\bm U^*)}\bm s\|_2^2/w_i^2\sim\operatorname{Beta}(\nu/2,1/2)$.
Moreover, $I_x(a,1/2)\le x^a$ for $0\le x\le1$, because
\[
\frac{I_x(a,1/2)}{x^a}
=\frac1{B(a,1/2)}\int_0^1t^{a-1}(1-xt)^{-1/2}dt
\]
is nondecreasing in $x$ and equal to one at $x=1$. Since
$2\sigma_0r_\star=\|\bm s\|_2(2H)^{-1/\nu}<w_i$ for every $i$, a union
bound gives
\[
1-q(\bm Y)\le\sum_{i=1}^N\left(\frac{2\sigma_0r}{w_i}\right)^{\nu}
\le\sum_{i=1}^N\left(\frac{2\sigma_0r_\star}{w_i}\right)^{\nu}=\frac12,
\]
and $\bE[\{1-q(\bm Y)\}^L\mid\|\bm U\|_2=r]\le2^{-L}\le1/(L+1)$.

Second, let $r>r_\star$. Given $\|\bm U\|_2=r$, the direction
$\bm V=\bm U/r$ is uniformly distributed on $\mathbb S^{n-1}$ by the
invariance under rotations, and the response
$\bm Y=\bm s+\sigma_0r\bm V$ lies outside the subspaces excluded in
Lemma~\ref{lem:tube-pointwise} with probability one, so that
$q(\bm Y)\ge c_\nu\rho(\bm Y)^\nu$. Since
$R_{(\bm A_i,\bm s)}\bm Y=\sigma_0rR_{(\bm A_i,\bm s)}\bm V$ and
$\|\bm Y\|_2\le\|\bm s\|_2+\sigma_0r$,
\[
\frac{w_i\|R_{(\bm A_i,\bm s)}\bm Y\|_2}{\|\bm s\|_2\|\bm Y\|_2}
\ge\frac{w_i}{\|\bm s\|_2}\frac{\sigma_0r}{\|\bm s\|_2+\sigma_0r}
 \|R_{(\bm A_i,\bm s)}\bm V\|_2
\ge\frac{w_i\tau_\star}{\|\bm s\|_2}\|R_{(\bm A_i,\bm s)}\bm V\|_2 ,
\]
and $\|R_{(\bm A_i,\bm s)}\bm V\|_2^2
\sim\operatorname{Beta}(\nu/2,(p+1)/2)$. As in the proof of
Theorem~\ref{ta-g-thm-tube}, the bound
$I_x(a,b)\le x^a/\{aB(a,b)\}$ for $b\ge1$ and a union bound give, for
every $z\ge0$,
\[
\bP\{\rho(\bm Y)\le z\mid\|\bm U\|_2=r\}
\le\frac{\Gamma(n/2)}{\Gamma(\nu/2+1)\Gamma((p+1)/2)}
 \tau_\star^{-\nu}Hz^\nu
\le\frac{n^{\nu/2}}{2^{\nu/2}\Gamma(\nu/2+1)}\tau_\star^{-\nu}Hz^\nu .
\]
Hence $\bP\{q(\bm Y)\le t\mid\|\bm U\|_2=r\}\le C_{\rm sph}t$ for every
$t\ge0$, and the integration at the end of the proof of
Theorem~\ref{ta-g-thm-tube} gives the bound $C_{\rm sph}/(L+1)$. Since
$C_{\rm sph}\ge1$, the bound also holds in the other two cases.
\end{proof}

\subsubsection{Sub-Gaussian errors}
For errors that satisfy a bound on the moment generating function with a
known constant, the $F$ regions of
\eqref{eq:confidence set beta zero} can be replaced by regions whose
radius is calibrated by that constant. Suppose that
\begin{equation}
\bm Y=\Pi_0\bm X\bm\beta_0+\bm b+\bm\varepsilon,\qquad
\bE\exp(\bm t^\top\bm\varepsilon)\le\exp(v^2\|\bm t\|_2^2/2)
\quad\text{for every }\bm t\in\R^n,
\label{eq:sg-model}
\end{equation}
where $\bm b\in\R^n$ is a deterministic error of the mean with
$\|\bm b\|_2\le\bar b$, and $v\ge0$ and $\bar b\ge0$ are known. If the
design is random, the condition on $\bm\varepsilon$ is required
conditionally on the design.

\begin{theorem}
\label{thm:sg-score-region}
Suppose \eqref{eq:sg-model} and $\rank(\bm X)=p$. For $\alpha\in(0,1)$,
let
\[
r_\alpha=v\left\{p+2\sqrt{p\log(1/\alpha)}+2\log(1/\alpha)\right\}^{1/2}
 +\bar b .
\]
Then, for every fixed set $\mathcal P$ of permutation matrices that
contains $\Pi_0$,
\begin{equation}
\bP\left\{\bm\beta_0\in\bigcup_{\Pi\in\mathcal P}
 \left\{\bm\beta\in\R^p:
 \|M_{\Pi\bm X}(\bm Y-\Pi\bm X\bm\beta)\|_2\le r_\alpha\right\}\right\}
\ge1-\alpha .
\label{eq:sg-score-coverage}
\end{equation}
If $\mathcal P$ is replaced by a candidate set $\mC^{(L)}$, the
probability is at least $1-\alpha-\bP\{\Pi_0\notin\mC^{(L)}\}$.
\end{theorem}
\begin{proof}
Let $\bm B$ be an $n\times p$ matrix whose columns form an orthonormal
basis of $\range(\Pi_0\bm X)$, and put $\bm\xi=\bm B^\top\bm\varepsilon$,
so that $\|M_{\Pi_0\bm X}\bm\varepsilon\|_2=\|\bm\xi\|_2$ and
$\bE\exp(\bm u^\top\bm\xi)\le\exp(v^2\|\bm u\|_2^2/2)$ for every
$\bm u\in\R^p$. If $v=0$, then $\bm\xi=\bm 0$ almost surely, and the
bound displayed below is trivial; let $v>0$, and let
$\bm g\sim\Normal(0,I_p)$ be independent of $\bm\xi$. For
$0\le\theta<1/(2v^2)$,
\[
\bE\exp(\theta\|\bm\xi\|_2^2)
=\bE\exp\{(2\theta)^{1/2}\bm g^\top\bm\xi\}
\le\bE\exp(\theta v^2\|\bm g\|_2^2)
=(1-2\theta v^2)^{-p/2}.
\]
The inequality $-\log(1-z)-z\le z^2/\{2(1-z)\}$ for $0\le z<1$ then
gives
$\log\bE\exp\{\theta(\|\bm\xi\|_2^2-pv^2)\}
\le pv^4\theta^2/(1-2v^2\theta)$. For $x=\log(1/\alpha)$, Markov's
inequality with $\theta=\sqrt x/\{v^2(\sqrt p+2\sqrt x)\}$ yields
\[
\bP\left\{\|\bm\xi\|_2^2>v^2(p+2\sqrt{px}+2x)\right\}\le e^{-x}=\alpha .
\]
Since
$M_{\Pi_0\bm X}(\bm Y-\Pi_0\bm X\bm\beta_0)
=M_{\Pi_0\bm X}(\bm b+\bm\varepsilon)$ has norm at most
$\bar b+\|\bm\xi\|_2$, the set of the union for $\Pi=\Pi_0$ contains
$\bm\beta_0$ with probability at least $1-\alpha$, which proves
\eqref{eq:sg-score-coverage}. The last statement follows from
\eqref{eq:union-principle}.
\end{proof}

Independence of the elements of $\bm\varepsilon$ is not required; for
more general versions of the inequality for $\|\bm\xi\|_2^2$, see
\cite{hsu2012tail}.

For the localization, suppose that \eqref{eq:sg-model} holds with
$\bm b=\bm 0$ and $v>0$, so that $\bm Y=\bm s+\bm\varepsilon$, and that
the repro samples are $\Normal(0,I_n)$ vectors independent of
$\bm\varepsilon$. Let $j_\Pi$ be defined as in the proof of
Proposition~\ref{cor:compact-small-list}.

\begin{proposition}
\label{ta-n-sg-sharp}
Suppose $\rank(\bm X)=p<n-1$, $N\ge1$, and $w_\Pi>0$ for every
$\Pi\in\mP_{n,k}\setminus\{\Pi_0\}$. Then, for every $0<t<1$ and
$L\ge1$,
\begin{equation}
\bP\{\Pi_0\notin\mC^{(L)}\}\le E_t+(1-E_t)q_t^L,
\label{ta-n-eq-sg-optimized}
\end{equation}
where
\[
E_t=\min\left\{1,\sum_{\Pi\in\mP_{n,k}\setminus\{\Pi_0\}}
 2^{j_\Pi/2}\exp\left(-\frac{tw_\Pi^2}{8v^2}\right)\right\},
\qquad
q_t=\min\left\{1,NI_t\left(\frac{n-p-1}{2},\frac12\right)\right\}.
\]
The right-hand side of \eqref{ta-n-eq-sg-optimized} is nondecreasing in
$E_t$, and $j_\Pi\le\min\{p,n-p-1\}$; hence the bound also holds with
$E_t$ replaced by
$\min[1,N2^{\min\{p,n-p-1\}/2}\exp\{-tw_{\min}^2/(8v^2)\}]$, where
$w_{\min}=\min_{\Pi\ne\Pi_0}w_\Pi$.
\end{proposition}

The bound is small only when the separation $w_{\min}^2/v^2$ is large
compared with $\log N$; for a fixed design, it does not tend to zero as
$L\to\infty$.

\begin{proof}[Proof of Proposition~\ref{ta-n-sg-sharp}]
Fix $\Pi\in\mP_{n,k}\setminus\{\Pi_0\}$ and a vector $\bm u\in\R^n$ for
which $(\Pi_0\bm X,\bm u)$ and $(\Pi\bm X,\bm u)$ have rank $p+1$. Let
$M_0$ and $M_1$ be the projection matrices onto the ranges of these two
matrices, and $M_2$ the projection matrix onto the intersection of the
two ranges. Then $M_1-M_2$ is a projection matrix of rank
$m=\rank(\Pi_0\bm X,\Pi\bm X,\bm u)-(p+1)\le j_\Pi$, and
\[
F_{\Pi_0}(\bm u)=\|(I_n-M_0)\bm\varepsilon\|_2^2
\le\|(I_n-M_2)\bm\varepsilon\|_2^2 .
\]
With $\bm r=(I_n-M_1)\bm s=R_{(\Pi\bm X,\bm u)}\bm s$, the event
$\{F_\Pi(\bm u)\le F_{\Pi_0}(\bm u)\}$ therefore implies
\[
\|\bm r\|_2^2+2\bm r^\top\bm\varepsilon
 -\bm\varepsilon^\top(M_1-M_2)\bm\varepsilon\le0 .
\]
Let $\bm B$ be an $n\times m$ matrix with orthonormal columns such that
$\bm B\bm B^\top=M_1-M_2$; then $\bm B^\top\bm r=\bm 0$. For
$\bm Z\sim\Normal(0,I_m)$ independent of $\bm\varepsilon$ and
$0<\theta<1/(2v^2)$, Markov's inequality and \eqref{eq:sg-model} give
\begin{align*}
\bP\{F_\Pi(\bm u)\le F_{\Pi_0}(\bm u)\}
&\le e^{-\theta\|\bm r\|_2^2}
 \bE\exp\left\{\theta\|\bm B^\top\bm\varepsilon\|_2^2
 -2\theta\bm r^\top\bm\varepsilon\right\}\\
&=e^{-\theta\|\bm r\|_2^2}
 \bE\exp\left\{\left((2\theta)^{1/2}\bm B\bm Z-2\theta\bm r\right)^\top
 \bm\varepsilon\right\}\\
&\le e^{-\theta\|\bm r\|_2^2}
 \bE\exp\left\{v^2\left\|(2\theta)^{1/2}\bm B\bm Z-2\theta\bm r\right\|_2^2/2\right\}\\
&=(1-2\theta v^2)^{-m/2}
 \exp\left\{-\theta(1-2\theta v^2)\|\bm r\|_2^2\right\}.
\end{align*}
With $\theta=1/(4v^2)$, we obtain
\begin{equation}
\bP\{F_\Pi(\bm u)\le F_{\Pi_0}(\bm u)\}
\le2^{j_\Pi/2}
 \exp\left\{-\frac{\|R_{(\Pi\bm X,\bm u)}\bm s\|_2^2}{8v^2}\right\}.
\label{ta-n-eq-sg-fixed-u}
\end{equation}

For a repro sample $\bm U^*$, by \eqref{eq:sharp-angle} and a union
bound, the event that
$\|R_{(\Pi\bm X,\bm U^*)}\bm s\|_2^2>tw_\Pi^2$ for every
$\Pi\in\mP_{n,k}\setminus\{\Pi_0\}$ has probability at least $1-q_t$.
Let $A$ be the event that it occurs for at least one of the repro
samples $\bm U_1^*,\ldots,\bm U_L^*$, so that $\bP(A^c)\le q_t^L$, and,
on $A$, let $\ell$ be the first index for which it occurs. Since the
repro samples are independent of $\bm\varepsilon$,
\eqref{ta-n-eq-sg-fixed-u} and a union bound show that, on $A$ and
conditionally on the repro samples,
$F_\Pi(\bm U_\ell^*)\le F_{\Pi_0}(\bm U_\ell^*)$ holds for some
$\Pi\in\mP_{n,k}\setminus\{\Pi_0\}$ with probability at most $E_t$. This
is necessary for $\Pi_0\notin\mC^{(L)}$, and therefore
\[
\bP\{\Pi_0\notin\mC^{(L)}\}\le\bP(A^c)+E_t\bP(A)
=E_t+(1-E_t)\bP(A^c)\le E_t+(1-E_t)q_t^L .
\]
\end{proof}

\subsubsection{Approximate conditional laws}
For distributions $P$ and $Q$, let
$d_{\rm TV}(P,Q)=\sup_B|P(B)-Q(B)|$. For a permutation $\Pi$, write
$S_\Pi(\bm y)=(S_{1,\Pi}(\bm y),S_{2,\Pi}(\bm y))$ for the statistics
\eqref{eq:sufficient statistics}, let $P_{\bm\zeta}$ denote the
conditional distribution of $\bm Y$ given $S_\Pi(\bm Y)=\bm\zeta$ under the
true distribution of $\bm Y$, and let $\tilde Q_{\bm\zeta}$ be the
distribution from which the conditional Monte Carlo samples are
generated; it may depend on $\bm\zeta$, but not otherwise on $\bm Y$. The
test statistic $D$ may depend, in addition to the response, on $\bm\zeta$
and on random variables $W$, such as the repro samples, that are
independent of $\bm Y$ conditionally on $S_\Pi(\bm Y)$.

\begin{theorem}
\label{thm:conditional-tv-transfer}
Given $S_\Pi(\bm Y)=\bm\zeta$, let $\bm Y_1^*,\ldots,\bm Y_M^*$ be
independent draws from $\tilde Q_{\bm\zeta}$ that are independent of
$(\bm Y,W)$, and let
\[
p_\Pi=\frac{1+\sum_{m=1}^M\1\{D(\bm Y_m^*)\ge D(\bm Y)\}}{M+1}.
\]
If $d_{\rm TV}(P_{\bm\zeta},\tilde Q_{\bm\zeta})\le\epsilon_\Pi(\bm\zeta)$ for
every $\bm\zeta$, then, for $\alpha\in[0,1]$ and
$\alpha_M=\lfloor\alpha(M+1)\rfloor/(M+1)$,
\begin{equation}
\bP(p_\Pi\le\alpha)\le\alpha_M+\bE\epsilon_\Pi(S_\Pi(\bm Y)).
\label{eq:tv-rank-level}
\end{equation}
Consequently, for every set $\mathcal P_0$ of permutation matrices that
contains $\Pi_0$,
\begin{align*}
\bP\left(\max_{\Pi\in\mathcal P_0}p_\Pi\le\alpha\right)
&\le\alpha_M+\bE\epsilon_{\Pi_0}(S_{\Pi_0}(\bm Y)),\\
\bP\left(\max_{\Pi\in\mC^{(L)}\cap\mathcal P_0}p_\Pi\le\alpha\right)
&\le\alpha_M+\bE\epsilon_{\Pi_0}(S_{\Pi_0}(\bm Y))
 +\bP\{\Pi_0\notin\mC^{(L)}\},
\end{align*}
where the maximum over the empty set is zero.
\end{theorem}

The test of Section~\ref{subsubsection:testing permutation} rejects
when $\max_{\Pi\in\mC^{(L)}\cap\mP_{n,k_0}}p_\Pi\le\alpha$. Under the
Gaussian model and under \eqref{eq:spherical}, the samples of that
section for $\Pi=\Pi_0$ are drawn from $\tilde Q_{\bm\zeta}=P_{\bm\zeta}$ by
\eqref{eq:conditional-representation}, so that $\epsilon_{\Pi_0}=0$
and the type-I error is at most
$\alpha_M+\bP\{\Pi_0\notin\mC^{(L)}\}$. In general, the
total-variation error enters the bound once, and is not multiplied by
the number $M$ of Monte Carlo samples. The theorem requires independent
draws from a distribution that depends on the response only through
$S_\Pi(\bm Y)$.

\begin{proof}[Proof of Theorem~\ref{thm:conditional-tv-transfer}]
Fix $\bm\zeta$ and the value of $W$, and suppose first that $\bm Y$ is
replaced by a draw $\bm Y_0^*$ from $\tilde Q_{\bm\zeta}$ that is
independent of $\bm Y_1^*,\ldots,\bm Y_M^*$. The $M+1$ variables
$D(\bm Y_m^*)$, $0\le m\le M$, are then exchangeable, so that the
probability of $\{p_\Pi\le\alpha\}$ is at most $\alpha_M$. The
conditional distributions of $(\bm Y,\bm Y_1^*,\ldots,\bm Y_M^*)$ and of
$(\bm Y_0^*,\bm Y_1^*,\ldots,\bm Y_M^*)$ are
$P_{\bm\zeta}\otimes\tilde Q_{\bm\zeta}^{\otimes M}$ and
$\tilde Q_{\bm\zeta}^{\otimes(M+1)}$, and their total-variation distance is
$d_{\rm TV}(P_{\bm\zeta},\tilde Q_{\bm\zeta})$. Hence the conditional
probability of $\{p_\Pi\le\alpha\}$ is at most
$\alpha_M+\epsilon_\Pi(\bm\zeta)$, and taking the expectation gives
\eqref{eq:tv-rank-level}. The other two bounds follow from
\eqref{eq:tv-rank-level} for $\Pi=\Pi_0$, because
$\max_{\Pi\in\mathcal P_0}p_\Pi\ge p_{\Pi_0}$, and
$\max_{\Pi\in\mC^{(L)}\cap\mathcal P_0}p_\Pi\ge p_{\Pi_0}$ when
$\Pi_0\in\mC^{(L)}$.
\end{proof}

\subsubsection{Mean misspecification}
Suppose that $\bm Y=\Pi_0\bm X\bm\beta_0+\bm b+\sigma_0\bm U$ with a
deterministic vector $\bm b\in\R^n$ and $\bm U\sim\Normal(0,I_n)$. Under
the assumptions of Theorem~\ref{thm:recovery-threshold}, for every
$L\ge1$ and every $0<\gamma\le1/4$ such that
$\|\bm b\|_2<\underline w\{(1+t_\gamma)^{1/2}-1\}$,
\begin{equation}
\bP\{\Pi_0\notin\mC^{(L)}\}
\le(1-a_\gamma)^L+Nb_\gamma
+\bP\left\{\|\bm U\|_2\ge
 \frac{\underline w\{(1+t_\gamma)^{1/2}-1\}-\|\bm b\|_2}{\sigma_0\gamma}
 \right\}.
\label{eq:mean-misspec}
\end{equation}
This bound concerns the candidate set only: for $\bm b\ne\bm 0$,
neither the $F_{p,n-p}$ distribution of $T$ nor the conditional
representation \eqref{eq:conditional-representation} is preserved.

\begin{proof}[Proof of \eqref{eq:mean-misspec}]
We use the events $A_\ell$ and $B_\ell$ of the proof of
Theorem~\ref{thm:recovery-threshold}, and put
$c=\|\bm b\|_2+\sigma_0\gamma\|\bm U\|_2$. On $A_\ell$, the norm of
$R_{(\Pi\bm X,\bm U_\ell^*)}(\bm b+\sigma_0\bm U)$ is at most $c$ for
every $\Pi$, so that $F_{\Pi_0}(\bm U_\ell^*)\le c^2$ and, on
$A_\ell\cap B_\ell^c$, as in \eqref{eq:app-separation},
\[
F_\Pi(\bm U_\ell^*)-F_{\Pi_0}(\bm U_\ell^*)
>t_\gamma w_\Pi^2-2cw_\Pi-c^2
\qquad\text{for every }\Pi\in\mP_{n,k}\setminus\{\Pi_0\}.
\]
The right-hand side is positive if $c<w_\Pi\{(1+t_\gamma)^{1/2}-1\}$,
which holds when $c<\underline w\{(1+t_\gamma)^{1/2}-1\}$. The complement
of the latter event is the event of the third term of
\eqref{eq:mean-misspec}, and the conclusion follows as in the proof of
Theorem~\ref{thm:recovery-threshold}.
\end{proof}

\section{Additional Figures}\label{Appendix:figures}
In addition to Section \ref{sec:numerical}, we report the Monte Carlo averages of the coverage probability of the candidate set (Figure \ref{fig:app-mc-coverage}) and the volume of the coefficient confidence sets relative to that of the oracle confidence set that uses $\Pi_0$ (Figure \ref{fig:app-coef-volume}).

\begin{figure}[htbp]
\centering
\includegraphics[width=0.95\linewidth]{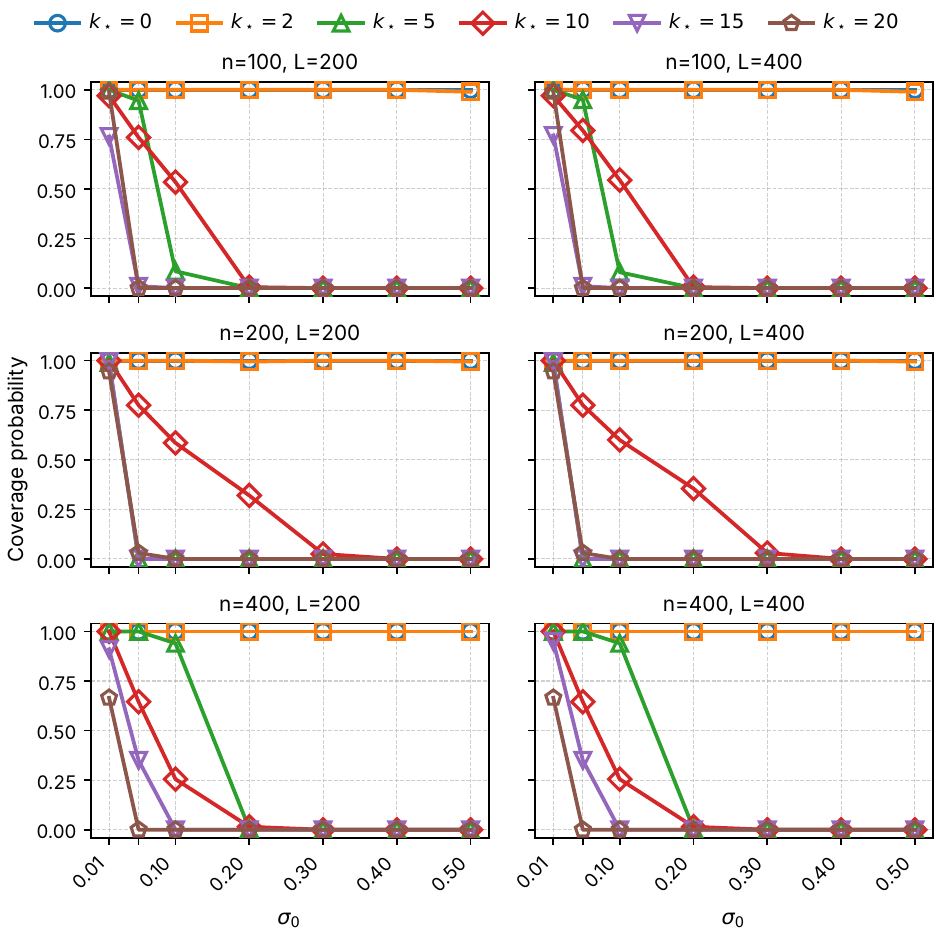}
\caption{Monte Carlo averages of the coverage probability $\bP(\Pi_0\in\mC^{(L)})$ of the candidate set, for $n\in\{100,200,400\}$ (rows) and $L\in\{200,400\}$ (columns).}
\label{fig:app-mc-coverage}
\end{figure}

\begin{figure}[htbp]
\centering
\includegraphics[width=0.95\linewidth]{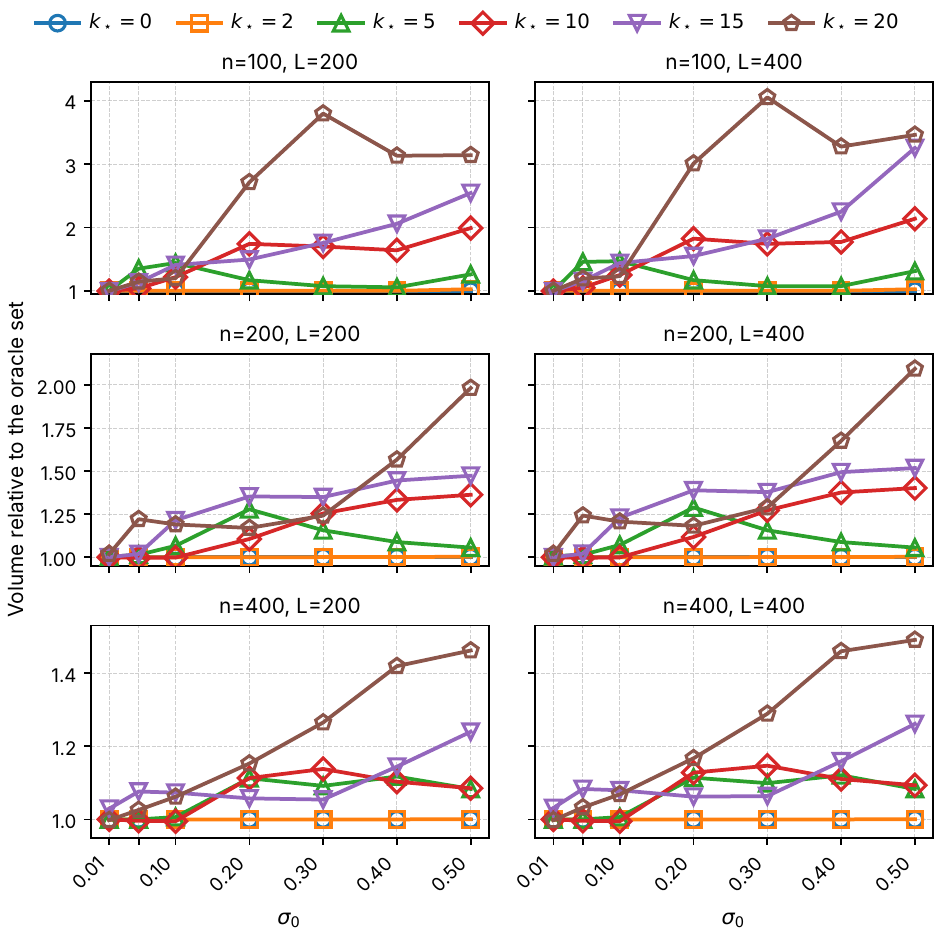}
\caption{Monte Carlo averages of the ratio of the volume of the confidence set for the coefficient vector to that of the oracle confidence set that uses $\Pi_0$, for $n\in\{100,200,400\}$ (rows) and $L\in\{200,400\}$ (columns); the volume of a set is that of the smallest box with sides parallel to the coordinate axes containing it.}
\label{fig:app-coef-volume}
\end{figure}

\clearpage
\bibliographystyle{alpha}
\bibliography{mybib}

\end{document}